\documentclass[11pt]{amsart}
\usepackage{amsmath}
\usepackage{amssymb}
\usepackage{amsfonts}
\usepackage{amsthm}
\usepackage{enumerate}
\usepackage[mathscr]{eucal}
\usepackage{verbatim}
\usepackage{amsthm}
\usepackage{amscd}
\usepackage[mathscr]{eucal}
\usepackage{appendix}
\usepackage{tikz}
\usepackage{soul}

\newtheorem{innercustomgeneric}{\customgenericname}
\providecommand{\customgenericname}{}

\newcommand{\newcustomtheorem}[2]{\newenvironment{#1}[1]
  {\renewcommand\customgenericname{#2}
   \renewcommand\theinnercustomgeneric{##1}\innercustomgeneric}{\endinnercustomgeneric}}

\newcustomtheorem{customthm}{Theorem}

\newcommand{\newcustomlemma}[2]{\newenvironment{#1}[1]
  {\renewcommand\customgenericname{#2}
   \renewcommand\theinnercustomgeneric{##1} \innercustomgeneric}{\endinnercustomgeneric}}

\newcustomlemma{customlemma}{Lemma}

\newcustomlemma{customproposition}{Proposition}

\theoremstyle{plain}
\newtheorem{theorem}{Theorem}

{}
\newtheorem{lemma}{Lemma}\numberwithin{lemma}{section}
\newtheorem{proposition}{Proposition}\numberwithin{proposition}{section}
\numberwithin{corollary}{section}
\theoremstyle{definition}
\numberwithin{definition}{section}
\theoremstyle{remark}
\newtheorem*{remark}{Remark}

\numberwithin{equation}{section}

\newcommand{\e}{\epsilon}

\newcommand{\la}{\lambda}
\newcommand{\ga}{\gamma}

\newcommand{\R}{\mathbb{R}}
\newcommand{\Z}{\mathbb{Z}}

\newcommand{\LL}{\mathcal{L}}

\newcommand{\q}{\quad}
\newcommand{\qq}{\qquad}

\newcommand{\qqqq}{\qquad\qquad\qquad\qquad}

\newcommand{\wt}{\widetilde}
\newcommand{\wh}{\widehat}
\newcommand{\supp}{\mbox{supp}}

\newcommand{\di}{\partial}

\newcommand{\bbn}{\mathbb{N}}
\newcommand{\bbz}{\mathbb{Z}}
\newcommand{\bbzn}{\mathbb Z^n}

\newcommand{\bbr}{\mathbb{R}}
\newcommand{\bbrn}{\mathbb R^n}

\newcommand{\xxxi}{\vec{\xi}}
\newcommand{\ccdot}{\vec{\cdot}\;}

\def\xxxi{\vec{\boldsymbol{\xi}}}
\def\000{\vec{\boldsymbol{0}}}

\def\zzz{\vec{\boldsymbol{z}}}

\def\bbbb{{\boldsymbol{b}}}
\def\mm{{\boldsymbol{m}}}
\def\xx{{\mathbf{x}}}
\def\cc{{\mathbf{c}}}
\def\ii{{\mathrm{i}}}

\newcommand{\II}{\mathcal{I}}
\newcommand{\JJ}{\mathcal{J}}
\newcommand{\UU}{\mathcal{U}}
\newcommand{\VV}{\mathcal{V}}

\newcommand{\DD}{\mathcal{D}}
\newcommand{\KK}{\mathscr{K}}
\newcommand{\RR}{\mathscr{R}}
\newcommand{\BB}{\mathscr{B}}

\newcommand{\Ga}{\Gamma}
\newcommand{\La}{\Lambda}

\title{Trilinear Fourier multipliers on Hardy spaces}

\subjclass[2010]{Primary 42B15, 42B30, 42B25, 47H60}
\keywords{Fourier muitipliers, Multilinear operators, Hardy spaces}

\author{Jin Bong Lee}
\address{J.B. Lee, Research Institute of Mathematics, Seoul National University, Seoul 08826, Republic of Korea}
\email{jinblee@snu.ac.kr}

\author{Bae Jun Park}
\address{B. Park, Department of Mathematics, Sungkyunkwan University, Suwon 16419, Republic of Korea}
\email{bpark43@skku.edu}

\thanks{J.B. Lee is supported by NRF grant 2021R1C1C2008252. B. Park is supported in part by NRF grant 2022R1F1A1063637 and by POSCO Science Fellowship of POSCO TJ Park Foundation}

\begin{document}

\begin{abstract}
In this paper, we obtain the $H^{p_1}\times H^{p_2}\times H^{p_3}\to H^p$ boundedness for trilinear Fourier multiplier operators, which is a trilinear analogue of the multiplier theorem of  Calder\'on and Torchinsky \cite{Ca_To1977}.
Our result improves the trilinear estimate in \cite{LHHLPPY} by additionally assuming an appropriate vanishing moment condition, which is natural in the boundedness into the Hardy space $H^p$ for $0<p\le 1$.

\end{abstract}

\maketitle
\tableofcontents

\section{Introduction}

For a function $\sigma$ on $\bbrn$, 
let $T_\sigma$ be the corresponding Fourier multiplier operator given by
$$
T_\sigma f(x) := \int_{\R^n} \sigma(\xi) \wh{f}(\xi) e^{2\pi i\langle x,\xi\rangle} \;d\xi
$$
for a Schwartz function $f$ on $\bbrn$, where $\wh{f}(\xi):=\int_{\bbrn}f(x)e^{2\pi i\langle x,\xi\rangle}dx$ is the Fourier transform of $f$.
The function $\sigma$ is called an $L^p$ multiplier if $T_\sigma$ is bounded on $L^p(\bbrn)$ for $1<p<\infty$. For several decades, figuring out a sharp condition for $\sigma$ to be an $L^p$ multiplier has been one of the most interesting problems in harmonic analysis. Although there is no complete answer to this question, we have some satisfactory results. In 1956, Mihlin \cite{Mi1956} proved that $\sigma$ is an $L^p$ multiplier provided that
\begin{equation}\label{mih condi}
|\di^\alpha \sigma (\xi)| \lesssim |\xi|^{-|\alpha|},\q \xi \not=0 \qq\text{ for any multi-indices }~ |\alpha|\le [n/2]+1.
\end{equation} 
This result was refined by H\"ormander \cite{Ho1960} who replaced \eqref{mih condi} by the weaker condition
\begin{equation*}
\sup_{k\in\bbz}\big\Vert \sigma(2^k\cdot)\wh{\psi} \big\Vert_{L_s^2(\bbrn)}<\infty \qq \text{ for }~ s>n/2,
\end{equation*}
 where $L_s^2(\bbrn)$ denotes the fractional Sobolev space on $\bbrn$ and $\psi$ is a Schwartz function on $\bbrn$ generating Littlewood-Paley functions, which will be officially defined in Section \ref{hardyspacesection}. We also remark that $s>n/2$ is the best possible regularity condition for the $L^p$ boundedness of $T_{\sigma}$.

Now we define the (real) Hardy space.
Let $\phi$ be a smooth function on $\bbrn$ that is supported in $\{x\in\bbrn: |x|\le 1\}$ and we define $\phi_l:=2^{ln}\phi(2^l\cdot)$. 
Then the Hardy space $H^p(\bbrn)$, $0<p\le\infty$, consists of tempered distributions $f$ on $\bbrn$ such that
\begin{equation}\label{hardydef2}
\Vert f\Vert_{H^p(\bbrn)}:= \Big\Vert \sup_{l\in\bbz} \big| \phi_l\ast f\big| \Big\Vert_{L^p(\bbrn)}
\end{equation} is finite. The space provides an extension to $0<p\le 1$ in the scale of classical $L^p$ spaces for $1<p\le\infty$, which is more natural and useful in many respects than the corresponding $L^p$ extension. Indeed, $L^p(\bbrn)=H^p(\bbrn)$ for $1<p\le\infty$ and several essential operators, such as singular integrals of Calder\'on-Zygmund type, that are well-behaved on $L^p(\bbrn)$ only for $1<p\le\infty$ are also well-behaved on $H^p(\bbrn)$ for $0<p\le 1$.
Now let $\mathscr{S}(\bbrn)$ denote  the Schwartz space on $\bbrn$ and $\mathscr{S}_0(\bbrn)$ be its subspace consisting of $f$ satisfying
$$\int_{\bbrn}x^{\alpha}f(x) \; dx=0 \q \text{ for all multi-indices}~\alpha.$$
Then it turns out that 
\begin{equation}\label{densesubset}
\mathscr{S}_0(\bbrn) ~ \text{ is dense in }~ H^p(\bbrn)~\text{ for all }~0<p<\infty.
\end{equation}
We remark that $\mathscr{S}(\bbrn)$ is also dense in $H^p(\bbrn)=L^p(\bbrn)$ for $1<p<\infty$, but not for $0<p\le 1$. See \cite[Chapter III, \S 5.2]{St1993} for more details.
Moreover, as mentioned in \cite[Chapter III, \S 5.4]{St1993}, if $f\in L^1(\bbrn)\cap H^p(\bbrn)$ for $0<p\le 1$, then 
\begin{equation}\label{hardyvanishing}
\int_{\bbrn}x^{\alpha}f(x) \; dx=0 \q \text{ for all multi-indices }~|\alpha|\le \frac{n}{p}-n.
\end{equation}
We refer to \cite{Bu_Gu_Si1971, Ca1977, Fe_St1972, St1993, Uch1985} for more details.

In 1977, Calder\'on and Torchinsky \cite{Ca_To1977} provided a natural extension of the result of H\"ormander to the Hardy space $H^p(\bbrn)$ for $0<p\le 1$.
For the purpose of investigating $H^p$ estimates for $0<p\le 1$, 
the operator $T_{\sigma}$ is assumed to initially act on $\mathscr{S}_0(\bbrn)$ and then to admit an $H^p$-bounded extension for $0<p< \infty$ via density, in view of \eqref{densesubset}. Then Calder\'on and Torchinsky proved 
\begin{customthm}{A}[\cite{Ca_To1977}] \label{thma}
Let $0<p\le 1$. Suppose that $s>n/p-n/2$. Then we have
\begin{equation*}
\big\Vert T_{\sigma}f\big\Vert_{H^p(\bbrn)}\lesssim \sup_{k\in \bbz}\big\Vert \sigma(2^k\cdot)\wh{\psi}\big\Vert_{L_s^2(\bbrn)}\Vert f\Vert_{H^p(\bbrn)}
\end{equation*}
for all $f\in \mathscr{S}_0(\bbrn)$.
\end{customthm}
For more information about the theory of Fourier multipliers, we also refer the reader to \cite{Ba_Sa1985, Gr_He_Ho_Ng2017, Gr_Park_IMRN, Gr_Sl2019, Park2019, Se1988, Se1989, Se_Tr2019} and the references therein.

\hfill

We now turn our attention to multilinear extensions of the above multiplier results. Let $m$ be a positive integer greater or equal to $2$.
For a bounded function $\sigma$ on $(\bbrn)^m$, let $T_\sigma$ now denote an $m$-linear Fourier multiplier operator given by
\begin{equation*}
T_\sigma\big(f_1, \cdots, f_m\big)(x) := \int_{(\bbrn)^m}  \sigma(\xxxi\,)\Big( \prod_{j=1}^m \wh{f_j}(\xi_j)\Big) \; e^{2\pi i\langle x,\xi_1+\dots+\xi_m\rangle}\, d\xxxi,\quad \xxxi :=(\xi_1, \cdots, \xi_m)
\end{equation*}
for $f_1,\dots,f_m\in \mathscr{S}_0(\bbrn)$.
The first important result concerning multilinear multipliers was obtained by Coifman and Meyer \cite{Co_Me1978} who proved that if $N$ is sufficiently large and
\begin{equation}\label{multilinearmihlin}
\big| \partial_{\xi_1}^{\alpha_m}\cdots \partial_{\xi_m}^{\alpha_m}\sigma(\xi_1,\dots,\xi_m)\big|\lesssim_{\alpha_1,\dots,\alpha_m} \big| (\xi_1,\dots,\xi_m)\big|^{-(|\alpha_1|+\dots+|\alpha_m|)}, \q (\xi_1,\dots,\xi_m)\not= \000
\end{equation} for all $|\alpha_1|+\cdots+ |\alpha_m|\le N$, then $T_{\sigma}$ is bounded from $L^{p_1}(\bbrn)\times \cdots\times L^{p_m}(\bbrn)$ into $L^p(\bbrn)$  for $1<p_1,\dots,p_m<\infty$ and $1\le p<\infty$.
This result is a multilinear analogue of Mihlin's result in which \eqref{mih condi} is required,  but the optimal regularity condition, such as $|\alpha|\le [n/2]+1$ in \eqref{mih condi}, is not considered in the result of Coifman and Meyer.
 Afterwards, Tomita \cite{Tom2010} provided a sharp estimate for multilinear multiplier $T_{\sigma}$, as a multilinear counterpart of H\"ormander's result.
Let $\Psi^{(m)}$ be a Schwartz function on $(\bbrn)^m$ having the properties that
\begin{equation*}
\supp(\widehat{\Psi^{(m)}})\subset \big\{\xxxi:=(\xi_1,\dots,\xi_m)\in (\bbrn)^m: 1/2\leq |\xxxi|\leq 2 \big\}, \q  \sum_{j\in\Z}{\widehat{\Psi^{(m)}}}(2^{-j}\xxxi)=1,~\xxxi\not= \vec{\boldsymbol{0}}.
\end{equation*}
For $s\geq 0$, we define the Sobolev norm
\begin{equation}\label{sobolevnorm}
\Vert F\Vert_{L_s^2((\bbrn)^m)}:=\Big( \int_{(\bbrn)^m}{\big(1+4\pi^2|\xxxi|^2 \big)^s\big|\wh{F}(\xxxi)\big|^2}\;d\xxxi\Big)^{1/2}.
\end{equation}
\begin{customthm}{B} [\cite{Tom2010}]\label{thmb}
Let $1<p,p_1,\dots,p_m<\infty$ with $1/p=1/p_1+\cdots+1/p_m$.
Suppose that 
\begin{equation*}
\sup_{k \in \mathbb{Z}} \big\|\sigma(2^k \vec{\cdot}\;) \widehat{\Psi^{(m)}}\; \big\|_{L^2_s((\bbrn)^m)}<\infty
\end{equation*}
for $s>mn/2$. Then we have
 \begin{equation}\label{tomitaest}
 \big\Vert T_{\sigma}\big(f_1,\dots,f_m \big)\big\Vert_{L^p} \lesssim \sup_{k \in \mathbb{Z}} \big\|\sigma(2^k \vec{\cdot}\;) \widehat{\Psi^{(m)}}\; \big\|_{L^2_s((\bbrn)^m)}\prod_{j=1}^{m}\Vert f_j\Vert_{L^{p_j}(\bbrn)}
 \end{equation} 
 for $f_1,\dots,f_m \in \mathscr{S}_0(\bbrn)$.
\end{customthm}
The standard Sobolev space $L_s^2((\bbrn)^m)$ in \eqref{tomitaest} is replaced by a product-type Sobolev space in many recent papers.
\begin{customthm}{C}[\cite{ Gr_Mi_Ng_Tom2017,Gr_Mi_Tom2013, Gr_Ng2016, Mi_Tom2013}] \label{thmc}
Let $0<p_1,\dots,p_m\leq \infty$ and $0<p<\infty$ with $1/p=1/p_1+\dots+1/p_m$. 
Suppose that
\begin{equation}\label{minimal0}
s_1,\dots,s_m>\frac{n}{2},\qquad \sum_{j\in J}\Big(\frac{s_j}{n}-\frac{1}{p_j} \Big)>-\frac{1}{2}
\end{equation}
for any nonempty subsets $J$ of $ \{1,\dots,m\}$, and 
\begin{equation}\label{sigmacondition1}
\sup_{k\in\bbz}\big\Vert \sigma(2^k\vec{\cdot}\;)\wh{\Psi^{(m)}}\big\Vert_{L_{(s_1,\dots,s_m)}^2((\bbrn)^m)}<\infty.
\end{equation}
Then we have
\begin{equation}\label{productest}
\big\Vert T_{\sigma}\big(f_1,\dots,f_n \big) \big\Vert_{L^p(\bbrn)} \lesssim \sup_{k\in\bbz}\big\Vert \sigma(2^k\vec{\cdot}\;)\wh{\Psi^{(m)}}\big\Vert_{L_{(s_1,\dots,s_m)}^2((\bbrn)^m)} \prod_{j=1}^{m}{\Vert f_j\Vert_{H^{p_j}(\bbrn)}}
\end{equation}
for $f_1,\dots,f_m\in\mathscr{S}_0(\bbrn)$.
\end{customthm}
Here, the space $L_{(s_1,\dots,s_m)}^{2}((\bbrn)^m)$ indicates the product type Sobolev space on $(\bbrn)^m$, in which the norm is defined by replacing the term $(1+4\pi^2 |\xxxi|^2)^s$ in \eqref{sobolevnorm} by $\prod_{j=1}^{m}\big( 1+4\pi^2|\xi_j|^2\big)^{s_j}$.
It is known in \cite{Park_potential} that the condition \eqref{minimal0} is sharp in the sense that if the condition does not hold, then there exists $\sigma$ such that the corresponding operator $T_{\sigma}$ does not satisfy \eqref{productest}. 
 We also refer the reader to \cite{Cr_Ng2021, Fu_Tom2012} for weighted estimates for multilinear Fourier multipliers.

As an extension of Theorem \ref{thma} to the whole range $0<p_1,\dots,p_m\le \infty$, in the recent paper  of the authors, Lee, Heo, Hong, Park, and Yang \cite{LHHLPPY}, we provide a multilinear multiplier theorem with standard Sobolev space conditions.
\begin{customthm}{D}[\cite{LHHLPPY}]\label{thmd}
Let $0<p_1, \cdots, p_m \le \infty$ and $0<p<\infty$ with $1/p=1/p_1+\cdots+1/p_m$. 
	Suppose that 
	\begin{equation} \label{basicregularity}
		s> \frac{mn}{2}\q \text{and}\q  \frac{1}{p}-\frac{1}{2} < \frac{s}{n}+\sum_{j \in J} \Big( \frac{1}{p_j}-\frac{1}{2}\Big)
	\end{equation}
  for any subsets $J$  of $\{1,\dots,m\}$, and
	\begin{equation}\label{sigmacondition2}
	 \sup_{k \in \mathbb{Z}} \big\|\sigma(2^k \;\vec{\cdot}\;) \widehat{\Psi^{(m)}} \big\|_{L^2_s((\bbrn)^m)}<\infty.
	\end{equation}
	Then we have
	\begin{equation} \label{ddbaa}
	\big\| {T}_{\sigma}(f_1,\dots,f_m)\|_{ L^p(\bbrn)} \lesssim  \sup_{k \in \mathbb{Z}} \big\|\sigma(2^k \;\vec{\cdot}\;) \widehat{\Psi^{(m)}} \big\|_{L^2_s((\bbrn)^m)} \, \prod_{j=1}^m \|f_j\|_{H^{p_j}(\bbrn)}
	\end{equation}
for $f_1,\dots,f_m\in\mathscr{S}_0(\bbrn)$.
\end{customthm}
The optimality of the condition \eqref{basicregularity} was  achieved by Grafakos, He, and H\'onzik \cite{Gr_He_Ho2018}  who proved that if \eqref{ddbaa} holds, then we must necessarily have
$s\ge mn/2$ and $1/p-1/2\le s/n+\sum_{j\in J}\big(1/p-1/2\big)$ for all subsets $J$ of $\{1,\dots,m\}$.  \\

We remark that in the bilinear case $m=2$, Theorem \ref{thmd} follows from Theorem \ref{thmc} as \eqref{basicregularity} implies the existence of $s_1$ and $s_2$, with $s_1+s_2=s$, satisfying \eqref{minimal0}. This is well described in the first proof of Theorem \ref{thmd} in \cite{LHHLPPY}. 
However, when $m\ge 3$, this inclusion is not evident even if similar types of regularity conditions are required in both theorems.

\hfill

 Unlike the estimate in Theorem \ref{thma}, the multilinear extensions in Theorems \ref{thmc} and \ref{thmd} consider the Lebesgue space $L^p$ as a target space when $p\le 1$ (recall that $L^p=H^p$ for $1<p<\infty$).

If a function $\sigma$ on $(\bbrn)^m$ satisfies (\ref{sigmacondition1}) for $s_1,\dots,s_m>n/2$ or (\ref{sigmacondition2}) for $s>mn/2$, then 
Theorems \ref{thmc} and \ref{thmd} imply that $T_{\sigma}(f_1,\dots,f_m)\in L^1$  for all $f_1,\dots,f_m\in\mathscr{S}_0(\bbrn)$.
Therefore, in order for $T_{\sigma}(f_1,\dots,f_m)$ to belong to $H^p(\bbrn)$ for $0<p\le 1$, it should be necessary that
\begin{equation}\label{vanishingtsigma}
\int_{\bbrn}x^{\alpha}T_{\sigma}\big(f_1,\dots,f_m \big)\; dx =0 \q\text{ for } ~ |\alpha|\le \frac{n}{p}-n,
\end{equation} 
in view of (\ref{hardyvanishing}).
However, this property is generally not guaranteed, even if all the functions $f_1,\dots,f_m$ satisfy the moment conditions, in the multilinear setting, while, in the linear case,
$$\int_{\bbrn}x^{\alpha}f(x)\; dx=0, ~ |\alpha|\le N \q \text{ implies }\q \int_{\bbrn}x^{\alpha}T_{\sigma}f(x)\; dx=0, ~ |\alpha|\le N$$
for $N\ge 0$.
Recently, by imposing additional cancellation conditions corresponding to (\ref{vanishingtsigma}), Grafakos, Nakamura, Nguyen, and Sawano \cite{Gr_Na_Ng_Sa2019, Gr_Na_Ng_Sa_appear} obtain a mapping property into Hardy spaces for $T_{\sigma}$.
\begin{customthm}{E}[\cite{Gr_Na_Ng_Sa2019, Gr_Na_Ng_Sa_appear}]\label{thme}
Let $0<p_1, \cdots, p_m \le \infty$ and $0<p\le 1$ with $1/p=1/p_1+\cdots+1/p_m$. 
Let $N$ be sufficiently large and $\sigma$ satisfy \eqref{multilinearmihlin} for all multi-indices $|\alpha_1|+\dots+|\alpha_m|\le N$.
Suppose that 
$$\int_{\bbrn}x^{\alpha}T_{\sigma}\big(a_1,\dots,a_m\big)(x)\; dx=0 $$
for all multi-indices $|\alpha|\le \frac{n}{p}-n$, where $a_j$'s are $(p_j,\infty)$-atoms.
	Then we have
	\begin{equation} \label{ddb}
	\big\| {T}_{\sigma}(f_1,\dots,f_m)\big\|_{ H^p(\bbrn)} \lesssim_{\sigma,N} \prod_{j=1}^m \|f_j\|_{H^{p_j}(\bbrn)}
	\end{equation}
 for  $f_1,\dots,f_m\in\mathscr{S}_0(\bbrn)$.
\end{customthm}
Here, the $(p,\infty)$-atom is similar, but more generalized concept than $H^{p}$-atoms defined in Section \ref{preliminary}, and we adopt the convention that $(\infty,\infty)$-atom $a$ simply means $a\in L^{\infty}(\bbrn)$ with no cancellation condition.
See \cite{Gr_Na_Ng_Sa2019, Gr_Na_Ng_Sa_appear} for the definition and properties of the $(p,\infty)$-atom.

We remark that Theorem \ref{thme} successfully shows the boundedness into $H^p(\bbrn)$, but the optimal regularity conditions considered in Theorems \ref{thmc} and \ref{thmd} are not pursued at all as it requires sufficiently large $N$.

The aim of this paper is to establish the boundedness  into $H^p$ for trilinear multiplier operators, analogous to (\ref{ddb}), with the same regularity conditions as in Theorem \ref{thmd}, which is significantly more difficult in general.
Unfortunately, we do not obtain the desired results for general $m$-linear operators for $m\ge 4$ and we will discuss some obstacles for this generalization in the appendix.

To state our main result, 
let us write $\Psi:=\Psi^{(3)}$ and in what follows, we will use the notation
$$\LL_s^2[\sigma]:=\sup_{k\in\bbz}\big\Vert \sigma( 2^k\ccdot) \wh{\Psi}\big\Vert_{L^2_s((\bbrn)^3)}$$
for a function $\sigma$ on $(\bbrn)^3$.
Let $0<p\le 1$ and we will consider trilinear multipliers $\sigma$ satisfying
\begin{equation}\label{vanishingmoment}
\int_{\bbrn} x^{\alpha}T_{\sigma}\big(f_1,f_2,f_3 \big)(x) \; dx =0 \q \text{ for all multi-indices }~ |\alpha|\le \frac{n}{p}-n 
\end{equation}
for all $f_1,f_2,f_3\in \mathscr{S}_0(\bbrn)$.
Then the main result is as follows: 
\begin{theorem}\label{main}
Let $0<p_1,p_2,p_3<\infty$ and $0<p\le 1$ with $1/p=1/p_1+1/p_2+1/p_3$.
Suppose that
\begin{equation}\label{regularitycondition}
s>\frac{3n}{2} \q \text{ and }\q \frac{1}{p}-\frac{1}{2}<\frac{s}{n}+\sum_{j\in J}\Big(\frac{1}{p_j}-\frac{1}{2} \Big)
\end{equation}
where $J$ is an arbitrary subset of $\{1,2,3\}$.
Let $\sigma$ be a function on $(\bbrn)^3$ satisfying $\LL_s^2[\sigma]<\infty$ and the vanishing moment condition \eqref{vanishingmoment}. 
Then we have
\begin{equation}\label{mainthmest}
\big\Vert T_{\sigma}(f_1,f_2,f_3)\big\Vert_{H^{p}(\bbrn)}\lesssim \LL_s^2[\sigma]\Vert f_1\Vert_{H^{p_1}(\bbrn)}\Vert f_2\Vert_{H^{p_2}(\bbrn)}\Vert f_3\Vert_{H^{p_3}(\bbrn)}
\end{equation} 
 for $f_1,f_2,f_3\in \mathscr{S}_0(\bbrn)$.
\end{theorem}

 \begin{figure}[h]
\begin{tikzpicture}
\path[fill=green!5] (3,0,2)--(3,1.5,2)--(5,1.5,2)--(5,0,2)--(3,0,2);
\path[fill=green!5] (3,1.5,2)--(3,1.5,0)--(5,1.5,0)--(5,1.5,2)--(3,1.5,2);
\path[fill=green!5] (5,0,0)--(5,1.5,0)--(5,1.5,2)--(5,0,2)--(5,0,0);

\path[fill=blue!5] (0,0,5.6)--(1.7,0,5.6)--(1.7,1.5,5.6)--(0,1.5,5.6)--(0,0,5.6);
\path[fill=blue!5] (1.7,0,5.6)--(1.7,0,4)--(1.7,1.5,4)--(1.7,1.5,5.6)--(1.7,0,5.6);
\path[fill=blue!5] (1.7,1.5,5.6)--(1.7,1.5,4)--(0,1.5,4)--(0,1.5,5.6)--(1.7,1,5.6);

\path[fill=red!5] (1.7,4.5,0)--(1.7,4.5,2)--(0,4.5,2)--(0,4.5,0)--(1.7,4.5,0);
\path[fill=red!5] (1.7,3,0)--(1.7,4.5,0)--(1.7,4.5,2)--(1.7,3,2)--(1.7,3,0);
\path[fill=red!5] (0,3,2)--(1.7,3,2)--(1.7,4.5,2)--(0,4.5,2)--(0,3,2);

\path[fill=yellow!5] (1.7,1.5,4)--(1.7,3,2)--(3,1.5,2)--(1.7,1.5,4);
\path[fill=yellow!5] (0,3,2)--(0,1.5,4)--(1.7,1.5,4)--(1.7,3,2)--(0,3,2);
\path[fill=yellow!5] (1.7,3,2)--(1.7,3,0)--(3,1.5,0)--(3,1.5,2)--(1.7,3,2);
\path[fill=yellow!5] (1.7,1.5,4)--(1.7,0,4)--(3,0,2)--(3,1.5,2)--(1.7,1.5,4);

\node [above] at (4,0.4,0.9) { $\RR_1$};
\node [above] at (1,3.8,1.5) { $\RR_2$};
\node [above] at (1.2,0.8,6) { $\RR_3$};
\node [above] at (1.4,1,1.6) { $\RR_0$};

\draw [->] (0,0,4)--(0,0,6);
\draw [->] (0,3,0)--(0,5,0);
\draw [->] (3,0,0)--(5,0,0);
\draw[-] (1.7,1.5,4)--(1.7,3,2)--(3,1.5,2)--(1.7,1.5,4);
\draw[-] (1.7,1.5,4)--(1.7,0,4)--(3,0,2)--(3,1.5,2);
\draw[-] (1.7,3,2)--(1.7,3,0)--(3,1.5,0)--(3,1.5,2);
\draw[-] (1.7,1.5,4)--(0,1.5,4)--(0,3,2)--(1.7,3,2);
\draw[dash pattern= { on 1pt off 1pt}] (0,0,0)--(0,0,4);
\draw[dash pattern= { on 1pt off 1pt}] (0,0,0)--(0,3,0);
\draw[dash pattern= { on 1pt off 1pt}] (0,0,0)--(3,0,0);
\draw[dash pattern= { on 1pt off 1pt}] (3,0,2)--(3,0,0)--(3,1.5,0);
\draw[dash pattern= { on 1pt off 1pt}] (0,3,2)--(0,3,0)--(1.7,3,0);
\draw[dash pattern= { on 1pt off 1pt}] (1.7,0,4)--(0,0,4)--(0,1.5,4);
\draw[-] (1.7,1.5,4)--(1.7,1.5,6);
\draw[-] (1.7,3,2)--(1.7,4.5,2);
\draw[-] (3,1.5,2)--(5,1.5,2);
\draw[-] (1.7,0,4)--(1.7,0,6);
\draw[-] (0,1.5,4)--(0,1.5,6);
\draw[-] (1.7,3,0)--(1.7,4.5,0);
\draw[-] (0,3,2)--(0,4.5,2);
\draw[-] (3,0,2)--(5,0,2);
\draw[-] (3,1.5,0)--(5,1.5,0);

\node [right] at (5,0,0) {\tiny$t_1$};
\node [left] at (0,5,0) {\tiny$t_2$};
\node [left] at (0,0,6) {\tiny$t_3$};

\draw[dash pattern= { on 2pt off 1pt}] (3,0,0)--(0,3,0)--(0,0,4)--(3,0,0);

\node [below] at (3.3,0,2) {\tiny$(1,0,\frac{1}{2})$};
\node [below] at (3.3,0,0) {\tiny$(1,0,0)$};
\node [above] at (3.4,1.4,0) {\tiny$(1,\frac{1}{2},0)$};
\node [right] at (1.7,3,0) {\tiny$(\frac{1}{2},1,0)$};
\node [right] at (1.7,3,2) {\tiny$(\frac{1}{2},1,\frac{1}{2})$};
\node [left] at (0.1,3,2) {\tiny$(0,1,\frac{1}{2})$};
\node [left] at (0,1.5,3.7) {\tiny$(0,\frac{1}{2},1)$};
\node [below] at (2,0,3.9) {\tiny$(\frac{1}{2},0,1)$};
\node [below] at (2.2,1.5,3.9) {\tiny$(\frac{1}{2},\frac{1}{2},1)$};


\path[fill=orange!5] (4.7,5.3,2)--(4.7,3.8,2)--(6,2.3,2)--(8,2.3,2)--(8,2.3,0)--(4.7,5.3,0)--(4.7,5.3,2);
\draw[orange!40, -] (4.7,5.3,2)--(8,2.3,2)--(8,2.3,0)--(4.7,5.3,0)--(4.7,5.3,2);

\node at (6,3.3,1) { $\RR_4$};

\draw[-] (4.7,5.3,2)--(4.7,3.8,2)--(6,2.3,2)--(8,2.3,2);
\draw[dash pattern= { on 1pt off 1pt}] (4.7,3.8,0)--(4.7,5.3,0);
\draw[dash pattern= { on 1pt off 1pt}] (6,2.3,0)--(8,2.3,0);
\draw[dash pattern= { on 1pt off 1pt}] (4.7,3.8,0)--(6,2.3,0);
\draw[dash pattern= { on 1pt off 1pt}] (4.7,3.8,2)--(4.7,3.8,0);
\draw[dash pattern= { on 1pt off 1pt}] (6,2.3,0)--(6,2.3,2);


\path[fill=purple!5] (-1.3,1.7,6)-- (-1.3,1.7,4)--(-1.3,3.2,2)--(-1.3,4.7,2)--(-3,4.7,2)--(-3,1.7,6)--(-1.3,1.7,6);
\draw[purple!40, -] (-1.3,1.7,6)--(-1.3,4.7,2)--(-3,4.7,2)--(-3,1.7,6)--(-1.3,1.7,6);

\node at (-1.7,3.3,4.6) { $\RR_5$};

\draw[-] (-1.3,1.7,6)-- (-1.3,1.7,4)--(-1.3,3.2,2)--(-1.3,4.7,2);

\draw[dash pattern= { on 1pt off 1pt}] (-1.3,1.7,4)--(-3,1.7,4);
\draw[dash pattern= { on 1pt off 1pt}]  (-3,3.2,2)--(-1.3,3.2,2);
\draw[dash pattern= { on 1pt off 1pt}] (-3,1.7,4)--(-3,3.2,2);
\draw[dash pattern= { on 1pt off 1pt}] (-3,1.7,4)--(-3,1.7,6);
\draw[dash pattern= { on 1pt off 1pt}] (-3,3.2,2)--(-3,4.7,2);


\path[fill=cyan!5] (3.5,0.5,9)--(3.5,0.5,7)--(4.8,0.5,5)--(6.8,0.5,5)--(6.8,-1,5)--(3.5,-1,9)--(3.5,0.5,9);
\draw[cyan!40, -] (3.5,0.5,9)--(6.8,0.5,5)--(6.8,-1,5)--(3.5,-1,9)--(3.5,0.5,9);

\node  at (4.3,-0.5,5.7) { $\RR_6$};

\draw[-] (3.5,0.5,9)--(3.5,0.5,7)--(4.8,0.5,5)--(6.8,0.5,5);
\draw[dash pattern= { on 1pt off 1pt}] (3.5,-1,7)--(3.5,-1,9);
\draw[dash pattern= { on 1pt off 1pt}] (4.8,-1,5)--(6.8,-1,5);
\draw[dash pattern= { on 1pt off 1pt}] (3.5,0.5,7)--(3.5,-1,7)--(4.8,-1,5)--(4.8,0.5,5);

\path[fill=gray!5] (11,-2,2)--(7.7,1,2)--(7.7,-2,6)--(11,-2,2);

\draw[dash pattern= { on 1pt off 1pt}] (7.7,-2,4)--(7.7,-0.5,2)--(9,-2,2)--(7.7,-2,4);
\draw[dash pattern= { on 1pt off 1pt}] (7.7,-2,4)--(7.7,-2,6);
\draw[dash pattern= { on 1pt off 1pt}] (7.7,-0.5,2)--(7.7,1,2);
\draw[dash pattern= { on 1pt off 1pt}] (9,-2,2)--(11,-2,2);

\draw[gray!40, -] (11,-2,2)--(7.7,1,2)--(7.7,-2,6)--(11,-2,2);

\node at (9,-1,4) { $\RR_7$};

\node[right] at (7.7,-0.5,2) { \tiny$(\frac{1}{2},1,\frac{1}{2})$};
\node[right] at (7.6,-2.1,4) { \tiny$(\frac{1}{2},\frac{1}{2},1)$};
\node[above right] at (8.8,-2,2) { \tiny$(1,\frac{1}{2},\frac{1}{2})$};

\end{tikzpicture}
\caption{The regions $\RR_{\ii}$, $0\le \ii\le 7$}\label{figregion}
\end{figure}

We remark that $(1/p_1,1/p_2,1/p_3)$ in Theorem \ref{main} is contained in one of the following sets:
\begin{align*}
\RR_0&:=        \big\{ ( t_1,t_2,t_3) :  0< t_1,t_2,t_3 < 1,~  0<t_{1}+t_2, t_2+t_3,t_3+t_1 <3/2, \\
&\qq\qq\qq\qq\qq\qq\qq\qq\qq\qq\qq 1\le t_1+t_2+t_3< 2 \big\},\\
\RR_{\ii} &:= \big\{ ( t_1,t_2,t_3) :   0 < t_j < {1}/{2},~1\le t_{\ii} <\infty  , \,\, j\not= \ii \big\}, \q \ii=1,2,3,\\
\RR_4 &:=\big\{ ( t_1,t_2,t_3) :  0< t_3 < {1}/{2},~ 1/2 \le t_{1}, t_2 <\infty, ~ 3/2\le t_1+t_2 \big\},  \\
\RR_5&:=  \big\{ ( t_1,t_2,t_3) :  0< t_1 <{1}/{2},~ 1/2 \le t_{2}, t_3 <\infty, ~ 3/2\le t_2+t_3 \big\},\\
\RR_6&:=    \big\{ ( t_1,t_2,t_3) :  0< t_2 < {1}/{2},~ 1/2 \le t_{1}, t_3 <\infty, ~ 3/2\le t_1+t_3 \big\},\\
\RR_7&:=  \big\{ ( t_1,t_2,t_3) :  1/2\le t_1,t_2,t_3<\infty,~ 2\le t_1+t_2+t_3<\infty \big\}.
\end{align*}
See Figure \ref{figregion} for the regions $\RR_{\ii}$.
Then the condition (\ref{regularitycondition}) becomes
\begin{equation}\label{conditionsinterpretation}
s>\begin{cases}
3n/2, & (1/p_1,1/p_2,1/p_3)\in \RR_{0},\\
n/p_{\ii}+n/2, & (1/p_1,1/p_2,1/p_3)\in \RR_{\ii}, ~\ii=1,2,3\\
n/p_1+n/p_2, & (1/p_1,1/p_2,1/p_3)\in \RR_{4},\\
n/p_2+n/p_3, & (1/p_1,1/p_2,1/p_3)\in \RR_{5},\\
n/p_3+n/p_1, & (1/p_1,1/p_2,1/p_3)\in \RR_{6},\\
n/p_1+n/p_2+n/p_3-n/2, & (1/p_1,1/p_2,1/p_3)\in \RR_7.
\end{cases}
\end{equation}

In the proof of Theorem \ref{main}, we will mainly focus on the case $(1/p_1,1/p_2,1/p_3)\in \RR_{\ii}$, $\ii=1,2,3$, in which $s>n/p_{\ii}+n/2$ is required.
Then the remaining cases follow from interpolation methods.  More precisely, via interpolation, 
$$\text{the estimates (\ref{mainthmest}) in $\RR_1$ and $\RR_2$ ~$\Rightarrow$~ the estimate (\ref{mainthmest}) in $\RR_4$},$$
$$\text{the estimates (\ref{mainthmest}) in $\RR_2$ and $\RR_3$ ~$\Rightarrow$~ the estimate (\ref{mainthmest}) in $\RR_5$},$$
$$\text{the estimates (\ref{mainthmest}) in $\RR_3$ and $\RR_1$ ~$\Rightarrow$~ the estimate (\ref{mainthmest}) in $\RR_6$},$$
$$\text{the estimates (\ref{mainthmest}) in $\RR_1$, $\RR_2$, and $\RR_3$ ~$\Rightarrow$~ the estimate (\ref{mainthmest}) in $\RR_0$},$$
$$\text{the estimates (\ref{mainthmest}) in $\RR_1$, $\RR_2$, and $\RR_3$ ~$\Rightarrow$~ the estimate (\ref{mainthmest}) in $\RR_7$}$$
 where the case $1/p_1+1/p_2+1/p_3=1$ for $(1/p_1,1/p_2,1/p_3)\in\RR_0$ will be treated separately. 
Here, a complex interpolation method will be applied, but the regularity condition on $s$ will be fixed. Moreover, the index $p$ will be also fixed so that  the vanishing moment condition (\ref{vanishingmoment}) will not be damaged in the process of the interpolation. 
For example, when $(1/p_1,1/p_2,1/p_3)\in \RR_4$, we set $s>n/p_1+n/p_2$ and fix the index $p$ with $1/p=1/p_1+1/p_2+1/p_3$. 
We also fix $\sigma$ satisfying the vanishing moment condition (\ref{vanishingmoment}).
Now we choose $(1/p_1^0,1/p_2^0,1/p_3)\in R_1$ and $(1/p_1^1,1/p_2^1,1/p_3)\in \RR_2$ so that
$$s>n/p_1^0+n/2,\q s>n/p_2^1+n/2, $$
$$ 1/p=1/p_1^0+1/p_2^0+1/p_3=1/p_1^1+1/p_2^1+1/p_3.$$
Then the two estimates
 $$\Vert T_{\sigma}\Vert_{H^{p_1^0}\times H^{p_2^0}\times H^{p_3}\to H^p},\Vert T_{\sigma}\Vert_{H^{p_1^1}\times H^{p_2^1}\times H^{p_3}\to H^p}\lesssim \LL_s^2[\sigma]$$
imply
$$\Vert T_{\sigma}\Vert_{H^{p_1}\times H^{p_2}\times H^{p_3}\to H^p}\lesssim \LL_s^2[\sigma].$$
The detailed arguments concerning the interpolation (for all the cases) will be provided in Section \ref{proofoftheorem1}.

The estimates for  $(1/p_1,1/p_2,1/p_3)\in \RR_{\ii}$, $\ii=1,2,3$, will be restated in Proposition \ref{mainproposition} below, and they will be proved throughout  three sections (Sections \ref{pfproposition}-\ref{estimatejj}).
Since one of $p_j$'s is less or equal to $1$, we benefit from the atomic decomposition for the Hardy space. Moreover, for other indices greater than 2, we employ the techniques of (variant) $\varphi$-transform, introduced by Frazier and Jawerth \cite{Fr_Ja1985, Fr_Ja1988, Fr_Ja1990} and Park \cite{Park_constr}, which will be presented in Section \ref{preliminary}. Then $T_{\sigma}(f_1,f_2,f_3)$ can be decomposed in the form
$$T_{\sigma}(f_1,f_2,f_3)=\sum_{\kappa\in \mathrm{K}}T^{\kappa}(f_1,f_2,f_3)$$
where $\mathrm{K}$ is a finite set, and then we will actually prove that each $T^{\kappa}(f_1,f_2,f_3)$ satisfies the estimate 
\begin{equation}\label{suplinstance}
\sup_{l\in\bbz}\big| \phi_l\ast \big(T^{\kappa}(f_1,f_2,f_3)\big)(x)\big| \lesssim \LL_s^2[\sigma]u_1(x)u_2(x)u_3(x)
\end{equation}
where $\Vert u_{\ii}\Vert_{L^{p_{\ii}}(\bbrn)}\lesssim \Vert f_{\ii}\Vert_{H^{p_{\ii}}(\bbrn)}$ for $\ii=1,2,3$.
Since the above estimate separates the left-hand side into three functions of $x$, we may apply H\"older's inequality with exponents $1/p=1/p_1+1/p_2+1/p_3$ to obtain, in view of (\ref{hardydef2}),
\begin{align*}
\big\Vert T^{\kappa}(f_1,f_2,f_3)\big\Vert_{H^p(\bbrn)}&=\Big\Vert \sup_{l\in\bbz}\big|\phi_l\ast \big( T^{\kappa}(f_1,f_2,f_3)\big) \big|\Big\Vert_{L^p(\bbrn)}\\
&\lesssim \LL_s^2[\sigma]\Vert f_1\Vert_{H^{p_1}(\bbrn)}\Vert f_2\Vert_{H^{p_2}(\bbrn)}\Vert f_3\Vert_{H^{p_3}(\bbrn)}.
\end{align*}
Such pointwise estimates (\ref{suplinstance}) will be described in several lemmas in Sections \ref{estimateii} and \ref{estimatejj}, and the proofs will be given in Section \ref{proofoflemmas} separately, which is one of the keys in this paper.

\subsection*{Notation}
For a cube $Q$ in $\bbrn$ let  $\xx_Q$ be the lower left corner of $Q$ and $\ell(Q)$ be the side-length of $Q$.
We denote by  $Q^*$, $Q^{**}$, and $Q^{***}$ the concentric dilates of $Q$ with $\ell(Q^*)=10\sqrt{n}\ell(Q)$, $\ell(Q^{**})=\big(10\sqrt{n} \big)^2\ell(Q)$, and $\ell(Q^{***})=\big(10\sqrt{n} \big)^3\ell(Q)$. 
Let $\DD$ stand for the family of all dyadic cubes in $\bbrn$ and $\DD_j$ be the subset of $\DD$ consisting of dyadic cubes of side-length $2^{-j}$.  
 For each $\xx\in\bbrn$ and $l\in\bbz$ let $B_{\xx}^l:=B(\xx,100n2^{-l})$ be the ball of radius $100n2^{-l}$ and center $\xx$.
We use the notation $\langle \cdot \rangle$ to denote both the inner product of functions and $\langle y\rangle:= (1+4\pi^2|y|^2)^{1/2}$ for $y\in \bbr^M$, $M\in\bbn$.
That is,  $\langle f, g \rangle=\int_{\bbrn} f(x) \overline{g(x)}\,dx$ for two functions $f$ and $g$, and $\langle x_1\rangle:=(1+4\pi^2|x_1|^2)^{1/2}$, $\langle (x_1,x_2)\rangle:=\big(1+4\pi^2(|x_1|^2+|x_2|^2)\big)^{1/2}$ for $x_1,x_2\in \bbrn$.

\section{Preliminaries}\label{preliminary}

\subsection{Hardy spaces}\label{hardyspacesection}

Let $\theta$ be a Schwartz function on $\bbrn$ such that $\supp(\wh{\theta})\subset \{\xi\in\bbrn: |\xi|\le 2\}$ and $\wh{\theta}(\xi)=1$ for $|\xi|\le 1$. Let $\psi:=\theta-2^{-n}\theta(2^{-1}\cdot)$, and for each $j\in\bbz$ we define $\theta_j:=2^{jn}\theta(2^j\cdot)$ and $\psi_j:=2^{jn}\psi(2^j\cdot)$.
Then $\{\psi_j\}_{j\in\bbz}$ forms a Littlewood-Paley partition of unity, satisfying
$$\supp(\wh{\psi_j})\subset \big\{\xi\in\bbrn: 2^{j-1}\le |\xi|\le 2^{j+1}\big\} \quad\text{ and }\quad \sum_{j\in\bbz}\wh{\psi_j}(\xi)=1, ~\xi\not= 0.$$
We define the convolution operators $\Ga_j$ and $\La_j$ by 
$$\Ga_jf:=\theta_j\ast f, \qq \La_jf:=\psi_j\ast f.$$

The Hardy space $H^p(\bbrn)$ can be characterized with the (quasi-)norm equivalences 
\begin{equation}\label{hardydef}
\Vert f\Vert_{H^p(\bbrn)}\sim\big\Vert \big\{ \Ga_jf \big\}_{j\in\bbz}\big\Vert_{L^p(\ell^{\infty})}, \qq  0<p\le \infty
\end{equation}
and
\begin{equation}\label{hardydeflittle}
\Vert f\Vert_{H^p(\bbrn)}\sim \big\Vert \big\{\La_jf\big\}_{j\in\bbz}\big\Vert_{L^p(\ell^{2})}, \qq 0<p<\infty,
\end{equation}
which is the Littlewood-Paley theory for Hardy spaces.
In addition, when $p\le 1$, every $f\in H^p(\bbrn)$ can be decomposed as
\begin{equation}\label{atomdecomp}
f=\sum_{k=1}^{\infty}\lambda_k a_k \q \text{ in the sense of tempered distributions}
\end{equation}
where $a_k$'s are $H^p$-atoms having the properties that 
$\supp(a_k)\subset Q_k$, $\Vert a_k\Vert_{L^{\infty}(\bbrn)}\le |Q_k|^{-1/p}$ for some cube $Q_k$,
$\int x^{\gamma}a_k(x)dx=0$ for all multi-indices $|\gamma|\le M$,
and $\big( \sum_{k=1}^{\infty}|\lambda_k|^p\big)^{1/p}\lesssim \Vert f\Vert_{H^p(\bbrn)}$
where $M$ is a fixed integer satisfying $M\ge [n/p-n]_+$, which may be actually  arbitrarily large.
Furthermore, each  $H^p$-atom $a_k$ satisfies
\begin{equation*}
\Vert a_k\Vert_{H^1(\bbrn)}\lesssim |Q_k|^{-1/p+1}.
\end{equation*}

\subsection{Maximal inequalities}
Let $\mathcal{M}$ denote the Hardy-Littlewood maximal operator, defined by
$$\mathcal{M}f(x):=\sup_{Q:x\in Q}\frac{1}{|Q|}\int_Q{|f(y)|}dy$$
for a locally integrable function $f$ on $\bbrn$, where the supremum ranges over all cubes $Q$ containing $x$. 
For given $0<r<\infty$, we define $\mathcal{M}_rf:=\big( \mathcal{M}\big(|f|^r\big)\big)^{1/r}$. Then it is well-known that
\begin{equation}\label{hlmax}
\big\Vert \big\{ \mathcal{M}_rf_k \big\}_{k\in\bbz}\big\Vert_{L^p(\ell^q)}\lesssim \Vert \{f_k\}_{k\in\bbz}\Vert_{L^p(\ell^q)}
\end{equation} whenever $r<p< \infty$ and $r<q\le \infty$.
We note that for $1\le r<\infty$
\begin{equation}\label{maximalbound}
\bigg\Vert \frac{f(x-\cdot)}{\langle 2^j\cdot\rangle^t}\bigg\Vert_{L^r(\bbrn)}\lesssim 2^{-{jn}/{r}}\mathcal{M}_rf(x)\qq\text{ if }~t>n/r.
\end{equation}

For $\mm \in \bbzn$ and any dyadic cubes $Q\in\DD$, we use the notation
		$$Q(\mm):=Q+ \ell(Q)\,\mm.$$
Then we define the dyadic shifted maximal operator $\mathcal{M}_{dyad}^{ \mm}$  by
\[\mathcal{M}_{dyad}^{ \mm}f(x):=\sup_{Q\in\DD: x\in Q} \frac{1}{|Q|} \int_{{Q(\mm})} |f(y)|\, dy\]
		where the supremum is taken over all dyadic cubes $Q$ containing $x$. 
It is clear that $\mathcal{M}_{dyad}^{\bold{0}}f(x)\le \mathcal{M}f(x)$ and accordingly, $\mathcal{M}_{dyad}^{\bold{0}}$ is bounded on $L^p$ for $p>1$.
In general, the following maximal inequality holds:
	For $1<p<\infty$ and $\mm\in\bbzn$ we have
	\begin{equation}\label{okl}
	\big\|\mathcal{M}_{dyad}^{\mm}f\big\|_{L^p(\bbrn)} \lesssim \big(\log{(10+|\mm|)}\big)^{n/p}\, \|f\|_{L^p(\bbrn)}.
	\end{equation}
The inequality (\ref{okl}) follows from the repeated use of the inequality in one dimensional setting that appears in \cite[Chapter II, \S 5.10]{St1993}, and we omit the detailed proof here. Refer to \cite[Appendix]{LHHLPPY} for the argument.

\subsection{Variants of $\varphi$-transform}

 For a sequence of complex numbers $\bbbb:=\{b_Q\}_{Q\in\mathcal{D}}$ we define 
 $$\Vert \bbbb\Vert_{\dot{f}^{p,q}}:=\big\Vert g^q(\bbbb)\big\Vert_{L^p(\bbrn)}$$
 for $0<p<\infty$, where 
 $$g^q(\bbbb)(x):=\big\Vert \big\{ |b_Q||Q|^{-{1}/{2}}\chi_Q(x)\big\}_{Q\in\mathcal{D}}\big\Vert_{\ell^q}, \qq  0<q\le\infty.$$
 Let $\wt{\psi_j}:=\psi_{j-1}+\psi_j+\psi_{j+1}$ for $j\in\bbz$. Observe that $\wt{\psi_j}$ enjoys the properties that $\supp(\wh{\wt{\psi}})\subset \{\xi\in\bbrn: 2^{j-2}\le|\xi|\le 2^{j+2}\}$ and $\psi_j=\psi_j\ast \wt{\psi_j}$.
Then we have the representation
\begin{equation}\label{character0}
\La_jf(x)=\sum_{Q\in\mathcal{D}_j}b_Q\psi^{Q}(x)
\end{equation}
 where $\psi^Q(x):=|Q|^{{1}/{2}}\psi_j(x-\xx_Q)$, $\wt{\psi}^Q(x):=|Q|^{{1}/{2}}\wt{\psi_j}(x-\xx_Q)$ for each $Q\in\mathcal{D}_j$, and $b_Q=\langle f,\wt{\psi}^Q\rangle$.
 This implies that
 $$f=\sum_{j\in\bbz}\La_jf=\sum_{j\in\bbz}\sum_{Q\in\mathcal{D}_j}b_Q\psi^Q \q \text{ in }~\mathcal{S}'/\mathcal{P}$$
 where $\mathcal{S}'/\mathcal{P}$ stands for a tempered distribution modulo polynomials.
 Moreover, in this case, we have 
 \begin{equation}\label{charactera}
 \Vert \bbbb\Vert_{\dot{f}^{p,q}}\sim \big\Vert \{\La_jf\}_{j\in\bbz}\big\Vert_{L^p(\ell^q)}.
 \end{equation}
 Therefore, the Hardy space $H^p(\bbrn)$ can be characterized by the discrete function space $\dot{f}^{p,2}$, in view of the equivalence in (\ref{hardydeflittle}). 
 We refer to \cite{Fr_Ja1985, Fr_Ja1988, Fr_Ja1990} for more details.

It is also known in \cite{Park_constr} that $\Ga_jf$ has a representation analogous to (\ref{character0}) with an equivalence similar to (\ref{charactera}), while $f\not= \sum_{j\in\bbz}\Ga_jf$ generally. Let $\wt{\theta}:=2^n\theta(2\cdot)$ and $\wt{\theta_j}:=2^{jn}\wt{\theta}(2^j\cdot)=\theta_{j+1}$ so that $\theta_j=\theta_j\ast \wt{\theta_{j}}$.
 Let $\theta^Q(x):=|Q|^{{1}/{2}}\theta_j(x-\xx_Q)$, $\wt{\theta}^Q(x):=|Q|^{{1}/{2}}\wt{\theta_{j}}(x-\xx_Q)$, and  $b_Q=\langle f,\wt{\theta}^Q\rangle$ for each $Q\in\mathcal{D}_j$.
Then we have
\begin{equation}\label{characterc1}
\Ga_jf(x)=\sum_{Q\in\mathcal{D}_j}b_Q\theta^Q(x)
\end{equation}
and  for $0<p<\infty$ and $0<q\le \infty$
\begin{equation}\label{characterc2}
\big\Vert \{\Ga_jf\}_{j\in\bbz}\big\Vert_{L^p(\ell^q)}\sim \Vert \bbbb\Vert_{\dot{f}^{p,q}}.
\end{equation}
We refer to \cite[Lemma 3.1]{Park_constr} for more details.

\section{Proof of Theorem \ref{main} : Reduction and interpolation}\label{proofoftheorem1}

The proof of Theorem \ref{main} can be obtained by interpolating the estimates in the following propositions. 

\begin{proposition}\label{mainproposition}
Let $0<p_1,p_2,p_3< \infty$ and $0<p< 1$ with $1/p=1/p_1+1/p_2+1/p_3$.
Suppose that $(1/p_1,1/p_2,1/p_3)\in \RR_{1}\cup \RR_2\cup\RR_3$ and
\begin{equation*}
s>\frac{n}{\min\{p_1,p_2,p_3\}}+\frac{n}{2}.
\end{equation*}
Let $\sigma$ be a function on $(\bbrn)^3$ satisfying $\LL_s^2[\sigma]<\infty$ and the vanishing moment condition \eqref{vanishingmoment}. 
Then we have
\begin{equation*}
\big\Vert T_{\sigma}(f_1,f_2,f_3)\big\Vert_{H^{p}(\bbrn)}\lesssim \LL_s^2[\sigma]\Vert f_1\Vert_{H^{p_1}(\bbrn)}\Vert f_2\Vert_{H^{p_2}(\bbrn)}\Vert f_3\Vert_{H^{p_3}(\bbrn)}
\end{equation*} 
for $f_1,f_2,f_3\in \mathscr{S}_0(\bbrn)$.
\end{proposition}
\begin{proposition}\label{mainproposition2}
Let $0<p\le 1$.
Suppose that one of $p_1,p_2,p_3$ is equal to $p$ and the other two are infinity.
Suppose that $s>n/p+n/2$.
Let $\sigma$ be a function on $(\bbrn)^3$ satisfying $\LL_s^2[\sigma]<\infty$ and the vanishing moment condition \eqref{vanishingmoment}. 
Then we have
\begin{equation*}
\big\Vert T_{\sigma}(f_1,f_2,f_3)\big\Vert_{H^{p}(\bbrn)}\lesssim \LL_s^2[\sigma]\Vert f_1\Vert_{H^{p_1}(\bbrn)}\Vert f_2\Vert_{H^{p_2}(\bbrn)}\Vert f_3\Vert_{H^{p_3}(\bbrn)}
\end{equation*} 
 for $f_1,f_2,f_3\in \mathscr{S}_0(\bbrn)$.
\end{proposition}

We present the proof of  Proposition \ref{mainproposition} in Sections \ref{pfproposition}, \ref{estimateii}, and \ref{estimatejj}, and that of Proposition \ref{mainproposition2} in Section \ref{pfpropositionp=1}.
For now, we proceed with the following interpolation argument, simply assuming the above propositions hold.
\begin{lemma}\label{interpolationlemma}
Let $0<p_1^0, p_2^0, p_3^0\le\infty$, $0<p_1^1, p_2^1, p_3^1\le\infty$, and $0<p^0,p^1<\infty$.
Suppose that 
$$\big\Vert T_{\sigma} \big\Vert_{H^{p_1^l}\times H^{p_2^l}\times H^{p_3^l}\to H^{p^l}}\lesssim \mathcal{A}, \qq l=0,1. $$
Then for any $0<\theta<1$, $0<p_1,p_2,p_3\le\infty$, and $0<p<\infty$ satisfying
$$1/p_j=(1-\theta)/p_j^0+\theta/p_j^1 \qq\text{for }~ j=1,2,3,$$
$$1/p=(1-\theta)/p^0+\theta/p^1,$$
we have
$$\Vert T_{\sigma}\Vert_{H^{p_1}\times H^{p_2}\times H^{p_3}\to H^p}\lesssim \mathcal{A}.$$
\end{lemma}
The proof of the lemma is essentially same as that of \cite[Lemma 2.4]{LHHLPPY}, so it is omitted here.

\subsection{Proof of \eqref{mainthmest} when $(1/p_1,1/p_2,1/p_3)\in \RR_4\cup \RR_5\cup \RR_6$}

We need to work only with $(1/p_1,1/p_2,1/p_3)\in \RR_4$ since the other cases are just symmetric versions.
In this case, $2<p_3<\infty$ and as mentioned in \eqref{conditionsinterpretation}, the condition \eqref{regularitycondition} is equivalent to 
$$s>n/p_1+n/p_2.$$
Now choose $\wt{p_1},\wt{p_2}<1$ such that
$$1/p_1+1/p_2=1/\wt{p_1}+1/2=1/2+1/\wt{p_2}$$
and thus $$s>n/\wt{p_1}+n/2, \q s>n/2+n/\wt{p_2}.$$
Let $\epsilon_1, \epsilon_2>0$ be numbers with
$$s>n/(\wt{p_1}-\epsilon_1)+n/2, \q s>n/2+n/(\wt{p_2}-\epsilon_2)$$
and select $q_1,q_2>2$  such that
$$1/p=1/(\wt{p_1}-\epsilon_1)+1/q_1+1/p_3=1/q_2+1/(\wt{p_2}-\epsilon_2)+1/p_3.$$
Then we observe that 
$$ (1-\theta)\Big(\frac{1}{\wt{p_1}-\epsilon_1},\frac{1}{q_1},\frac{1}{p_3} \Big)+\theta \Big( \frac{1}{q_2},\frac{1}{\wt{p_2}-\epsilon_2},\frac{1}{p_3}\Big)=\Big(\frac{1}{p_1},\frac{1}{p_2},\frac{1}{p_3}\Big)$$
for some $0<\theta<1$.
Let $C_1:=(1/(\wt{p_1}-\epsilon_1),1/q_1,1/p_3)$ and $C_2:=(1/q_2,1/(\wt{p_2}-\epsilon_2),1/p_3)$.
It is obvious that $C_1\in \RR_1$, $C_2\in \RR_2$ and thus it follows from Proposition \ref{mainproposition} that
\begin{align*}
\Vert T_{\sigma}\Vert_{H^{\wt{p_1}-\epsilon_1}\times H^{q_1}\times H^{p_3}\to H^p}\lesssim \LL_s^2[\sigma] \q &\text{ at }~ C_1=(1/(\wt{p_1}-\epsilon_1),1/q_1,1/p_3)\in \RR_1, \\
\Vert T_{\sigma}\Vert_{H^{q_2}\times H^{\wt{p_2}-\epsilon_2}\times H^{p_3}\to H^p}\lesssim \LL_s^2[\sigma] \q &\text{ at }~ C_2=(1/q_2,1/(\wt{p_2}-\epsilon_2),1/p_3)\in \RR_2.
\end{align*}
Finally, the assertion \eqref{mainthmest} for $(1/p_1,1/p_2,1/p_3)\in \RR_4$ is derived by means of interpolation in Lemma \ref{interpolationlemma}.
 \begin{figure}[h]
\begin{tikzpicture}
\path[fill=orange!5] (4.7,5.3,2)--(4.7,3.8,2)--(6,2.3,2)--(8,2.3,2)--(8,2.3,0)--(4.7,5.3,0)--(4.7,5.3,2);
\draw[orange!40, -] (4.7,5.3,2)--(8,2.3,2)--(8,2.3,0)--(4.7,5.3,0)--(4.7,5.3,2);
\draw[-][line width=0.25mm] (4.7,5.3,2)--(4.7,3.8,2)--(6,2.3,2)--(8,2.3,2);
\draw[dash pattern= { on 1pt off 1pt}] (4.7,3.8,0)--(4.7,5.3,0);
\draw[dash pattern= { on 1pt off 1pt}] (6,2.3,0)--(8,2.3,0);
\draw[dash pattern= { on 1pt off 1pt}] (4.7,3.8,0)--(6,2.3,0);
\draw[dash pattern= { on 1pt off 1pt}] (4.7,3.8,2)--(4.7,3.8,0);
\draw[dash pattern= { on 1pt off 1pt}] (6,2.3,0)--(6,2.3,2);
\filldraw[fill=black] (6.9,3.3,1.5)  circle[radius=0.5mm];
\filldraw[fill=black] (4.535,5.45,1.5)  circle[radius=0.3mm];
\filldraw[fill=black] (8.165,2.15,1.5)  circle[radius=0.3mm];
\filldraw[fill=black] (4.7,5.3,1.5)  circle[radius=0.3mm];
\filldraw[fill=black] (8,2.3,1.5)  circle[radius=0.3mm];
\draw [dash pattern= { on 2pt off 1pt}]  (4.535,5.45,1.5)--(8.165,2.15,1.5);
\draw[-] (6.9,3.3,1.5)--(6.4,4,-0.7);
\node[above right] at (6.4,4,-0.7) {\tiny$(\frac{1}{p_1},\frac{1}{p_2},\frac{1}{p_3})\in \RR_4$};
\node[left] at (4.535,5.45,1.5) {\tiny$\RR_2 \ni C_2=(\frac{1}{q_2},\frac{1}{\wt{p_2}-\epsilon_2},\frac{1}{p_3})$};
\draw[-] (4.7,5.3,1.5)--(5.7,5.3,-0.5);
\node[right] at (5.7,5.3,-0.5) {\tiny$(\frac{1}{2},\frac{1}{\wt{p_2}},\frac{1}{p_3})$};
\draw[-](8,2.3,1.5)--(7.5,2.65,-0.7);
\node[right] at (7.5,2.65,-0.7) {\tiny$(\frac{1}{\wt{p_1}},\frac{1}{2},\frac{1}{p_3})$};
\node[right] at (8.165,2.15,1.5) {\tiny$C_1=(\frac{1}{\wt{p_1}-\epsilon_1},\frac{1}{q_1},\frac{1}{p_3})\in \RR_1$};
\end{tikzpicture}
\caption{$(1-\theta)\big(\frac{1}{\wt{p_1}-\epsilon_1},\frac{1}{q_1},\frac{1}{p_3} \big)+\theta \big( \frac{1}{q_2},\frac{1}{\wt{p_2}-\epsilon_2},\frac{1}{p_3}\big)=(\frac{1}{p_1},\frac{1}{p_2},\frac{1}{p_3})\in \RR_4$}\label{fig4}
\end{figure}
See Figure \ref{fig4} for the interpolation.

\subsection{Proof of \eqref{mainthmest} when $(1/p_1,1/p_2,1/p_3)\in \RR_0$}
We first fix $1/2<p<1$ such that $1/p_1+1/p_2+1/p_3=1/p$ and assume that, in view of \eqref{conditionsinterpretation},
$$s>3n/2=n/1+n/2.$$
Then we choose $2<p_0<\infty$ such that $1+1/2+1/p_0=1/p$.
Then it is clear that $(1/p_1,1/p_2,1/p_3)$ is located inside the hexagon with the vertices  $(1,1/p_0,1/2)$, $(1,1/2,1/p_0)$, $(1/2,1,1/p_0)$,  $(1/p_0,1,1/2)$, $(1/p_0,1/2,1)$, and $(1/2,1/p_0,1)$.
Now we choose a sufficiently small $\epsilon>0$ and $2<\wt{p_0}<\infty$ such that 
$$\frac{1}{2+\epsilon}+\frac{1}{\wt{p_0}}=\frac{1}{2}+\frac{1}{p_0},$$ and
the point $(1/p_1,1/p_2,1/p_3)$ is still inside the smaller hexagon with  $D_1:=(1,1/\wt{p_0},1/(2+\epsilon))$, $D_2:=(1,1/(2+\epsilon),1/\wt{p_0})$,   $D_3:=(1/(2+\epsilon),1,1/\wt{p_0})$, $D_4:=(1/\wt{p_0},1,1/(2+\epsilon))$,  $D_5:=(1/\wt{p_0},1/(2+\epsilon),1)$, and $D_6:=(1/(2+\epsilon),1/\wt{p_0},1)$.
Now Proposition \ref{mainproposition} deduces that
$$\big\Vert T_{\sigma}\big\Vert_{H^{q_1}\times H^{q_2}\times H^{q_3}\to H^p}\lesssim \LL_s^2[\sigma]$$
for  $(1/q_1,1/q_2,1/q_3)\in \{D_1,D_2,D_3,D_4,D_5,D_6 \}$, as $D_1,D_2\in\RR_1$, $D_3,D_4\in\RR_2$, and $D_5,D_6\in\RR_3$.
This implies, via interpolation in Lemma \ref{interpolationlemma},
$$\big\Vert T_{\sigma}(f_1,f_2,f_3)\big\Vert_{H^p(\bbrn)}\lesssim \LL_s^2[\sigma]\Vert f_1\Vert_{H^{p_1}(\bbrn)}\Vert f_2\Vert_{H^{p_2}(\bbrn)}\Vert f_3\Vert_{H^{p_3}(\bbrn)}.$$
See Figure \ref{fig7} for the interpolation.

For the case $p=1$, we interpolate the estimates in Proposition \ref{mainproposition2}. To be specific, for any given $0<p_1,p_2,p_3<\infty$ with $1/p_1+1/p_2+1/p_3=1$, the estimate \eqref{mainthmest} with $p=1$ follows from interpolating
\begin{align*}
\big\Vert T_{\sigma}(f_1,f_2,f_3)\big\Vert_{H^1(\bbrn)}&\lesssim \LL_s^2[\sigma]\Vert f_1\Vert_{H^1(\bbrn)}\Vert f_2\Vert_{H^{\infty}(\bbrn)}\Vert f_3\Vert_{H^{\infty}(\bbrn)},\\
\big\Vert T_{\sigma}(f_1,f_2,f_3)\big\Vert_{H^1(\bbrn)}&\lesssim \LL_s^2[\sigma]\Vert f_1\Vert_{H^{\infty}(\bbrn)}\Vert f_2\Vert_{H^{1}(\bbrn)}\Vert f_3\Vert_{H^{\infty}(\bbrn)},\\
\big\Vert T_{\sigma}(f_1,f_2,f_3)\big\Vert_{H^1(\bbrn)}&\lesssim \LL_s^2[\sigma]\Vert f_1\Vert_{H^{\infty}(\bbrn)}\Vert f_2\Vert_{H^{\infty}(\bbrn)}\Vert f_3\Vert_{H^{1}(\bbrn)}.
\end{align*}

 \begin{figure}[h]
\begin{tikzpicture}

\path[fill=yellow!5] (1.6,1.5,4)--(1.6,3,2)--(3,1.5,2)--(1.6,1.5,4);
\path[fill=yellow!5] (0,3,2)--(0,1.5,4)--(1.6,1.5,4)--(1.6,3,2)--(0,3,2);
\path[fill=yellow!5] (1.6,3,2)--(1.6,3,0)--(3,1.5,0)--(3,1.5,2)--(1.6,3,2);
\path[fill=yellow!5] (1.6,1.5,4)--(1.6,0,4)--(3,0,2)--(3,1.5,2)--(1.6,1.5,4);

\path[fill=yellow!5] (0,3,0)--(1.6,3,0)--(1.6,3,2)--(0,3,2)--(0,3,0);
\path[fill=yellow!5] (0,0,4)--(1.6,0,4)--(1.6,1.5,2)--(0,1.5,4)--(0,0,4);
\path[fill=yellow!5] (3,0,0)--(3,0,2)--(3,1.5,2)--(3,1.5,0)--(3,0,0);

\draw[-][line width=0.25mm] (1.6,1.5,4)--(1.6,3,2)--(3,1.5,2)--(1.6,1.5,4);
\draw[-][line width=0.25mm] (1.6,1.5,4)--(1.6,0,4)--(3,0,2)--(3,1.5,2);
\draw[-][line width=0.25mm] (1.6,3,2)--(1.6,3,0)--(3,1.5,0)--(3,1.5,2);
\draw[-][line width=0.25mm] (1.6,1.5,4)--(0,1.5,4)--(0,3,2)--(1.6,3,2);
\draw[-][line width=0.25mm] (3,0,2)--(3,0,0)--(3,1.5,0);
\draw[-][line width=0.25mm] (0,3,2)--(0,3,0)--(1.6,3,0);
\draw[-][line width=0.25mm] (1.6,0,4)--(0,0,4)--(0,1.5,4);

\draw[dash pattern= { on 1pt off 1pt}] (3,0,0)--(0,3,0)--(0,0,4)--(3,0,0);

\draw[dash pattern= { on 2pt off 1pt}]  (1,1.4,4)--(1.5,1,4)--(3,1,1.75)--(3,1.4,1.3)--(1.5,3,1.3)--(1,3,1.75)--(1,1.4,4);

\filldraw[fill=black] (2,2.6,2.5)  circle[radius=0.5mm];
\filldraw[fill=black] (1,1.4,4)  circle[radius=0.3mm];
\filldraw[fill=black] (1.5,1,4)  circle[radius=0.3mm];

\filldraw[fill=black] (1.5,3,1.3)  circle[radius=0.3mm];
\filldraw[fill=black] (1,3,1.75)  circle[radius=0.3mm];

\filldraw[fill=black] (3,1.4,1.3)  circle[radius=0.3mm];
\filldraw[fill=black] (3,1,1.75)  circle[radius=0.3mm];

\draw[-] (2,2.6,2.5)--(4,2.6,1.5);
\node[right] at (4,2.6,1.5) {\tiny$(\frac{1}{{p_1}},\frac{1}{p_2},\frac{1}{p_3})\in \RR_0$};

\draw[-] (3,1,1.75)--(3.5,0.3,0.5);
\node[right] at (3.5,0.3,0.5){\tiny$D_1=(1,\frac{1}{\wt{p_0}},\frac{1}{1+\epsilon}) \in \RR_1$};

\draw[-] (3,1.4,1.3)--(3.5,1,0);
\node[right] at (3.5,1,0){\tiny$D_2=(1,\frac{1}{2+\epsilon},\frac{1}{\wt{p_0}}) \in \RR_1$};

\draw[-] (1.5,3,1.3)--(1.4,3.3,-0.5);
\node[right] at (1.4,3.3,-0.5){\tiny$D_3=(\frac{1}{2+\epsilon},1,\frac{1}{\wt{p_0}})\in \RR_2$};

\draw[-] (1,3,1.75)--(-1,3.3,-0.5);
\node[above] at (-1,3.3,-0.5){\tiny$D_4=(\frac{1}{\wt{p_0}},1,\frac{1}{2+\epsilon})\in \RR_2$};

\draw[-] (1,1.4,4)--(-0.5,2.2,3.5);
\node[left] at (-0.5,2.2,3.5){\tiny$\RR_3 \ni D_5=(\frac{1}{\wt{p_0}},\frac{1}{2+\epsilon},1)$};

\draw[-] (1.5,1,4)--(-1,0.5,3.5);
\node[left] at (-1,0.5,3.5){\tiny$\RR_3 \ni D_6=(\frac{1}{2+\epsilon},\frac{1}{\wt{p_0}},1)$};


\end{tikzpicture}
\caption{$\big(\frac{1}{p_1},\frac{1}{p_2},\frac{1}{p_3}\big)\in \RR_0$}\label{fig7}
\end{figure}

\subsection{Proof of \eqref{mainthmest} when $(1/p_1,1/p_2,1/p_3)\in \RR_7$}

Let $0<p\le 1/2$ be such that $1/p=1/p_1+1/p_2+1/p_3$, and assume that
$$s>n/p-n/2.$$
We choose $0<p_0\le 1$, satisfying $1/p_0+1=1/p$, so that
$$s>n/p_0+n/2.$$
Then there exist $\epsilon>0$ and $2<q<\infty$ so that $s>n/(p_0-\epsilon)+n/2$ and $1/p=1/(p_0-\epsilon)+2/q$.
Let $E_1:=\big(1/(p_0-\epsilon),1/q,1/q\big)$, $E_2:=\big(1/q,1/(p_0-\epsilon),1/q\big)$, and $E_3:=\big(1/q,1/q,1/(p_0-\epsilon)\big)$.
Then it is immediately verified that $E_1\in \RR_1$, $E_2\in \RR_2$, $E_3\in \RR_3$, and 
$$\theta_1\Big(\frac{1}{(p_0-\epsilon)},\frac{1}{q},\frac{1}{q}\Big)+\theta_2\Big(\frac{1}{q},\frac{1}{(p_0-\epsilon)},\frac{1}{q}\Big)+\theta_3\Big(\frac{1}{q},\frac{1}{q},\frac{1}{(p_0-\epsilon)}\Big)=\Big(\frac{1}{p_1},\frac{1}{p_2},\frac{1}{p_3}\Big)$$
for some  $0<\theta_1,\theta_2,\theta_3<1$ with $\theta_1+\theta_2+\theta_3=1$.
Therefore, Proposition \ref{mainproposition} yields that
\begin{align*}
\Vert T_{\sigma}\Vert_{H^{p_0-\epsilon}\times H^{q}\times H^{q}\to H^p}\lesssim \LL_s^2[\sigma] \q &\text{ at }~ E_1=\big(1/(p_0-\epsilon),1/q,1/q\big)\in \RR_1, \\
\Vert T_{\sigma}\Vert_{H^{q}\times H^{p_0-\epsilon}\times H^{q}\to H^p}\lesssim \LL_s^2[\sigma] \q &\text{ at }~ E_2=\big(1/q,1/(p_0-\epsilon),1/q\big)\in \RR_2, \\
\Vert T_{\sigma}\Vert_{H^{q}\times H^{q}\times H^{p_0-\epsilon}\to H^p}\lesssim \LL_s^2[\sigma] \q &\text{ at }~ E_3=\big(1/q,1/q,1/(p_0-\epsilon)\big)\in \RR_3,
\end{align*}
and using the interpolation method in Lemma \ref{interpolationlemma}, we conclude the estimate \eqref{mainthmest} holds for $(1/p_1,1/p_2,1/p_3)\in \RR_7$.
See Figure \ref{fig8} for the interpolation.

 \begin{figure}[h]
\begin{tikzpicture}

\path[fill=gray!5] (5,1.5,2)--(1.7,4.5,2)--(1.7,1.5,6)--(5,1.5,2);

\draw[dash pattern= { on 1pt off 1pt}] (1.7,1.5,4)--(1.7,3,2)--(3,1.5,2)--(1.7,1.5,4);
\draw[dash pattern= { on 1pt off 1pt}] (1.7,1.5,4)--(1.7,1.5,6);
\draw[dash pattern= { on 1pt off 1pt}] (1.7,3,2)--(1.7,4.5,2);
\draw[dash pattern= { on 1pt off 1pt}] (3,1.5,2)--(5,1.5,2);

\draw[gray!40, -] (5,1.5,2)--(1.7,4.5,2)--(1.7,1.5,6)--(5,1.5,2);

\filldraw[fill=black] (2.5,2.7,3.3)  circle[radius=0.5mm];
\filldraw[fill=black] (1.8,4,2)  circle[radius=0.3mm];
\filldraw[fill=black] (1.85,1.5,5.3)  circle[radius=0.3mm];
\filldraw[fill=black] (4.3,1.6,2)  circle[radius=0.3mm];

\draw[dash pattern= { on 2pt off 1pt}]  (1.8,4,2)--(1.85,1.5,5.3)--(4.3,1.6,2)--(1.8,4,2);
\draw[dotted] (1.8,4,2)--(2.5,2.7,3.3);
\draw[dotted] (1.85,1.5,5.3)--(2.5,2.7,3.3);
\draw[dotted] (4.3,1.6,2)--(2.5,2.7,3.3);

\draw[-] (2.5,2.7,3.3)--(3.5,2.7,0.5);
\node[right] at (3.5,2.7,0.5){\tiny$(\frac{1}{p_1},\frac{1}{p_2},\frac{1}{p_3})\in \RR_7$};

\draw[-] (1.8,4,2)--(2.5,4.2,2);
\node[right] at (2.5,4.2,2){\tiny$E_2=(\frac{1}{q},\frac{1}{p_0-\epsilon},\frac{1}{q})\in \RR_2$};

\draw[-] (1.85,1.5,5.3)--(0.5,1.5,3);
\node[left] at (0.5,1.5,3){\tiny$\RR_3\ni E_3=(\frac{1}{q},\frac{1}{q},\frac{1}{p_0-\epsilon})$};

\draw[-] (4.3,1.6,2)--(5,1.3,3);
\node[below right] at (5,1.3,3){\tiny$E_1=(\frac{1}{p_0-\epsilon},\frac{1}{q},\frac{1}{q}) \in \RR_1$};

\end{tikzpicture}
\caption{$\theta_1\big(\frac{1}{p_0-\epsilon},\frac{1}{q},\frac{1}{q}\big)+\theta_2\big(\frac{1}{q},\frac{1}{p_0-\epsilon},\frac{1}{q}\big)+\theta_3\big(\frac{1}{q},\frac{1}{q},\frac{1}{p_0-\epsilon}\big)=\big(\frac{1}{p_1},\frac{1}{p_2},\frac{1}{p_3}\big)\in \RR_7$}\label{fig8}
\end{figure}

\section{Auxiliary lemmas}
This section is devoted to providing several technical results which will be repeatedly used in the proof of Propositions \ref{mainproposition} and \ref{mainproposition2}.
\begin{lemma}\label{epsilon control}
Let $N\in\mathbb{N}$ and $a\in\R^n$.
Suppose that a Schwartz function $f$, defined on $\bbrn$, satisfies
\begin{equation}\label{vanishingp}
\int_{\R^n}{x^{\alpha}f(x)}dx=0 \quad \text{for all multi-indices}~ \alpha ~ \text{ with }~|\alpha|\leq N.
\end{equation}
Then for any $0\leq \epsilon\leq 1$, there exists a constant $C_{\epsilon}>0$ such that
\begin{equation*}
\big\Vert \phi_l \ast f\big\Vert_{L^{\infty}(\R^n)}\le C_{\epsilon} 2^{l(N+n+\epsilon)}\int_{\R^n}{|y-a|^{N+\epsilon}|f(y)|}\; dy
\end{equation*}
\end{lemma}
\begin{proof}
Using the Taylor theorem for $\phi_l$, we write
\begin{align*}
\phi_l(x-y)&=\sum_{|\alpha|\le N-1}\frac{\partial^{\alpha}\phi_l(x-a)}{\alpha !}(a-y)^{\alpha}\\
&\qq +N\sum_{|\alpha|=N}\frac{1}{\alpha !}\Big(\int_0^1{(1-t)^{N-1}\partial^{\alpha}\phi_l\big(x-a+t(a-y)\big)}dt \Big) (a-y)^{\alpha}.
\end{align*}
Then it follows from the condition \eqref{vanishingp} that
\begin{align*}
\big| \phi_l\ast f(x)\big|&\lesssim_{N} \sum_{|\alpha|=N}\Big| \int_{\R^n}{\Big( \int_0^1 (1-t)^{N-1}\partial^{\alpha}\phi_l\big(x-a+t(a-y)\big)   dt\Big) (a-y)^{\alpha}f(y)}\; dy\Big|\\
&\lesssim_{N} \sum_{|\alpha|=N}\Big| \int_{\R^n}{\Big[ \int_0^1 (1-t)^{N-1}\partial^{\alpha}\phi_l\big(x-a+t(a-y)\big)  \; dt}\\
&\qqqq  {-\int_0^1 (1-t)^{N-1}\partial^{\alpha}\phi_l(x-a)   dt       \Big] (a-y)^{\alpha}f(y)}\; dy\Big|\\
&\lesssim \sum_{|\alpha|=N}\int_{\R^n}\Big(\int_0^1\big|\partial^{\alpha}\phi_l\big(x-a+t(a-y)\big)-\partial^{\alpha}\phi_l(x-a) \big|dt\Big) |y-a|^N |f(y)| \; dy.
\end{align*}
For $|\alpha|=N$, we note that
\begin{equation}\label{avg1}
\big|\partial^{\alpha}\phi_l\big(x-a+t(a-y)\big)-\partial^{\alpha}\phi_l(x-a) \big|\lesssim 2^{l(N+n+1)}|y-a|
\end{equation}
and
\begin{align}\label{avg2}
&\big|\partial^{\alpha}\phi_l\big(x-a+t(a-y)\big)-\partial^{\alpha}\phi_l(x-a) \big|\nonumber\\
&\le \big|\partial^{\alpha}\phi_l\big(x-a+t(a-y)\big)\big|+\big|\partial^{\alpha}\phi_l(x-a)\big|\lesssim 2^{l(N+n)}.
\end{align}
Then by averaging both \eqref{avg1} and \eqref{avg2}, we obtain that 
\begin{equation*}
\big|\partial^{\alpha}\phi_l\big(x-a+t(a-y)\big)-\partial^{\alpha}\phi_l(x-a) \big|\lesssim_{\epsilon} 2^{l(N+n+\epsilon)}|y-a|^{\epsilon}, \qq 0\le \epsilon\le 1,
\end{equation*}
which completes the proof.
\end{proof}

Now we recall that $\wt{\psi_j}=\psi_{j-1}+\psi_j+\psi_{j+1}$ and $\wt{\theta_j}=2^{n}\theta_j(2\cdot)$, and then define
$\wt{\La_j}g:=\wt{\psi_j}\ast g$ and $\wt{\Ga_j}g:=\wt{\theta_j}\ast g$.
\begin{lemma}\label{almost orthogonality}
Let $2\le q<\infty$, $s>{n}/{q}$, and $L>n,s$. 
Let  $\varphi$ be a function on $\bbrn$ satisfying $$|\varphi(x)|\lesssim_M\frac{1}{(1+|x|)^M} \qq\text{ for all}~ M>0.$$ 
For $j\in\Z$ and for each $Q\in \mathcal{D}_j$ let 
$$\varphi^Q(x):=2^{{jn}/{2}}\varphi\big(2^j(x-\xx_Q)\big),$$ and for a Schwartz function $g$ on $\bbrn$ let
$$\BB_Q(g):=\Big\langle \big| \wt{\La_j}g\big|,\frac{2^{{jn}/{2}}}{\langle 2^j(\cdot-\xx_Q)\rangle^L}\Big\rangle \q\text{ or }\q \Big\langle \big| \wt{\Ga_j}g\big|,\frac{2^{{jn}/{2}}}{\langle 2^j(\cdot-\xx_Q)\rangle^L}\Big\rangle.$$
Then we have
\begin{align*}
&\bigg\Vert \sum_{Q\in\mathcal{D}_j} |\BB_Q(g)|\frac{\chi_{Q^c}(x)}{\langle 2^j(x-\xx_Q)\rangle^{s}}\big| \varphi^Q(z)\big|\bigg\Vert_{L^q(z)}\\
&\lesssim_L  2^{-{jn}/{q}}\bigg(\sum_{Q\in\mathcal{D}_j}\Big(|\BB_Q(g)| |Q|^{-1/2}\frac{\chi_{Q^c}(x)}{\langle 2^j(x-\xx_Q)\rangle^{s}} \Big)^q \bigg)^{{1}/{q}}.
\end{align*}
\end{lemma}

\begin{proof}
For $2\le q<\infty$, we have
\begin{align*}
&\Big(\sum_{Q\in\mathcal{D}_j} |\BB_Q(g)|\frac{\chi_{Q^c}(x)}{\langle 2^j(x-\xx_Q)\rangle^{s}}\big| \varphi^Q(z)\big|\Big)^q\\
&\lesssim   \Big(\sum_{Q\in\mathcal{D}_j} |\BB_Q(g)|^{{q}/{2}}\frac{\chi_{Q^c}(x)}{\langle 2^j(x-\xx_Q)\rangle^{{qs}/{2}}}\big| \varphi^Q(z)\big|\Big)^2\Big(\sum_{Q\in\mathcal{D}_j}\big|\varphi^Q(z) \big| \Big)^{q-2}
\end{align*}
where H\"older's inequality is applied if $2<q<\infty$.
Clearly,
\begin{align*}
\Big(\sum_{Q\in\mathcal{D}_j}\big|\varphi^Q(z) \big| \Big)^{q-2}&\lesssim_M 2^{{jn(q-2)}/{2} }\Big(\sum_{Q\in\mathcal{D}_j}\frac{1}{\langle 2^j(z-\xx_Q)\rangle^M} \Big)^{q-2}\\
&=2^{{jn(q-2)}/{2} }\Big(\sum_{\mm \in \Z^n}\frac{1}{\langle 2^jz-\mm\rangle^M} \Big)^{q-2}\lesssim_M 2^{{jn(q-2)}/{2}}
\end{align*} for sufficiently large $M>n$.
Therefore, the left-hand side of the claimed estimate is less than a constant times
\begin{equation}\label{mainestest}
2^{jn({1}/{2}-{1}/{q})}\Big\Vert  \sum_{Q\in\mathcal{D}_j} |\BB_P(g)|^{{q}/{2}}\frac{\chi_{Q^c}(x)}{\langle 2^j(x-\xx_Q)\rangle^{{qs}/{2}}}\big| \varphi^Q(z)\big|   \Big\Vert_{L^2(z)}^{{2}/{q}}.
\end{equation}
 The $L^2$ norm is dominated by 
\begin{align*}
& \Big( \sum_{Q\in\mathcal{D}_j}|\BB_Q(g)|^{{q}/{2}}\frac{\chi_{Q^c(x)}}{\langle 2^j(x-\xx_Q)\rangle^{{qs}/{2}}} \sum_{R\in\mathcal{D}_j}|\BB_R(g)|^{{q}/{2}}\frac{1}{\langle 2^j(x-\xx_R)\rangle^{{qs}/{2}}}\big\langle |\varphi^Q|,|\varphi^R|\big\rangle\Big)^{{1}/{2}}.
\end{align*}
Note that $$\big\langle |\varphi^Q|,|\varphi^R|\big\rangle\lesssim_{q,L} \frac{1}{\langle 2^j(\xx_Q-\xx_R)\rangle^{{3Lq}/{2}}}$$
and thus the preceding term is controlled by a constant multiple of
$$\Big(\sum_{Q\in\mathcal{D}_j}\big|\BB_Q(g)\big|^q\frac{\chi_{Q^c}(x)}{\langle 2^j(x-\xx_Q)\rangle^{qs}}\sum_{R\in\mathcal{D}_j}\frac{1}{\langle 2^j(\xx_Q-\xx_R)\rangle^{{Lq}/{2}}}\Big)^{{1}/{2}}.$$
Here, we used the facts that
$$\frac{\big|\BB_R(g)\big|}{(1+2^j|\xx_Q-\xx_R|)^{L}}\le \big|\BB_Q(g)\big|$$
and
$$\frac{1}{\langle 2^j(x-\xx_R)\rangle^{{qs}/{2}}\langle 2^j(\xx_Q-\xx_R)\rangle^{{Lq}/{2}}}\le \frac{1}{\langle 2^j(x-\xx_Q)\rangle^{{qs}/{2}}}.$$
Since the sum over $R\in\mathcal{D}_j$ converges, we deduce
\begin{align*}
\eqref{mainestest}&\lesssim 2^{-{jn}/{q}} \bigg(\sum_{Q\in\mathcal{D}_j}  \Big( \big|\BB_Q(g)\big||Q|^{-{1}/{2}}\frac{\chi_{Q^c}(x)}{\langle 2^j(x-\xx_Q)\rangle^{s}}\Big)^q \bigg)^{{1}/{q}}
\end{align*}
and thus the desired result follows.
\end{proof}

\begin{lemma}\label{lacunarylemma}
Let $2\le p,q< \infty$, $s>{n}/\min{\{p,q\}}$, and $L>n,s$.
For $j\in\bbz$ and $Q\in\DD_j$, let 
$$\BB_Q(g):= \Big\langle \big| \wt{\La_j}g\big|,\frac{2^{{jn}/{2}}}{\langle 2^j(\cdot-\xx_Q)\rangle^L}\Big\rangle$$
where $g$ is a Schwartz function on $\bbrn$.
 Then we have
\begin{align}\label{lacunaryest}
&\bigg\Vert \bigg( \sum_{j\in\bbz}\sum_{Q\in\DD_j}\Big( \big|\BB_Q(g)\big||Q|^{-{1}/{2}}\frac{\chi_{Q^c}(\cdot)}{\langle 2^j(\cdot-\xx_Q)\rangle^s}\Big)^q \bigg)^{{1}/{q}} \bigg\Vert_{L^p(\bbrn)}\lesssim \Vert g\Vert_{L^p(\bbrn)}.
\end{align}
\end{lemma}
\begin{proof}
It is easy to verify that for $Q\in\DD_j$
$$\frac{1}{\langle 2^j|x-\xx_Q|\rangle^{s}}\chi_{Q^c}(x)\lesssim \mathcal{M}_{\frac{n}{s}}\chi_Q(x)  $$
and thus the left-hand side of \eqref{lacunaryest} is less than a constant multiple of 
\begin{align*}
&\bigg\Vert     \bigg( \sum_{Q\in\DD} \Big( \big|\BB_Q(g) \big| |Q|^{-{1}/{2}} \mathcal{M}_{\frac{n}{s}}\chi_Q(\cdot)\Big)^q\bigg)^{{1}/{q}}\bigg\Vert_{L^p(\bbrn)}\\
&\lesssim \bigg\Vert     \bigg( \sum_{Q\in\DD} \Big( \big|\BB_Q(g) \big| |Q|^{-{1}/{2}} \chi_Q(\cdot)\Big)^q\bigg)^{{1}/{q}}\bigg\Vert_{L^p(\bbrn)}
\end{align*}
by virtue of the maximal inequality \eqref{hlmax} with $s>{n}/\min{\{p,q \}}$.
 We see that
\begin{align*}
& \bigg( \sum_{Q\in\DD} \Big( \big|\BB_Q(g) \big| |Q|^{-{1}/{2}} \chi_Q(x)\Big)^q\bigg)^{{1}/{q}} \le  \bigg( \sum_{Q\in\DD} \Big( \big|\BB_Q(g) \big| |Q|^{-{1}/{2}} \chi_Q(x)\Big)^2\bigg)^{{1}/{2}}\\
&=\bigg( \sum_{j\in\bbz}\sum_{Q\in\DD_j}\chi_Q(x)\Big(\int_{\bbrn}\big| \wt{\La_j}g(y)\big|\frac{2^{jn}}{\langle 2^j(y-\xx_Q)\rangle^L} \; dy \Big)^2\bigg)^{{1}/{2}}\\
&\lesssim\bigg( \sum_{j\in\bbz}\Big(\int_{\bbrn}\big| \wt{\La_j}g(y)\big|\frac{2^{jn}}{\langle 2^j(y-x)\rangle^L} \;dy \Big)^2\bigg)^{{1}/{2}}\lesssim \big\Vert \big\{ \mathcal{M}\La_jg(x)\big\}_{j\in\bbz}\big\Vert_{\ell^2}
\end{align*}
since  $\ell^2\hookrightarrow \ell^q$, $L>n$ and $\langle 2^j(y-\xx_Q)\rangle\gtrsim \langle 2^j(y-x)\rangle$ for $Q\in\DD_j$ and $x\in Q$.
 Using \eqref{hlmax} again, the left-hand side of \eqref{lacunaryest} is less than a constant times
\begin{equation*}
\big\Vert    \big\{ \La_jg\big\}_{j\in\bbz}  \big\Vert_{L^{p}(\ell^2)}\sim \Vert g\Vert_{L^{p}(\bbrn)}.
\end{equation*}
\end{proof}

\begin{lemma}\label{nonlacunarylemma}
Let $1 \le q < \infty$,  $s>{n}/{q}$, and $L>n,s$.
For $j\in\bbz$ and $Q\in\DD_j$, let 
$$\BB_Q(g):= \Big\langle \big| \wt{\Ga_j}g\big|,\frac{2^{\frac{jn}{2}}}{\langle 2^j(\cdot-c_Q)\rangle^L}\Big\rangle$$
where $g$ is a Schwartz function on $\bbrn$.
 Then  for $1<p\le \infty$ with $q\le p$ we have
\begin{align}\label{nonlacunaryest}
&\bigg\Vert   \sup_{j\in\bbz}\bigg( \sum_{Q\in\DD_j}\Big( \big|\BB_Q(g)\big||Q|^{-{1}/{2}}\frac{1}{\langle 2^j(\cdot-c_Q)\rangle^s}\Big)^q \bigg)^{{1}/{q}}  \bigg\Vert_{L^p(\bbrn)}\lesssim \Vert g\Vert_{L^p(\bbrn)}.
\end{align}
\end{lemma}

\begin{proof}
For any $j\in\bbz$ and $Q\in\DD_j$, there exists a unique lattice $\mm_Q\in\bbzn$ such that $\xx_Q=2^{-j}\mm_Q$.
For any $j\in\bbz$ and $x\in\bbrn$, let $Q_{j,x}$ be a unique dyadic cube in $\DD_j$ containing $x$.
Then we have the representations  $\xx_{Q_{j,x}}=2^{-j}\mm_{Q_{j,x}}$ for $\mm_{Q_{j,x}}\in\bbzn$ and
$$x=2^{-j}(\mm_{Q_{j,x}}+u_x) \qq \text{ for some }~ ~u_x\in [0,1)^n.$$
Now for $Q\in\DD_j$, we write
\begin{align*}
|\BB_Q(g)| |Q|^{-{1}/{2}}&\lesssim_L\int_{\bbrn} \big| \mathcal{M}g(y)\big|  \frac{2^{jn}}{\langle 2^j(y-\cc_Q)\rangle^L}      dy\lesssim \int_{\bbrn} \big| \mathcal{M}g(y)\big|  \frac{2^{jn}}{\langle 2^j(y-\xx_Q)\rangle^L}  \;    dy\\
&=\int_{\bbrn} \big| \mathcal{M}g(y)\big|  \frac{2^{jn}}{\big\langle 2^j(y-x)+\mm_{Q_{j,x}}-\mm_Q+u_x \big\rangle^L}  \;    dy\\
&\lesssim_L\int_{\bbrn} \big| \mathcal{M}g(y)\big|  \frac{2^{jn}}{\big\langle 2^j(y-x)+\mm_{Q_{j,x}}-\mm_Q \big\rangle^L}     \; dy\\
&\lesssim \mathcal{M}_{dyad}^{\mm_{Q_{j,x}}-\mm_Q}\mathcal{M}g(x)
\end{align*}
 where the penultimate inequality follows from the fact that $u_x\in [0,1)^n$.
This deduces 
\begin{align*}
&\sup_{j\in\bbz}\bigg( \sum_{Q\in\DD_j}\Big( \big|\BB_Q(g)\big||Q|^{-{1}/{2}}\frac{1}{\langle 2^j(x-c_Q)\rangle^s}\Big)^q \bigg)^{{1}/{q}}\\
&\lesssim \sup_{j\in\bbz}\bigg(  \sum_{\mm\in\bbzn}\Big(  \mathcal{M}_{dyad}^{\mm_{Q_{j,x}}-\mm}\mathcal{M}g(x)\frac{1}{\langle \mm_{Q_{j,x}}-\mm\rangle^s} \Big)^q      \bigg)^{{1}/{q}}\\
&=\bigg(  \sum_{\mm\in\bbzn}\Big(  \mathcal{M}_{dyad}^{\mm}\mathcal{M}g(x)\frac{1}{\langle \mm\rangle^s} \Big)^q      \bigg)^{{1}/{q}}.
\end{align*}
Therefore, the left-hand side of \eqref{nonlacunaryest} is less than a constant times
\begin{align*}
\Big( \sum_{\mm\in\bbzn} \langle \mm\rangle^{-sq}\big\Vert  \mathcal{M}_{dyad}^{\mm}\mathcal{M}g    \big\Vert_{L^p(\bbrn)}^q \Big)^{{1}/{q}}&\lesssim \Vert \mathcal{M}g\Vert_{L^p(\bbrn)}\Big( \sum_{\mm\in\bbzn}\langle \mm\rangle^{-sq}\big(\log \big(10+|\mm| \big) \big)^{{qn}/{p}}\Big)^{{1}/{q}}\\
&\lesssim \Vert g\Vert_{L^p(\bbrn)}
\end{align*}
as $sq>n$, where we applied Minkowski's inequality if $p >  q$ and the maximal inequality \eqref{okl}. 
This completes the proof.
\end{proof}

\begin{lemma}\label{technical lemma}
Let $a$ be an $H^p$-atom associated with $Q$, satisfying 
\begin{equation}\label{vmcondition}
\int_{\bbrn}x^{\ga}a(x) \; dx=0 \q \text{ for all multi-indices }~ |\ga|\le M,
\end{equation}
 and fix $L_0>0$.
Then we have
\begin{equation}\label{estlaak}
\big| \La_j a (x) \big| \lesssim_{L_0} l(Q)^{-{n}/{p}} \min\big\{ 1, \big( 2^jl(Q) \big)^{M+n+1} \big\} \bigg(  \chi_{Q^{*}}(x)  + \chi_{(Q^{*})^c}(x)\frac{1}{\langle 2^j(x-\xx_Q)\rangle^{L_0}} \bigg),
\end{equation}and
\begin{equation}\label{estgaak}
\big| \Ga_j a (x) \big| \lesssim_{L_0} l(Q)^{-{n}/{p}} \min\big\{ 1, \big( 2^jl(Q) \big)^{M+n+1} \big\} \bigg(  \chi_{Q^{*}}(x)  + \chi_{(Q^{*})^c}(x)\frac{1}{\langle 2^j(x-\xx_Q)\rangle^{L_0}} \bigg).
\end{equation}
Moreover, for $1\leq r \leq \infty$,
\begin{equation}\label{lajalr}
\| \La_j a \|_{L^r(\bbrn)}, \| \Ga_j a \|_{L^r(\bbrn)} \lesssim l(Q)^{-{n}/{p} + {n}/{r}} \min\{ 1, (2^jl(Q))^{M+n - {n}/{r}+1} \}.
\end{equation}
\end{lemma}
\begin{proof}
We will prove only the estimates for $\La_ja$ and the exactly same argument is applicable to $\Ga_ja$ as well.
Let us first assume $2^j\ell(Q)\ge 1$.
Then we have
\begin{align*}
\big| \La_j a(x)\big|&\le \ell(Q)^{-{n}/{p}}\Big(  \chi_{Q^*}(x)\Vert \psi_j\Vert_{L^1(\bbrn)}+\chi_{(Q^*)^c}(x)\int_{y\in Q}\big|\psi_j(x-y)\big| \; dy     \Big)\\
&\lesssim_{L_0}  \ell(Q)^{-{n}/{p}}\Big(  \chi_{Q^*}(x)+\chi_{(Q^*)^c}(x)\int_{y\in Q}\frac{2^{jn}}{\langle 2^j(x-y)\rangle^{n+1}}\frac{1}{\langle 2^j(x-y)\rangle^{L_0}} \; dy     \Big)\\
&\lesssim  \ell(Q)^{-{n}/{p}}\Big(  \chi_{Q^*}(x)+\chi_{(Q^*)^c}(x)\frac{1}{\langle 2^j(x-\xx_Q)\rangle^{L_0}}     \Big)
\end{align*}
since $|x-y|\gtrsim |x-\xx_Q|    $ for $x\in (Q^*)^c$ and $y\in Q$.

Now suppose that $2^j\ell(Q)<1$.
By using the vanishing moment condition \eqref{vmcondition}, we obtain
$$\big| \La_j a(x)\big|\le 2^{j(M+n+1)}\int_Q \int_0^1  \frac{1}{\langle 2^j(x-ty-(1-t)\xx_Q)\rangle^{L_0}}|y-\xx_Q|^{M+1}|a(y)|    \;  dt dy.$$
If $x\in Q^*$, 
then it is clear that
$$\big| \La_j a(x)\big|\lesssim \big( 2^j\ell(Q)\big)^{M+n+1} \ell(Q)^{-{n}/{p}}.$$
If $x\in (Q^*)^c$, then
we have 
$$\langle 2^j(x-ty-(1-t)\xx_Q)\rangle^{-1}\lesssim \langle 2^j(x-\xx_Q)\rangle^{-1},$$
which implies
$$\big| \La_j a(x)\big|\lesssim_{L_0} \frac{1}{\langle 2^j(x-\xx_Q)\rangle^{L_0}}\big( 2^j\ell(Q)\big)^{M+n+1} \ell(Q)^{-{n}/{p}}.$$
This proves \eqref{estlaak}.

Moreover, using the estimate \eqref{estlaak},  we have
\begin{align*}
\big\Vert \La_ja\big\Vert_{L^r(\bbrn)}&\le \big\Vert \La_ja\big\Vert_{L^r(Q^*)}+\big\Vert \La_ja\big\Vert_{L^r((Q^*)^c)}\\
&\lesssim  \ell(Q)^{-{n}/{p}}\min\big\{ 1, \big( 2^jl(Q) \big)^{M+n+1} \big\}\Big(|Q|^{{1}/{r}}+\Big\Vert \frac{1}{\langle 2^j(\cdot-\xx_Q)\rangle^{{n+1}}}\Big\Vert_{L^r(\bbrn)} \Big)\\
&\lesssim \ell(Q)^{-{n}/{p}+{n}/{r}}\min\big\{ 1, \big( 2^jl(Q) \big)^{M+n+1} \big\}\Big( 1+\big( 2^j\ell(Q)\big)^{-{n}/{r}}\Big)\\
&\le \ell(Q)^{-{n}/{p}+{n}/{r}}\min\big\{ 1, \big( 2^jl(Q) \big)^{M+n-{n}/{r}+1} \big\}.
\end{align*}
This concludes the proof of \eqref{lajalr}.
\end{proof}

\section{Proof of Proposition \ref{mainproposition} : Reduction }\label{pfproposition}

\subsection{Reduction via paraproduct}
Without loss of generality, we may assume
$$\Vert f_1\Vert_{H^{p_1}(\bbrn)}=\Vert f_2\Vert_{H^{p_2}(\bbrn)}=\Vert f_3\Vert_{H^{p_3}(\bbrn)}=1 \q \text{and}\q \LL_s^2[\sigma]=1.$$

We first note that $T_{\sigma}(f_1,f_2,f_3)$ can be written as
\begin{equation*}
T_{\sigma}(f_1,f_2,f_3)=\sum_{j_1,j_2,j_3\in\bbz}T_{\sigma}\big(\La_{j_1}f_1,\La_{j_2}f_2,\La_{j_3}f_3\big).
\end{equation*} 
We shall work with only the case $j_1\ge j_2\ge j_3$ as other cases follow from a symmetric argument.
When $j_1\ge j_2\ge j_3$, it is easy to verify that
$$T_{\sigma}\big(\La_{j_1}f_1,\La_{j_2}f_2,\La_{j_3}f_3\big)=T_{\sigma_{j_1}}\big(\La_{j_1}f_1,\La_{j_2}f_2,\La_{j_3}f_3\big)$$
where $\sigma_j(\xxxi):=\sigma(\xxxi) \wh{\Theta}(\xxxi/2^j)$ and $\wh{\Theta}(\xxxi):=\sum_{l=-2}^2\wh{\Psi}(2^{l}\xxxi)$ so that
$\wh{\Theta}(\xxxi)=1$ for $2^{-2}\le |\xxxi|\le 2^{2}$ and $\supp(\wh{\Theta})\subset \{\xxxi\in (\bbrn)^3:2^{-3}\le |\xxxi|\le 2^3\}$.
Then we observe that
\begin{equation}\label{supequiv}
\sup_{k\in\bbz}\big\Vert \sigma_k(2^k\ccdot)\big\Vert_{L_s^2((\bbrn)^3)}\lesssim \LL_{s}^{2}[\sigma]=1
\end{equation}
by virtue of the triangle inequality.
Moreover, using the fact that 
$
\Ga_jf=\sum_{k\le j}\La_kf,
$
we write
\begin{align*}
&\sum_{j_1\in\bbz}\;\sum_{j_2,j_3\in\bbz: j_3\le j_2\le j_1}T_{\sigma_{j_1}}\big(\La_{j_1}f_1,\La_{j_2}f_2,\La_{j_3}f_3\big)\\
&=\sum_{j\in\bbz}T_{\sigma_{j}}\big(\La_{j}f_1,\Ga_{j-10}f_2,\Ga_{j-10}f_3\big)+\sum_{k=0}^{9}\sum_{j\in\bbz}T_{\sigma_{j}}\big(\La_{j}f_1,\La_{j-k}f_2,\Ga_{j-k}f_3\big)\\
&=:T_{\sigma}^{(1)}(f_1,f_2,f_3)+\sum_{k=0}^{9}T_{\sigma}^{(2),k}(f_1,f_2,f_3),
\end{align*}
and especially, let $T_{\sigma}^{(2)}:=T_{\sigma}^{(2),0}$.
Then it is enough to prove that 
\begin{equation}\label{mainreductionest}
\big\Vert T_{\sigma}^{(\mu)}(f_1,f_2,f_3)\big\Vert_{H^p(\bbrn)}\lesssim 1, \qq \mu=1,2
\end{equation}
since the operator $T_{\sigma}^{(2),k}$, $1\le k\le 9$, can be handled in the same way as $T_{\sigma}^{(2)}$.

It should be remarked that
 the vanishing moment condition \eqref{vanishingmoment} now implies 
\begin{equation}\label{vanishingreduction}
\int_{\bbrn}x^{\alpha}T_{\sigma_j}\big( f_1,f_2,f_3\big)(x) \; dx=0 \q \text{ for all multi-indices ~}|\alpha|\le \frac{n}{p}-n.
\end{equation}

\subsection{Proof of \eqref{mainreductionest} for $\mu=1$}
In this case, we may simply follow the arguments used in the proof of Theorems \ref{thmb} and \ref{thmd}. 
The proof is based on the fact that if $\wh{g_k}$ is supported in $\{\xi \in\bbrn: C_0^{-1} 2^{k}\le |\xi|\le C_02^{k}\}$ for $C_0>1$ then
\begin{equation}\label{marshall}
\bigg\Vert \Big\{ \La_j \Big(\sum_{k\in\bbz}{g_k}\Big)\Big\}_{j\in\mathbb{Z}}\bigg\Vert_{L^p(\ell^q)}\lesssim_{C_0} \big\Vert \big\{ g_j\big\}_{j\in\mathbb{Z}}\big\Vert_{L^p(\ell^q)}.
\end{equation} 
 The proof of \eqref{marshall} is elementary and standard, simply using the estimate
 $$\bigg| \La_j \Big(\sum_{k\in\bbz}{g_k}\Big)(x)\bigg|=\bigg| \La_j\Big(\sum_{k=j-h}^{j+h}{g_k}\Big)(x)\bigg|\lesssim_{r,h}\mathcal{M}_rg_j(x)$$
 for all $0<r<\infty$ and for some $h\in\bbn$,  depending on $C_0$, and the maximal inequality \eqref{hlmax}.
 We refer to  \cite[Theorem 3.6]{Ya1986} for details.

By using the equivalence in \eqref{hardydeflittle}, 
\begin{equation*}
\big\Vert T_{\sigma}^{(1)}(f_1,f_2,f_3)\big\Vert_{H^p(\bbrn)}\sim \bigg\Vert \bigg\{ \La_j \Big(\sum_{k\in\bbz}{ T_{\sigma_k}\big(\La_kf_1,\Ga_{k-10}f_2,\Ga_{k-10}f_3 \big)}\Big)\bigg\}_{j\in\mathbb{Z}}\bigg\Vert_{L^p(\ell^2)}.
\end{equation*}
We see that the Fourier transform of $T_{\sigma_k}\big(\La_kf_1,\Ga_{k-10}f_2,\Ga_{k-10}f_3 \big)$ is supported in $\big\{\xi\in\bbrn : 2^{k-2}\leq |\xi|\leq 2^{k+2} \big\}$
 and thus the estimate \eqref{marshall} yields that
 \begin{equation*}
\big\Vert T_{\sigma}^{(1)}(f_1,f_2,f_3)\big\Vert_{H^p(\bbrn)}\lesssim \Big\Vert \Big( \sum_{j\in\bbz}\big|    T_{\sigma_j}\big(\La_jf_1,\Ga_{j-10}f_2,\Ga_{j-10}f_3 \big)     \big|^2\Big)^{1/2}\Big\Vert_{L^p(\bbrn)}.
\end{equation*}
Then it is already proved in \cite[(3.14)]{Gr_Mi_Ng_Tom2017} that the preceding expression is dominated by the right-hand side of \eqref{mainreductionest} for $s>n/\min{\{p_1,p_2,p_3\}}+n/2$, where
 we remark that $\min{\{p_1,p_2,p_3\}}\le 1$.
This proves \eqref{mainreductionest} for $\mu=1$.

\subsection{Proof of \eqref{mainreductionest} for $\mu=2$}
Recall that
\begin{equation}\label{t22}
T_{\sigma}^{(2)}(f_1,f_2,f_3)=\sum_{j\in\bbz}T_{\sigma_j}\big(\La_jf_1,\La_jf_2,\Ga_jf_3\big)
\end{equation}
and observe that
$$T_{\sigma_j}\big(\La_jf_1,\La_jf_2,\Ga_jf_3\big)(x)=\sigma_j^{\vee}\ast_{3n}\big(\La_jf_1\otimes \La_jf_2\otimes \Ga_jf_3 \big)(x,x,x)$$
where $\ast_{3n}$ means the convolution on $\bbr^{3n}$.

It suffices to consider the case when $(1/p_1,1/p_2,1/p_3)$ belongs to $\RR_1$ or $\RR_3$, as the remaining case is symmetrical to the case $(1/p_1,1/p_2,1/p_3)\in \RR_1$, in view of \eqref{t22}.
We will mainly focus on the case $(1/p_1,1/p_2,1/p_3)\in \RR_1$, while simply providing a short description for the case $(1/p_1,1/p_2,1/p_3)\in\RR_3$ in the remark below as almost same arguments will be applied in that case.

Therefore, we now assume $0<p_1\le 1$ and $2<p_2,p_3<\infty$,  and in turn, suppose that $s>n/p_1+n/2$.
By using the atomic decomposition in \eqref{atomdecomp}, the function $f_1\in H^{p_1}(\bbrn)$ can be written as
$f_1=\sum_{k=1}^{\infty}\la_k a_k$ where $a_k$'s are $H^{p_1}$-atoms associated with cubes $Q_k$, and 
\begin{equation}\label{hardylambda}
\Big(\sum_{k=1}^{\infty}|\lambda_k|^{p_1}\Big)^{1/{p_1}}\lesssim 1.
\end{equation}
 As mentioned before, we may assume that $M$ is sufficiently large and $\int{x^{\ga}a_k(x)}dx=0$ holds for all multi-indices $|\ga|\le M$.

By the definition in \eqref{hardydef2}, we have
\begin{align*}
\big\Vert T_{\sigma}^{(2)}(f_1,f_2,f_3)\big\Vert_{H^p(\bbrn)}\sim \bigg\Vert \sup_{l\in\bbz} \Big| \sum_{k=0}^{\infty}\la_k\phi_l\ast \Big( \sum_{j\in\bbz} T_{\sigma_j}(\La_ja_k,\La_jf_2,\Ga_jf_3)  \Big)     \Big|\bigg\Vert_{L^p(\bbrn)}
\end{align*}
and thus we need to prove that
\begin{equation}\label{mainmaingoal}
\bigg\Vert \sup_{l\in\bbz} \Big| \sum_{k=0}^{\infty}\la_k\phi_l\ast \Big( \sum_{j\in\bbz} T_{\sigma_j}(\La_ja_k,\La_jf_2,\Ga_jf_3)  \Big)     \Big|\bigg\Vert_{L^p(\bbrn)}\lesssim 1.
\end{equation}
The left-hand side is less than the sum of
$$\II:=\bigg\Vert \sup_{l\in\bbz}\Big| \sum_{k=0}^{\infty}\la_k \chi_{Q_k^{***}} \phi_l\ast \Big( \sum_{j\in\bbz} T_{\sigma_j}(\La_ja_k,\La_jf_2,\Ga_jf_3)    \Big)   \Big|\bigg\Vert_{L^p(\bbrn)}$$
and
$$\JJ:=\bigg\Vert \sup_{l\in\bbz}\Big| \sum_{k=0}^{\infty}\la_k  \chi_{(Q_k^{***})^c}\phi_l\ast \Big( \sum_{j\in\bbz} T_{\sigma_j}(\La_ja_k,\La_jf_2,\Ga_jf_3)   \Big)   \Big|\bigg\Vert_{L^p(\bbrn)}$$
recalling that $Q_k^{***}$ is the dilate of $Q_k$ by a factor $(10\sqrt{n})^3$. 
The two terms $\II$ and $\JJ$ will be treated separately in the next two sections.

\begin{remark}
When $(1/p_1,1/p_2,1/p_3)\in \RR_3$ (that is, $0<p_3\le 1, ~2<p_1,p_2< \infty$), we need to prove
$$\bigg\Vert \sup_{l\in\bbz} \Big| \sum_{k=0}^{\infty}\la_k\phi_l\ast \Big( \sum_{j\in\bbz} T_{\sigma_j}(\La_jf_1,\La_jf_2,\Ga_j\wt{a_k})  \Big)     \Big|\bigg\Vert_{L^p(\bbrn)}\lesssim 1$$
where $\wt{a_k}$ is the $H^{p_3}$-atom associated with $f_3$.
This is actually, via symmetry, equivalent to the estimate that for $0<p_1\le 1$ and  $2<p_2,p_3< \infty$,
\begin{equation}\label{symmetricmain}
\bigg\Vert \sup_{l\in\bbz} \Big| \sum_{k=0}^{\infty}\la_k\phi_l\ast \Big( \sum_{j\in\bbz} T_{\sigma_j}(\Ga_ja_k,\La_jf_2,\La_jf_3)  \Big)     \Big|\bigg\Vert_{L^p(\bbrn)}\lesssim 1
\end{equation}
where $a_k$ is the $H^{p_1}$-atom for $f_1$.
The proof of \eqref{symmetricmain} is almost same as that of \eqref{mainmaingoal} which will be discussed in Sections \ref{estimateii} and \ref{estimatejj}.
So this will not be pursued in this paper, just saying that \eqref{estgaak} will be needed rather than \eqref{estlaak}, and
the estimate
$\big\Vert \big\{ \Ga_ja_k\big\}_{j\in\bbz}\big\Vert_{L^r(\ell^{\infty})}\sim \Vert a_k\Vert_{H^r(\bbrn)}$
will be required in place of the equivalence $\big\Vert \big\{ \La_ja_k\big\}_{j\in\bbz}\big\Vert_{L^r(\ell^{2})}\sim \Vert a_k\Vert_{H^r(\bbrn)}$.

\end{remark}

\section{Proof of Proposition \ref{mainproposition} : Estimate for $\II$}\label{estimateii}

For the estimation of $\II$, we need the following lemma whose proof will be given in Section \ref{proofoflemmas}.
\begin{lemma}\label{keylemma1}
Let $0<p_1\le 1$ and $2<p_2,p_3<\infty$ and suppose that $\Vert f_1\Vert_{H^{p_1}(\bbrn)}=\Vert f_2\Vert_{H^{p_2}(\bbrn)}=\Vert f_3\Vert_{H^{p_3}(\bbrn)}=1$ and $\LL_s^2[\sigma]=1$ for  $s>n/p_1+n/2$. 
Then there exist nonnegative functions $u_1$, $u_2$, and $u_3$ on $\bbrn$ such that
$$\Vert u_{\ii}\Vert_{L^{p_{\ii}}(\bbrn)}\lesssim 1 \q \text{ for }~ \ii=1,2,3,$$
and for $x\in\bbrn$
\begin{equation}\label{keylemma1est}
\sup_{l\in\bbz}\Big|\sum_{k=0}^{\infty}\la_k\chi_{Q_k^{***}}(x)\phi_l\ast \Big( \sum_{j\in\bbz} T_{\sigma_j}(\La_ja_k,\La_jf_2,\Ga_jf_3)    \Big)(x)\Big| \lesssim u_1(x) u_2(x) u_3(x).
\end{equation}
\end{lemma}
This lemma, together with H\"older's inequality, clearly shows that 
$$\II\lesssim \Vert u_1\Vert_{L^{p_1}(\bbrn)}\Vert u_2\Vert_{L^{p_2}(\bbrn)}\Vert u_3\Vert_{L^{p_3}(\bbrn)}\lesssim 1.$$

\section{Proof of Proposition \ref{mainproposition} : Estimate for $\JJ$}\label{estimatejj}
Recall that for each $Q_k$ and $l\in\bbz$, $B_{\xx_{Q_k}}^l=B(\xx_{Q_k},100n2^{-l})$ stands for the ball of radius $100n2^{-l}$ and center $\xx_{Q_k}$.
Simply writing $B_k^l:=B_{\xx_{Q_k}}^l$, we bound $\JJ$ by the sum of 
$$\JJ_1:=\bigg\Vert \sup_{l\in\bbz}\Big| \sum_{k=0}^{\infty}\la_k  \chi_{(Q_k^{***})^c}\chi_{(B_{k}^l)^c}\phi_l\ast \Big( \sum_{j\in\bbz} T_{\sigma_j}(\La_ja_k,\La_jf_2,\Ga_jf_3)   \Big)   \Big|\bigg\Vert_{L^p(\bbrn)}$$
and
$$\JJ_2:=\bigg\Vert \sup_{l\in\bbz}\Big| \sum_{k=0}^{\infty}\la_k  \chi_{(Q_k^{***})^c}\chi_{B_k^l}\phi_l\ast \Big( \sum_{j\in\bbz} T_{\sigma_j}(\La_ja_k,\La_jf_2,\Ga_jf_3)   \Big)   \Big|\bigg\Vert_{L^p(\bbrn)}$$
 and treat them separately.

\subsection{Estimate for $\JJ_1$}
Using the representations in \eqref{character0} and \eqref{characterc1}, we write
$$\La_jf_2(x):=\sum_{P\in\mathcal{D}_j}b_P^2\psi^P(x),\qq \Ga_jf_3(x):=\sum_{R\in\mathcal{D}_j}b_R^3\theta^R(x)$$
where we recall $\psi^P(x)=|P|^{{1}/{2}}\psi_j(x-\xx_P)$ and $\theta^R(x)=|R|^{{1}/{2}}\theta_j(x-\xx_R)$ for $P,R\in\mathcal{D}_j$.
Then it follows from \eqref{charactera}, \eqref{characterc2}, \eqref{hardydef}, and \eqref{hardydeflittle} that
\begin{equation}\label{bp2}
\big\Vert \{b_P^2\}_{P\in\mathcal{D}}\big\Vert_{\dot{f}^{p_2,2}}\sim \big\Vert \big\{\La_jf_2\big\}_{j\in\bbz}\big\Vert_{L^{p_2}(\ell^2)}\sim \Vert f_2\Vert_{H^{p_2}(\bbrn)}= 1
\end{equation}
and
\begin{equation}\label{br3}
\big\Vert \{b_R^3\}_{R\in\mathcal{D}}\big\Vert_{\dot{f}^{p_3,\infty}}\sim \big\Vert \big\{\Ga_jf_3\big\}_{j\in\bbz}\big\Vert_{L^{p_3}(\ell^{\infty})}\sim \Vert f_3\Vert_{H^{p_3}(\bbrn)}=1.
\end{equation}
We write
\begin{equation*}
 \phi_l\ast \Big( \sum_{j\in\bbz} T_{\sigma_j}(\La_ja_k,\La_jf_2,\Ga_jf_3)   \Big)(x)= \sum_{\nu=1}^{4}  \phi_l\ast \big( \UU_{\nu}(x,\cdot)\big)(x)
\end{equation*}
where
\begin{equation}\label{omeganudef}
\Omega_{\nu}(P,R):=\begin{cases}
P\cap R & \nu=1\\
P^c\cap R & \nu=2\\
P\cap R^c & \nu=3\\
P^c \cap R^c & \nu=4
\end{cases}
\end{equation}
and
\begin{equation}\label{unudef}
\UU_{\nu}(x,y):= \sum_{j\in\bbz} \sum_{P\in\mathcal{D}_j}\sum_{R\in\mathcal{D}_j} b_P^2 b_R^3  T_{\sigma_j}\big(\La_ja_k,\psi^P,\theta^R\big)(y)  \chi_{\Omega_{\nu}(P,R)}(x), \q \nu=1,2,3,4.
\end{equation}

Then we have 
$$\JJ_1 \lesssim_p \JJ_{1}^1+\JJ_{1}^2+\JJ_{1}^3+\JJ_{1}^4$$
 where
\begin{equation*}
\JJ_{1}^{\nu}:=\bigg\Vert \sup_{l\in\bbz}\; \Big| \sum_{k=0}^{\infty}\la_k\chi_{(Q_k^{***})^c}(x)\chi_{(B_k^l)^c}(x) \phi_l\ast\big( \UU_{\nu}(x,\cdot)\big)(x)\Big|\bigg\Vert_{L^p(x)},\qq \nu=1,2,3,4.
\end{equation*} 

Now we will show that 
\begin{equation}\label{jkeyest}
\JJ_1^{\nu}\lesssim 1, \qq \nu=1,2,3,4.
\end{equation}

\subsubsection{Proof of \eqref{jkeyest} for $\nu=1$}
We further decompose $\UU_1(x,y)$ as 
$$\UU_1(x,y)=\UU_1^{\mathrm{in}}(x,y)+\UU_1^{\mathrm{out}}(x,y)$$ 
where
\begin{equation}\label{u1inoutdef}
\begin{aligned}
\UU_1^{\mathrm{in}}(x,y)&:=\sum_{j\in\bbz} \sum_{P\in\mathcal{D}_j}\sum_{R\in\mathcal{D}_j} b_P^2 b_R^3  T_{\sigma_j}\big(\chi_{Q_k^*}\La_ja_k,\psi^P,\theta^R\big)(y)  \chi_{P\cap R}(x),\\
\UU_1^{\mathrm{out}}(x,y)&:=\sum_{j\in\bbz} \sum_{P\in\mathcal{D}_j}\sum_{R\in\mathcal{D}_j} b_P^2 b_R^3  T_{\sigma_j}\big(\chi_{(Q_k^*)^c}\La_ja_k,\psi^P,\theta^R\big)(y)  \chi_{P\cap R}(x),
\end{aligned}
\end{equation}
and accordingly, we define
 \begin{equation*}
 \JJ_1^{1,\mathrm{in}/\mathrm{out}}:=\bigg\Vert \sup_{l\in\bbz}\; \Big| \sum_{k=0}^{\infty}\la_k\chi_{(Q_k^{***})^c}(x)\chi_{(B_k^l)^c}(x) \phi_l\ast\big( \UU_{1}^{\mathrm{in}/\mathrm{out}}(x,\cdot)\big)(x)\Big|\bigg\Vert_{L^p(x)}.
 \end{equation*}

Then we claim the following lemma.
\begin{lemma}\label{keylemma2}
Let $0<p_1\le 1$ and $2<p_2,p_3<\infty$ and let $\UU_1^{\mathrm{in}/\mathrm{out}}$ be defined as in \eqref{u1inoutdef}.
Suppose that $\Vert f_1\Vert_{H^{p_1}(\bbrn)}=\Vert f_2\Vert_{H^{p_2}(\bbrn)}=\Vert f_3\Vert_{H^{p_3}(\bbrn)}=1$ and $\LL_s^2[\sigma]=1$ for $s>n/p_1+n/2$. 
Then there exist nonnegative functions $u_1^{\mathrm{in}}$, $u_1^{\mathrm{out}}$, $u_2$, and $u_3$ on $\bbrn$ such that
$$\big\Vert u_1^{\mathrm{in}/\mathrm{out}}\big\Vert_{L^{p_1}(\bbrn)}\lesssim 1, \qq \Vert u_{\ii}\Vert_{L^{p_{\ii}}(\bbrn)}\lesssim 1 \q \text{ for }~ \ii=2,3,$$
and for $x\in\bbrn$
\begin{equation}\label{keylemma2est}
\sup_{l\in\bbz}\; \Big| \sum_{k=0}^{\infty}\la_k\chi_{(Q_k^{***})^c}(x)\chi_{(B_k^l)^c}(x) \phi_l\ast\big( \UU_{1}^{\mathrm{in}/\mathrm{out}}(x,\cdot)\big)(x)\Big| \lesssim u_1^{\mathrm{in}/\mathrm{out}}(x) u_2(x) u_3(x).
\end{equation}
\end{lemma}

The proof of Lemma \ref{keylemma2}  will be given in Section \ref{proofoflemmas}.
Taking the lemma for granted and using H\"older's inequality, we can easily show that
$$\JJ_1^1\lesssim_p \JJ_1^{1,\mathrm{in}}+\JJ_1^{1,\mathrm{out}}\lesssim 1.$$

\subsubsection{Proof of \eqref{jkeyest} for $\nu=2$}
For $P\in\DD$ and $l\in\bbz$ let $B_P^l:=B_{\xx_P}^l=B(\xx_P,100n2^{-l})$.
By introducing
\begin{equation}\label{u2etainoutdef}
\begin{aligned}
\UU_2^{1,\mathrm{in}}(x,y)&:=\sum_{j\in\bbz} \sum_{P\in\mathcal{D}_j}\sum_{R\in\mathcal{D}_j} b_P^2 b_R^3  T_{\sigma_j}\big(\chi_{Q_k^*}\La_ja_k,\psi^P,\theta^R\big)(y)  \chi_{P^c\cap (B_P^l)^c\cap R}(x),\\
\UU_2^{1,\mathrm{out}}(x,y)&:=\sum_{j\in\bbz} \sum_{P\in\mathcal{D}_j}\sum_{R\in\mathcal{D}_j} b_P^2 b_R^3  T_{\sigma_j}\big(\chi_{(Q_k^*)^c}\La_ja_k,\psi^P,\theta^R\big)(y) \chi_{P^c\cap (B_P^l)^c\cap R}(x),\\
\UU_2^{2,\mathrm{in}}(x,y)&:=\sum_{j\in\bbz} \sum_{P\in\mathcal{D}_j}\sum_{R\in\mathcal{D}_j} b_P^2 b_R^3  T_{\sigma_j}\big(\chi_{Q_k^*}\La_ja_k,\psi^P,\theta^R\big)(y)  \chi_{P^c\cap B_P^l\cap R}(x),\\
\UU_2^{2,\mathrm{out}}(x,y)&:=\sum_{j\in\bbz} \sum_{P\in\mathcal{D}_j}\sum_{R\in\mathcal{D}_j} b_P^2 b_R^3  T_{\sigma_j}\big(\chi_{(Q_k^*)^c}\La_ja_k,\psi^P,\theta^R\big)(y)  \chi_{P^c\cap B_P^l\cap R}(x),
\end{aligned}
\end{equation}
we write $\UU_2=\UU_2^{1,\mathrm{in}}+\UU_2^{1,\mathrm{out}}+\UU_2^{2,\mathrm{in}}+\UU_2^{2,\mathrm{out}}$
and consequently,
$$\JJ_1^2 \lesssim_p \JJ_1^{2,1,\mathrm{in}}+\JJ_1^{2,1,\mathrm{out}}+\JJ_1^{2,2,\mathrm{in}}+\JJ_1^{2,2,\mathrm{out}}$$
where 
\begin{equation*}
\JJ_{1}^{2,\eta,\mathrm{in}/\mathrm{out}}:=\bigg\Vert \sup_{l\in\bbz}\; \Big|\sum_{k=0}^{\infty}\la_k \chi_{(Q_k^{***})^c}(x)\chi_{(B_k^l)^c}(x)\;  \phi_l\ast\big( \UU_{2}^{\eta,\mathrm{in}/\mathrm{out}}(x,\cdot)\big)(x) \Big| \bigg\Vert_{L^p(x)},\q \eta=1,2.
\end{equation*}

Then we apply the following lemma that will be proved in Section \ref{proofoflemmas}.
\begin{lemma}\label{keylemma3}
Let $0<p_1\le 1$ and $2<p_2,p_3<\infty$ and let $\UU_2^{\eta,\mathrm{in}/\mathrm{out}}$ be defined as in \eqref{u2etainoutdef}.
Suppose that $\Vert f_1\Vert_{H^{p_1}(\bbrn)}=\Vert f_2\Vert_{H^{p_2}(\bbrn)}=\Vert f_3\Vert_{H^{p_3}(\bbrn)}=1$ and $\LL_s^2[\sigma]=1$ for $s>n/p_1+n/2$. 
Then there exist nonnegative functions $u_1^{\mathrm{in}}$, $u_1^{\mathrm{out}}$, $u_2$, and $u_3$ on $\bbrn$ such that
$$\big\Vert u_1^{\mathrm{in}/\mathrm{out}}\big\Vert_{L^{p_1}(\bbrn)}\lesssim 1, \qq \Vert u_{\ii}\Vert_{L^{p_{\ii}}(\bbrn)}\lesssim 1 \q \text{ for }~ \ii=2,3,$$
and for each $\eta=1,2$ 
\begin{equation}\label{keylemma3est}
\sup_{l\in\bbz}\; \Big| \sum_{k=0}^{\infty}\la_k\chi_{(Q_k^{***})^c}(x)\chi_{(B_k^l)^c}(x) \phi_l\ast\big( \UU_{2}^{\eta,\mathrm{in}/\mathrm{out}}(x,\cdot)\big)(x)\Big| \lesssim u_1^{\mathrm{in}/\mathrm{out}}(x) u_2(x) u_3(x).
\end{equation}
\end{lemma}

Then Lemma \ref{keylemma3} and H\"older's inequality yield that $\JJ_1^2$ is controlled by the sum of four terms in the form
\begin{align*}
\big\Vert u_1^{\mathrm{in}/\mathrm{out}}\big\Vert_{L^{p_1}(\bbrn)}\Vert u_2\Vert_{L^{p_2}(\bbrn)} \Vert u_3\Vert_{L^{p_3}(\bbrn)},
\end{align*}
which is obviously less than a constant.
This proves \eqref{jkeyest} for $\nu=2$.

\subsubsection{Proof of \eqref{jkeyest} for $\nu=3$}
This case is essentially symmetrical to the case $\nu=2$.
For $R\in\DD$ and $l\in\bbz$ let $B_R^l:=B_{\xx_R}^l=B(\xx_R,100n2^{-l})$.
Let
\begin{equation}\label{u3etainoutdef}
\begin{aligned}
\UU_3^{1,\mathrm{in}}(x,y)&:=\sum_{j\in\bbz} \sum_{P\in\mathcal{D}_j}\sum_{R\in\mathcal{D}_j} b_P^2 b_R^3  T_{\sigma_j}\big(\chi_{Q_k^*}\La_ja_k,\psi^P,\theta^R\big)(y)  \chi_{P\cap R^c\cap (B_R^l)^c}(x),\\
\UU_3^{1,\mathrm{out}}(x,y)&:=\sum_{j\in\bbz} \sum_{P\in\mathcal{D}_j}\sum_{R\in\mathcal{D}_j} b_P^2 b_R^3  T_{\sigma_j}\big(\chi_{(Q_k^*)^c}\La_ja_k,\psi^P,\theta^R\big)(y)  \chi_{P\cap R^c\cap (B_R^l)^c}(x),\\
\UU_3^{2,\mathrm{in}}(x,y)&:=\sum_{j\in\bbz} \sum_{P\in\mathcal{D}_j}\sum_{R\in\mathcal{D}_j} b_P^2 b_R^3  T_{\sigma_j}\big(\chi_{Q_k^*}\La_ja_k,\psi^P,\theta^R\big)(y)  \chi_{P\cap R^c\cap B_R^l}(x),\\
\UU_3^{2,\mathrm{out}}(x,y)&:=\sum_{j\in\bbz} \sum_{P\in\mathcal{D}_j}\sum_{R\in\mathcal{D}_j} b_P^2 b_R^3 T_{\sigma_j}\big(\chi_{(Q_k^*)^c}\La_ja_k,\psi^P,\theta^R\big)(y)  \chi_{P\cap R^c\cap B_R^l}(x),
\end{aligned}
\end{equation}
and then we write 
$$\JJ_1^3\lesssim_p \JJ_1^{3,1,\mathrm{in}}+\JJ_1^{3,1,\mathrm{out}}+\JJ_1^{3,2,\mathrm{in}}+\JJ_1^{3,2,\mathrm{out}}$$
where
\begin{equation*}
\JJ_{1}^{3,\eta,\mathrm{in}/\mathrm{out}}:=\bigg\Vert \sup_{l\in\bbz}\; \Big| \sum_{k=0}^{\infty} \la_k \chi_{(Q_k^{***})^c}(x)\chi_{(B_k^l)^c}(x)\; \phi_l\ast\big( \UU_{3}^{\eta,\mathrm{in}/\mathrm{out}}(x,\cdot)\big)(x) \Big|\bigg\Vert_{L^p(x)}, \q\eta=1,2.
\end{equation*}
Now \eqref{jkeyest} for $\nu=3$ follows from the lemma below.
\begin{lemma}\label{keylemma4}
Let $0<p_1\le 1$ and $2<p_2,p_3<\infty$ and let $\UU_3^{\eta,\mathrm{in}/\mathrm{out}}$ be defined as in \eqref{u3etainoutdef}.
Suppose that $\Vert f_1\Vert_{H^{p_1}(\bbrn)}=\Vert f_2\Vert_{H^{p_2}(\bbrn)}=\Vert f_3\Vert_{H^{p_3}(\bbrn)}=1$ and $\LL_s^2[\sigma]=1$ for $s>n/p_1+n/2$. 
Then there exist nonnegative functions $u_1^{\mathrm{in}}$, $u_1^{\mathrm{out}}$, $u_2$, and $u_3$ on $\bbrn$ such that
$$\big\Vert u_1^{\mathrm{in}/\mathrm{out}}\big\Vert_{L^{p_1}(\bbrn)}\lesssim 1, \qq \Vert u_{\ii}\Vert_{L^{p_{\ii}}(\bbrn)}\lesssim 1 \q \text{ for }~ \ii=2,3,$$
and for each $\eta=1,2$
\begin{equation*}
\sup_{l\in\bbz}\; \Big| \sum_{k=0}^{\infty}\la_k\chi_{(Q_k^{***})^c}(x)\chi_{(B_k^l)^c}(x) \phi_l\ast\big( \UU_{3}^{\eta,\mathrm{in}/\mathrm{out}}(x,\cdot)\big)(x)\Big| \lesssim u_1^{\mathrm{in}/\mathrm{out}}(x) u_2(x) u_3(x).
\end{equation*}
\end{lemma}
The proof of the lemma will be provided in Section \ref{proofoflemmas}.

\subsubsection{Proof of \eqref{jkeyest} for $\nu=4$}
In this case, we divide $\UU_4$ into eight types depending on whether $x$ belongs to each of $B_P^l$ and $B_R^l$ and whether $\La_ja_k$ is supported in $Q_k^*$.
Indeed,  let
\begin{equation}\label{xietadef}
\Xi_{\eta}(P,R,l):=\begin{cases}
P^c\cap R^c\cap (B_P^l)^c\cap (B_R^l)^c, & \eta=1\\
P^c\cap R^c\cap (B_P^l)^c\cap B_R^l, & \eta=2\\
P^c\cap R^c\cap B_P^l\cap (B_R^l)^c, & \eta=3\\
P^c\cap R^c\cap B_P^l\cap B_R^l, & \eta=4
\end{cases}
\end{equation}
and we define 
\begin{equation}\label{u4etainoutdef}
\begin{aligned}
\UU_4^{\eta,\mathrm{in}}(x,y)&:=\sum_{j\in\bbz} \sum_{P\in\mathcal{D}_j}\sum_{R\in\mathcal{D}_j} b_P^2 b_R^3  T_{\sigma_j}\big(\chi_{Q_k^*}\La_ja_k,\psi^P,\theta^R\big)(y)  \chi_{\Xi_{\eta}}(x),\\
\UU_4^{\eta,\mathrm{out}}(x,y)&:=\sum_{j\in\bbz} \sum_{P\in\mathcal{D}_j}\sum_{R\in\mathcal{D}_j} b_P^2 b_R^3  T_{\sigma_j}\big(\chi_{(Q_k^*)^c}\La_ja_k,\psi^P,\theta^R\big)(y)  \chi_{\Xi_{\eta}}(x)
\end{aligned}
\end{equation}
for $\eta=1,2,3,4$.

Then we use the following lemma to obtain the desired result.
\begin{lemma}\label{keylemma5}
Let $0<p_1\le 1$ and $2<p_2,p_3<\infty$ and let $\UU_4^{\eta,\mathrm{in}/\mathrm{out}}$ be defined as in \eqref{u4etainoutdef}.
Suppose that $\Vert f_1\Vert_{H^{p_1}(\bbrn)}=\Vert f_2\Vert_{H^{p_2}(\bbrn)}=\Vert f_3\Vert_{H^{p_3}(\bbrn)}=1$ and $\LL_s^2[\sigma]=1$ for $s>n/p_1+n/2$. 
Then there exist nonnegative functions $u_1^{\mathrm{in}}$, $u_1^{\mathrm{out}}$, $u_2$, and $u_3$ on $\bbrn$ such that
\begin{equation}\label{keylemma5vconditions}
\big\Vert u_1^{\mathrm{in}/\mathrm{out}}\big\Vert_{L^{p_1}(\bbrn)}\lesssim 1, \qq \Vert u_{\ii}\Vert_{L^{p_{\ii}}(\bbrn)}\lesssim 1 \q \text{ for }~ \ii=2,3,
\end{equation}
and for each $\eta=1,2,3,4,$
\begin{equation}\label{keylemma5est}
\sup_{l\in\bbz}\; \Big| \sum_{k=0}^{\infty}\la_k\chi_{(Q_k^{***})^c}(x)\chi_{(B_k^l)^c}(x) \phi_l\ast\big( \UU_{4}^{\eta,\mathrm{in}/\mathrm{out}}(x,\cdot)\big)(x)\Big| \lesssim u_1^{\mathrm{in}/\mathrm{out}}(x) u_2(x) u_3(x).
\end{equation}
\end{lemma}
We will prove the lemma in Section \ref{proofoflemmas}.

\subsection{Estimate for $\JJ_2$}

Let $x\in (Q_k^{***})^c\cap B_{k,l}$. For $\nu=1,2,3,4$, let $\Omega_{\nu}(P,R)$ be defined as in \eqref{omeganudef}.
Then as in the proof of the estimate for $\JJ_1$, we consider the four cases:  $x \in \Omega_{1}(P,R)$, $x \in \Omega_{2}(P,R)$, $x \in \Omega_{3}(P,R)$, and $x \in \Omega_{4}(P,R)$.
That is, for each $\nu=1,2,3,4$, let $\UU_{\nu}$ be defined as in \eqref{unudef} and 
$$\JJ_2^{\nu}:=\bigg\Vert   \sup_{l\in\bbz} \Big| \sum_{k=0}^{\infty} \la_k \chi_{(Q_k^{***})^c}(x)\chi_{B_k^l}(x) \phi_l\ast \big(\UU_{\nu}(x,\cdot) \big)(x)\Big|  \bigg\Vert_{L^p(x)}.$$
Then it suffices to show that for each $\nu=1,2,3,4$,
\begin{equation}\label{j2keyest}
\JJ_2^{\nu}\lesssim 1.
\end{equation}

\subsubsection{Proof of \eqref{j2keyest} for $\nu=1$}
In this case, the proof can be simply reduced to the following lemma, which will be proved in Section \ref{proofoflemmas}.
\begin{lemma}\label{keylemma6}
Let $0<p_1\le 1$ and $2<p_2,p_3<\infty$ and let $\UU_1$ be defined as in \eqref{unudef}.
Suppose that $\Vert f_1\Vert_{H^{p_1}(\bbrn)}=\Vert f_2\Vert_{H^{p_2}(\bbrn)}=\Vert f_3\Vert_{H^{p_3}(\bbrn)}=1$ and $\LL_s^2[\sigma]=1$ for $s>n/p_1+n/2$. 
Then there exist nonnegative functions $u_1$, $u_2$, and $u_3$ on $\bbrn$ such that
$$\Vert u_{\ii}\Vert_{L^{p_{\ii}}(\bbrn)}\lesssim 1 \q \text{ for }~ \ii=1,2,3,$$
and for $x\in\bbrn$
\begin{equation}\label{keylemma6est}
\sup_{l\in\bbz}\; \Big| \sum_{k=0}^{\infty}\la_k\chi_{(Q_k^{***})^c}(x)\chi_{B_k^l}(x) \phi_l\ast\big( \UU_{1}(x,\cdot)\big)(x)\Big| \lesssim u_1(x) u_2(x) u_3(x).
\end{equation}
\end{lemma}
Then it follows from H\"older's inequality that
$$\JJ_2^1\lesssim \Vert u_1\Vert_{L^{p_1}(\bbrn)}\Vert u_2\Vert_{L^{p_2}(\bbrn)}\Vert u_3\Vert_{L^{p_3}(\bbrn)}\lesssim 1.$$

\hfill

\subsubsection{Proof of \eqref{j2keyest} for $\nu=2$}
For $P\in\DD$ and $l\in\bbz$ let $B_P^l:=B_{\xx_P}^l$ be the ball of center $\xx_P$ and radius $100n2^{-l}$ as before. 
We define
\begin{equation}\label{u212def}
\begin{aligned}
\UU_2^{1}(x,y)&:=\sum_{j\in\bbz} \sum_{P\in\mathcal{D}_j}\sum_{R\in\mathcal{D}_j} b_P^2 b_R^3  T_{\sigma_j}\big(\La_ja_k,\psi^P,\theta^R\big)(y)  \chi_{P^c\cap (B_{P}^l)^c\cap R}(x),\\
\UU_2^{2}(x,y)&:=\sum_{j\in\bbz} \sum_{P\in\mathcal{D}_j}\sum_{R\in\mathcal{D}_j} b_P^2 b_R^3  T_{\sigma_j}\big(\La_ja_k,\psi^P,\theta^R\big)(y)  \chi_{P^c\cap B_P^l\cap R}(x),
\end{aligned}
\end{equation}
and write
$$\JJ_2^2\lesssim\JJ_2^{2,1}+\JJ_2^{2,2}$$
where
$$\JJ_2^{2,\eta}:=\bigg\Vert   \sup_{l\in\bbz} \Big| \sum_{k=0}^{\infty} \la_k \chi_{(Q_k^{***})^c}(x)\chi_{B_k^l}(x) \phi_l\ast \big(\UU_{2}^{\eta}(x,\cdot) \big)(x)\Big|\;   \bigg\Vert_{L^p(x)}, \q \eta=1,2.$$
Then we need the following lemmas.
\begin{lemma}\label{keylemma7}
Let $0<p_1\le 1$ and $2<p_2,p_3<\infty$ and let $\UU_2^{1}$ be defined as in \eqref{u212def}.
Suppose that $\Vert f_1\Vert_{H^{p_1}(\bbrn)}=\Vert f_2\Vert_{H^{p_2}(\bbrn)}=\Vert f_3\Vert_{H^{p_3}(\bbrn)}=1$ and $\LL_s^2[\sigma]=1$ for $s>n/p_1+n/2$. 
Then there exist nonnegative functions $u_1$, $u_2$, and $u_3$ on $\bbrn$ such that
\begin{equation*}
\Vert u_{\ii}\Vert_{L^{p_{\ii}}(\bbrn)}\lesssim 1 \q \text{ for }~ \ii=1,2,3,
\end{equation*}
and 
\begin{equation}\label{keylemma7est}
\sup_{l\in\bbz}\; \Big| \sum_{k=0}^{\infty}\la_k\chi_{(Q_k^{***})^c}(x)\chi_{B_k^l}(x) \phi_l\ast\big( \UU_{2}^{1}(x,\cdot)\big)(x)\Big| \lesssim u_1(x) u_2(x) u_3(x).
\end{equation}
\end{lemma}

\begin{lemma}\label{keylemma8}
Let $0<p_1\le 1$ and $2<p_2,p_3<\infty$ and let $\UU_2^{2}$ be defined as in \eqref{u212def}.
Suppose that $\Vert f_1\Vert_{H^{p_1}(\bbrn)}=\Vert f_2\Vert_{H^{p_2}(\bbrn)}=\Vert f_3\Vert_{H^{p_3}(\bbrn)}=1$ and $\LL_s^2[\sigma]=1$ for $s>n/p_1+n/2$. 
Then there exist nonnegative functions $u_1$, $u_2$, and $u_3$ on $\bbrn$ such that
\begin{equation*}
 \Vert u_{\ii}\Vert_{L^{p_{\ii}}(\bbrn)}\lesssim 1 \q \text{ for }~ \ii=1,2,3,
\end{equation*}
and 
\begin{equation*}
\sup_{l\in\bbz}\; \Big| \sum_{k=0}^{\infty}\la_k\chi_{(Q_k^{***})^c}(x)\chi_{B_k^l}(x) \phi_l\ast\big( \UU_{2}^{2}(x,\cdot)\big)(x)\Big| \lesssim u_1(x) u_2(x) u_3(x).
\end{equation*}
\end{lemma}
The above lemmas will be proved in Section \ref{proofoflemmas}.
Using Lemmas \ref{keylemma7} and \ref{keylemma8}, we obtain
\begin{equation*}
\JJ_2^{2,\eta}\lesssim \Vert u_1\Vert_{L^{p_1}(\bbrn)}\Vert u_2\Vert_{L^{p_2}(\bbrn)}\Vert u_3\Vert_{L^{p_3}(\bbrn)}\lesssim 1,\qq \eta=1,2,
\end{equation*}
which finishes the proof of \eqref{j2keyest} for $\nu=2$.\\

\subsubsection{Proof of \eqref{j2keyest} for $\nu=3$}
We use the notation $B_R^l:=B_{\xx_R}^l$ for $R\in \DD$ and $l\in\bbz$ as before, and write
$$\JJ_2^3\lesssim \JJ_2^{3,1}+\JJ_2^{3,2}$$
where
\begin{equation}\label{u312def}
\begin{aligned}
\UU_3^{1}(x,y)&:=\sum_{j\in\bbz} \sum_{P\in\mathcal{D}_j}\sum_{R\in\mathcal{D}_j} b_P^2 b_R^3  T_{\sigma_j}\big(\La_ja_k,\psi^P,\theta^R\big)(y)  \chi_{P\cap R^c\cap (B_R^l)^c}(x),\\
\UU_3^{2}(x,y)&:=\sum_{j\in\bbz} \sum_{P\in\mathcal{D}_j}\sum_{R\in\mathcal{D}_j} b_P^2 b_R^3  T_{\sigma_j}\big(\La_ja_k,\psi^P,\theta^R\big)(y)  \chi_{P\cap R^c\cap B_R^l}(x),
\end{aligned}
\end{equation}
and 
$$\JJ_2^{3,\eta}:=\bigg\Vert   \sup_{l\in\bbz} \Big| \sum_{k=0}^{\infty} \la_k \chi_{(Q_k^{***})^c}(x)\chi_{B_k^l}(x)  \phi_l\ast \big(\UU_{3}^{\eta}(x,\cdot) \big)(x)\Big| \bigg\Vert_{L^p(x)}, \q \eta=1,2.$$

As in the proof of the case $\nu=2$, it suffices to prove the following two lemmas.
\begin{lemma}\label{keylemma9}
Let $0<p_1\le 1$ and $2<p_2,p_3<\infty$ and let $\UU_3^{1}$ be defined as in \eqref{u312def}.
Suppose that $\Vert f_1\Vert_{H^{p_1}(\bbrn)}=\Vert f_2\Vert_{H^{p_2}(\bbrn)}=\Vert f_3\Vert_{H^{p_3}(\bbrn)}=1$ and $\LL_s^2[\sigma]=1$ for  $s>n/p_1+n/2$. 
Then there exist nonnegative functions $u_1$, $u_2$, and $u_3$ on $\bbrn$ such that
\begin{equation}\label{keylemma9conditions}
 \Vert u_{\ii}\Vert_{L^{p_{\ii}}(\bbrn)}\lesssim 1 \q \text{ for }~ \ii=1,2,3,
\end{equation}
and 
\begin{equation}\label{keylemma9est}
\sup_{l\in\bbz}\; \Big| \sum_{k=0}^{\infty}\la_k\chi_{(Q_k^{***})^c}(x)\chi_{B_k^l}(x) \phi_l\ast\big( \UU_{3}^{1}(x,\cdot)\big)(x)\Big| \lesssim u_1(x) u_2(x) u_3(x).
\end{equation}
\end{lemma}

\begin{lemma}\label{keylemma10}
Let $0<p_1\le 1$ and $2<p_2,p_3<\infty$ and let $\UU_3^{2}$ be defined as in \eqref{u312def}.
Suppose that $\Vert f_1\Vert_{H^{p_1}(\bbrn)}=\Vert f_2\Vert_{H^{p_2}(\bbrn)}=\Vert f_3\Vert_{H^{p_3}(\bbrn)}=1$ and $\LL_s^2[\sigma]=1$ for  $s>n/p_1+n/2$. 
Then there exist nonnegative functions $u_1$, $u_2$, and $u_3$ on $\bbrn$ such that
\begin{equation}\label{keylemma10conditions}
 \Vert u_{\ii}\Vert_{L^{p_{\ii}}(\bbrn)}\lesssim 1 \q \text{ for }~ \ii=1,2,3,
\end{equation}
and 
\begin{equation}\label{keylemma10est}
\sup_{l\in\bbz}\; \Big| \sum_{k=0}^{\infty}\la_k\chi_{(Q_k^{***})^c}(x)\chi_{B_k^l}(x) \phi_l\ast\big( \UU_{3}^{2}(x,\cdot)\big)(x)\Big| \lesssim u_1(x) u_2(x) u_3(x).
\end{equation}
\end{lemma}
The proof of Lemmas \ref{keylemma9} and \ref{keylemma10} will be provided in Section \ref{proofoflemmas}.

\hfill

\subsubsection{Proof of \eqref{j2keyest} for $\nu=4$}

Let $B_P^l:=B_{\xx^P}^l$ and $B_R^l:=B_{\xx_R}^l$ for $P,R\in \DD$ and $l\in\bbz$, and let $\Xi_{\eta}(P,R,l)$ be defined as in \eqref{xietadef}.
Now we write
$$\UU_4=\UU_4^{1}+\UU_4^{2}+\UU_4^{3}+\UU_4^{4}$$
where
\begin{equation}\label{u4etadef}
\UU_4^{\eta}(x,y):=\sum_{j\in\bbz} \sum_{P\in\mathcal{D}_j}\sum_{R\in\mathcal{D}_j} b_P^2 b_R^3  T_{\sigma_j}\big(\La_ja_k,\psi^P,\theta^R\big)(y)  \chi_{\Xi_{\eta}(P,R,l)}(x), \q \eta=1,2,3,4.
\end{equation}
Accordingly, we define  
$$\JJ_2^{4,\eta}:=\bigg\Vert   \sup_{l\in\bbz} \Big| \sum_{k=0}^{\infty}\la_k \chi_{(Q_k^{***})^c}(x)\chi_{B_k^l}(x)  \phi_l\ast \big(\UU_{4}^{\eta}(x,\cdot) \big)(x)\Big|\;  \bigg\Vert_{L^p(x)}, \q \eta=1,2,3,4.$$
Then we obtain the desired result from the following lemmas. 
\begin{lemma}\label{keylemma11}
Let $0<p_1\le 1$ and $2<p_2,p_3<\infty$ and let $\UU_4^{\eta}$, $\eta=1,2,3$, be defined as in \eqref{u4etadef}.
Suppose that $\Vert f_1\Vert_{H^{p_1}(\bbrn)}=\Vert f_2\Vert_{H^{p_2}(\bbrn)}=\Vert f_3\Vert_{H^{p_3}(\bbrn)}=1$ and $\LL_s^2[\sigma]=1$ for $s>n/p_1+n/2$. 
Then there exist nonnegative functions $u_1$, $u_2$, and $u_3$ on $\bbrn$ such that
\begin{equation*}
\Vert u_{\ii}\Vert_{L^{p_{\ii}}(\bbrn)}\lesssim 1 \q \text{ for }~ \ii=1,2,3,
\end{equation*}
and for each $\eta=1,2,3$
\begin{equation}\label{keylemma11est}
\sup_{l\in\bbz}\; \Big| \sum_{k=0}^{\infty}\la_k\chi_{(Q_k^{***})^c}(x)\chi_{B_k^l}(x) \phi_l\ast\big( \UU_{4}^{\eta}(x,\cdot)\big)(x)\Big| \lesssim u_1(x) u_2(x) u_3(x).
\end{equation}
\end{lemma}

\begin{lemma}\label{keylemma12}
Let $0<p_1\le 1$ and $2<p_2,p_3<\infty$ and let $\UU_4^{4}$ be defined as in \eqref{u4etadef}.
Suppose that $\Vert f_1\Vert_{H^{p_1}(\bbrn)}=\Vert f_2\Vert_{H^{p_2}(\bbrn)}=\Vert f_3\Vert_{H^{p_3}(\bbrn)}=1$ and $\LL_s^2[\sigma]=1$ for $s>n/p_1+n/2$. 
Then there exist nonnegative functions $u_1$, $u_2$, and $u_3$ on $\bbrn$ such that
\begin{equation}\label{keylemma12conditions}
 \Vert u_{\ii}\Vert_{L^{p_{\ii}}(\bbrn)}\lesssim 1 \q \text{ for }~ \ii=1,2,3,
\end{equation}
and 
\begin{equation}\label{keylemma12est}
\sup_{l\in\bbz}\; \Big| \sum_{k=0}^{\infty}\la_k\chi_{(Q_k^{***})^c}(x)\chi_{B_k^l}(x) \phi_l\ast\big( \UU_{4}^{4}(x,\cdot)\big)(x)\Big| \lesssim u_1(x) u_2(x) u_3(x).
\end{equation}
\end{lemma}
The proof of the lemmas will be given in Section \ref{proofoflemmas}.

\section{Proof of Proposition \ref{mainproposition2}}\label{pfpropositionp=1}


 We need to deal only with, via symmetry, the case when $0<p_1=p\le 1$ and $p_2=p_3=\infty$. 
As before, we assume that $\| f_1 \|_{H^p(\bbrn)} = \|f_2 \|_{L^\infty(\bbrn)} = \| f_3 \|_{L^\infty(\bbrn)}=1$ and $\LL_s^2[\sigma]=1$ for $s>n/p+n/2$.
In this case, we do not decompose the frequencies of $f_2, f_3$ and only make use of the atomic decomposition on $f_1$. 
Let $a_k$'s be $H^p$-atoms associated with $Q_k$ so that  $f_1=\sum_{k=1}^{\infty}\la_k a_k$ and $\big(\sum_{k=1}^{\infty}|\la_k|^p\big)^{1/p}\lesssim 1$.
Then we will prove that
\begin{equation}\label{qkinside}
\bigg\| \sup_{l\in\mathbb{Z}} \Big|\sum_{k=1}^{\infty} \lambda_k \chi_{Q_k^{***}}\;\phi_l \ast T_{\sigma}(a_k, f_2, f_3)\Big| \bigg\|_{L^p(\bbrn)} \lesssim 1
\end{equation}
and
\begin{equation}\label{qkout}
\bigg\| \sup_{l\in\mathbb{Z}} \Big|\sum_{k=1}^{\infty} \lambda_k \chi_{(Q_k^{***})^c}\;\phi_l \ast T_{\sigma}(a_k, f_2, f_3)\Big| \bigg\|_{L^p(\bbrn)} \lesssim 1
\end{equation}

\subsection{Proof of \eqref{qkinside}}
Since $$\big| \phi_l\ast T_{\sigma}(a_k,f_2,f_3)(x)\big|\lesssim \mathcal{M}T_{\sigma}(a_k,f_2,f_3)(x),$$
the left-hand side of \eqref{qkinside} is controlled by
\begin{equation*}
\Big(\sum_{k=1}^{\infty}|\la_k|^p\big\Vert \mathcal{M}T_{\sigma}(a_k,f_2,f_3) \big\Vert_{L^p(Q^{***})}^p \Big)^{1/p}.
\end{equation*}
Using H\"older's inequality, the $L^2$ boundedness of $\mathcal{M}$, and Theorem \ref{thmd}, we have
\begin{align*}
\big\Vert \mathcal{M}T_{\sigma}(a_k,f_2,f_3) \big\Vert_{L^p(Q^{***})}\lesssim |Q_k|^{1/p-1/2}\big\Vert  T_{\sigma}(a_k,f_2,f_3)\big\Vert_{L^2(\bbrn)}\lesssim |Q_k|^{1/p-1/2}\Vert a_k\Vert_{L^2(\bbrn)}\lesssim 1
\end{align*}
and thus \eqref{qkinside} follows from  $\big(\sum_{k=1}^{\infty}|\la_k|^p\big)^{1/p}\lesssim 1$.

\subsection{Proof of \eqref{qkout}}

Let $B_k^l=B(\xx_{Q_k},100 n2^{-l})$ as before. We now decompose the left-hand side of \eqref{qkout} as the sum of 
$$\VV_1:= \bigg\| \sup_{l\in\mathbb{Z}} \Big|\sum_{k=1}^{\infty} \lambda_k \chi_{(Q_k^{***})^c}\chi_{(B_k^l)^c}\;\phi_l \ast T_{\sigma}(a_k,f_2,f_3)\Big| \bigg\|_{L^p(\bbrn)},$$
$$\VV_2:= \bigg\| \sup_{l\in\mathbb{Z}} \Big|\sum_{k=1}^{\infty} \lambda_k \chi_{(Q_k^{***})^c}\chi_{B_k^l}\;\phi_l \ast T_{\sigma}(a_k,f_2,f_3)\Big)\Big| \bigg\|_{L^p(\bbrn)},$$
and thus we need to show that 
\begin{equation*}
\VV_{1},\VV_2\lesssim 1.
\end{equation*}
Actually, the proof of these estimates will be complete once we have verified the following lemmas.
\begin{lemma}\label{keylemma13}
Let $0<p\le 1$.
Suppose that $\Vert f_1\Vert_{H^{p}(\bbrn)}=\Vert f_2\Vert_{H^{\infty}(\bbrn)}=\Vert f_3\Vert_{H^{\infty}(\bbrn)}=1$ and $\LL_s^2[\sigma]=1$ for $s>n/p+n/2$. 
Then there exist nonnegative functions $u_1$, $u_2$, and $u_3$ on $\bbrn$ such that
\begin{equation*}
 \Vert u_{1}\Vert_{L^{p}(\bbrn)}\lesssim 1, \qq \Vert u_{\ii}\Vert_{L^{\infty}(\bbrn)}\lesssim 1\q \text{ for }~ \ii=2,3,
\end{equation*}
and 
\begin{equation}\label{keylemma13est}
\sup_{l\in\mathbb{Z}} \Big|\sum_{k=1}^{\infty} \lambda_k \chi_{(Q_k^{***})^c}(x)\chi_{(B_k^l)^c}(x)\;\phi_l \ast T_{\sigma}\big(a_k,f_2,f_3\big)(x)\Big|\lesssim u_1(x) u_2(x) u_3(x).
\end{equation}
\end{lemma}

\begin{lemma}\label{keylemma14}
Let $0<p\le 1$. Suppose that $\Vert f_1\Vert_{H^{p}(\bbrn)}=\Vert f_2\Vert_{H^{\infty}(\bbrn)}=\Vert f_3\Vert_{H^{\infty}(\bbrn)}=1$ and $\LL_s^2[\sigma]=1$ for $s>n/p+n/2$. 
Then there exist nonnegative functions $u_1$, $u_2$, and $u_3$ on $\bbrn$ such that
\begin{equation*}
 \Vert u_{1}\Vert_{L^{p}(\bbrn)}\lesssim 1, \qq \Vert u_{\ii}\Vert_{L^{\infty}(\bbrn)}\lesssim 1\q \text{ for }~ \ii=2,3,
\end{equation*}
and 
\begin{equation}\label{keylemma14est}
\sup_{l\in\mathbb{Z}} \Big|\sum_{k=1}^{\infty} \lambda_k \chi_{(Q_k^{***})^c}(x)\chi_{B_k^l}(x)\;\phi_l \ast T_{\sigma}\big(a_k,f_2,f_3\big)(x)\Big|\lesssim u_1(x) u_2(x) u_3(x).
\end{equation}
\end{lemma}
The proof of the two lemmas will be given in Section \ref{proofoflemmas}.

\section{Proof of the key lemmas}\label{proofoflemmas}
\subsection{Proof of Lemma \ref{keylemma1}}

Let $1<r<2$ such that $s>{3n}/{r}>{3n}/{2}$ and we
claim the pointwise estimate 
\begin{equation}\label{tsijpt}
\big|T_{\sigma_j}(\La_ja_k,\La_jf_2,\Ga_jf_3)(y)\big|\lesssim \mathcal{M}_r\La_ja_k(y)\mathcal{M}_r\La_jf_2(y)\mathcal{M}_r\Ga_jf_3(y).
\end{equation}
Indeed, choosing $t$ so that ${3n}/{r}<3t<s$, we apply H\"older's inequality to bound the left-hand side of \eqref{tsijpt} by
\begin{align*}
 &\int_{(\bbrn)^3} \langle 2^j\zzz\rangle^{3t} \big| \sigma_j^{\vee}(\zzz)\big| \frac{|\La_ja_k(y-z_1)|}{\langle 2^jz_1\rangle^t} \frac{|\La_jf_2(y-z_2)|}{\langle 2^jz_2\rangle^t} \frac{|\Ga_jf_3(y-z_3)|}{\langle 2^jz_3\rangle^t} \; d\zzz\\
&\le \big\Vert \langle 2^j\ccdot\rangle^{3t} \sigma_j^{\vee}\big\Vert_{L^{r'}((\bbrn)^3)} \bigg\Vert  \frac{\La_ja_k(y-\cdot)}{\langle 2^j\cdot \rangle^t}\bigg\Vert_{L^{r}(\bbrn)}   \bigg\Vert  \frac{\La_jf_2(y-\cdot)}{\langle 2^j\cdot \rangle^t}\bigg\Vert_{L^{r}(\bbrn)}\bigg\Vert  \frac{\Ga_jf_3(y-\cdot)}{\langle 2^j\cdot \rangle^t}\bigg\Vert_{L^{r}(\bbrn)}.
\end{align*}
We observe that 
$$\big\Vert \langle 2^j\ccdot\rangle^{3t} \sigma_j^{\vee}\big\Vert_{L^{r'}((\bbrn)^3)}\lesssim 2^{{3jn}/{r}}\Vert \sigma(2^j\ccdot) \Vert_{L^r_{3t}((\bbrn)^3)}\lesssim 2^{{3jn}/{r}}\Vert \sigma(2^j\ccdot)\Vert_{L^2_s((\bbrn)^3)}\lesssim 2^{{3jn}/{r}}$$
using the Hausdorff-Young inequality, \eqref{supequiv}, and the inclusion 
$$L_{s_0}^{t_0}(A)\hookrightarrow L_{s_1}^{t_1}(A)\q \text{ for } ~s_0\ge s_1,~ t_0\ge t_1$$ where $A$ is a ball of a constant radius, whose proof is contained in \cite[(1.8)]{Gr_Park_IMRN}. 
Applying \eqref{maximalbound} to the remaining three $L^r$ norms, we finally obtain \eqref{tsijpt}.

Now we choose $\wt{r}$ and $q$ such that $2<\wt{r}<p_2,p_3$ and ${1}/{q}+{2}/{\wt{r}}=1$.
Finally, using the estimate \eqref{tsijpt} and H\"older's inequality, we have
\begin{align*}
&\Big|\sum_{k=0}^{\infty}\la_k\chi_{Q_k^{***}}(x)\phi_l\ast \Big( \sum_{j\in\bbz} T_{\sigma_j}(\La_ja_k,\La_jf_2,\Ga_jf_3)    \Big)(x)\Big|\\
&\lesssim \sum_{k=0}^{\infty}\la_k\chi_{Q_k^{***}}(x)      2^{ln}\int_{|x-y|\le  2^{-l}   } \big\Vert \big\{ \mathcal{M}_r\La_ja_k(y)\big\}_{j\in\bbz}\big\Vert_{\ell^2}\big\Vert \big\{ \mathcal{M}_r\La_jf_2(y)\big\}_{j\in\bbz}\big\Vert_{\ell^2} \\
&\qq\qq\qq\qq\qq\qq\qq\qq\qq\qq\qq\times  \big\Vert \big\{ \mathcal{M}_r\Ga_jf_3(y)\big\}_{j\in\bbz}\big\Vert_{\ell^{\infty}}\; dy\\
&\lesssim u_1(x) u_2(x)u_3(x)
\end{align*}
where we choose
\begin{align*}
u_1(x)&:= \sum_{k=0}^{\infty}\la_k\chi_{Q_k^{***}}(x) \mathcal{M}_{q}\big( \big\Vert \big\{ \mathcal{M}_r\La_ja_k\big\}_{j\in\bbz}\big\Vert_{\ell^2}\big)(x),\\
u_2(x)&:=\mathcal{M}_{\wt{r}}\big( \big\Vert \big\{ \mathcal{M}_r\La_jf_2\big\}_{j\in\bbz}\big\Vert_{\ell^2}\big)(x),\\
u_3(x)&:=\mathcal{M}_{\wt{r}}\big( \big\Vert \big\{ \mathcal{M}_r\Ga_jf_3\big\}_{j\in\bbz}\big\Vert_{\ell^{\infty}}\big)(x)
\end{align*}  and this proves \eqref{keylemma1est}.
Moreover,
\begin{align*}
 \Vert u_1\Vert_{L^{p_1}(\bbrn)} \le \Big( \sum_{k=0}^{\infty}|\la_k|^{p_1}\Big\Vert    \mathcal{M}_{q}\big( \big\Vert \big\{ \mathcal{M}_r\La_ja_k\big\}_{j\in\bbz}\big\Vert_{\ell^2}\big)      \Big\Vert_{L^{p_1}(Q_k^{***})}^{p_1} \Big)^{{1}/{p_1}}\lesssim 1
\end{align*}
where the last inequality follows from \eqref{hardylambda} and the estimate
\begin{align*}
&\Big\Vert    \mathcal{M}_{q}\big( \big\Vert \big\{ \mathcal{M}_r\La_ja_k\big\}_{j\in\bbz}\big\Vert_{\ell^2}\big)      \Big\Vert_{L^{p_1}(Q_k^{***})}\lesssim |Q_k|^{{1}/{p_1}-{1}/{r_0}}\Big\Vert    \mathcal{M}_{q}\big( \big\Vert \big\{ \mathcal{M}_r\La_ja_k\big\}_{j\in\bbz}\big\Vert_{\ell^2}\big)      \Big\Vert_{L^{r_0}(\bbrn)}\\
&\qq\qq\qq\qq\qq \lesssim |Q_k|^{{1}/{p_1}-{1}/{r_0}}\big\Vert \big\{ \La_j a_k\big\}_{j\in\bbz}\big\Vert_{L^{r_0}(\ell^2)}\sim  |Q_k|^{{1}/{p_1}-{1}/{r_0}}\Vert a_k\Vert_{L^{r_0}(\bbrn)}\lesssim 1
\end{align*} for $q<r_0<\infty$.
Here, we applied H\"older's inequality, the maximal inequality \eqref{hlmax}, the equivalence in \eqref{hardydeflittle}, and properties of the $H^{p_1}$-atom $a_k$.
It is also easy to verify
$$\Vert u_2\Vert_{L^{p_2}(\bbrn)}\lesssim \big\Vert  \big\{ \La_jf_2\big\}_{j\in\bbz}\big\Vert_{L^{p_2}(\ell^2)}\sim 1$$
and 
$$\Vert u_3\Vert_{L^{p_3}(\bbrn)}\lesssim \big\Vert  \big\{ \Ga_jf_3\big\}_{j\in\bbz}\big\Vert_{L^{p_3}(\ell^{\infty})}\sim 1$$
using \eqref{hlmax}, \eqref{hardydef}, and \eqref{hardydeflittle}.

\subsection{Proof of Lemma \ref{keylemma2}}
Since $$s>{n}/{p_1}+{n}/{2}=\big({n}/{p_1}-{n}/{2}\big)+{n}/{2}+{n}/{2},$$
we can choose $s_1,s_2,s_3$ such that
$s_1>{n}/{p_1}-{n}/{2}$, $s_2,s_3>{n}/{2}$, and $s=s_1+s_2+s_3$.

Using the estimates
$$\big\Vert \psi^P \big\Vert_{L^{\infty}(\bbrn)}\le |P|^{-{1}/{2}}\q \text{ and }\q \big\Vert \theta^R \big\Vert_{L^{\infty}(\bbrn)}\le |R|^{-{1}/{2}},$$
we have 
\begin{align}\label{u1inoutest}
\big|\UU_1^{\mathrm{in}}(x,y)\big|&\lesssim\sum_{j\in\bbz}  \Big(\sum_{P\in\DD_j}|b_P^2||P|^{-{1}/{2}}\chi_P(x) \Big)\Big(\sum_{R\in\DD_j}|b_R^3| |R|^{-{1}/{2}}\chi_R(x) \Big) \nonumber \\
 &\qq\qq\qq\times \int_{(\bbrn)^3}    \big| \sigma_j^{\vee}(y-z_1,z_2,z_3)\big|\big| \La_ja_k(z_1)\big|\chi_{Q_k^*}(z_1)  \; d\zzz   \nonumber \\
&\le g^2\big(\{b_P^2\}_{P\in\DD} \big)(x) g^{\infty}\big( \{b_R^3\}_{R\in\DD}\big)(x)\\
&\qq\times   \bigg( \sum_{j\in\bbz}\Big(\int_{(\bbrn)^3}     \big| \sigma_j^{\vee}(y-z_1,z_2,z_3)\big|\big| \La_ja_k(z_1)\big| \chi_{Q_k^*}(z_1) \;  d\zzz \Big)^2 \bigg)^{{1}/{2}}     .\nonumber
\end{align}
We observe that for $|x-y|\le 2^{-l}$, $x\in (Q_k^{***})^c\cap (B_k^l)^c$, and $z_1\in Q_k^*$,
\begin{equation}\label{xcyz}
|x-\xx_{Q_k}|\lesssim |y-z_1|
\end{equation}
and thus, by using Lemma \ref{technical lemma},
\begin{align*}
&\langle 2^j(x-\xx_{Q_k})\rangle^{s_1}\int_{(\bbrn)^3}     \big| \sigma_j^{\vee}(y-z_1,z_2,z_3)\big|\big| \La_ja_k(z_1)\big|  \chi_{Q_k^*}(z_1) \; d\zzz \\
&\lesssim \ell(Q_k)^{-{n}/{p_1}}\min \big\{1,\big(2^j\ell(Q_k)\big)^M \big\}    \int_{(\bbrn)^3} \langle 2^j(y-z_1)\rangle^{s_1}\big| \sigma_j^{\vee}(y-z_1,z_2,z_3)\big|\chi_{Q_k^*}(z_1) \; d\zzz \\
&\lesssim \ell(Q_k)^{-{n}/{p_1}}\min \big\{1,\big(2^j\ell(Q_k)\big)^M \big\}  \big\Vert \langle 2^j\cdot \rangle^{-s_2}\big\Vert_{L^2(\bbrn)}\big\Vert \langle 2^j\cdot \rangle^{-s_3}\big\Vert_{L^2(\bbrn)} 2^{jn} I_{k,j,s}^{\mathrm{in}}(y)    \\
&\sim \ell(Q_k)^{-{n}/{p_1}}\min \big\{1,\big(2^j\ell(Q_k)\big)^M \big\} I_{k,j,s}^{\mathrm{in}}(y)
\end{align*}
for sufficiently large $M$,
where
\begin{equation}\label{ikjsindef}
I_{k,j,s}^{\mathrm{in}}(y):= 2^{-jn}\int_{Q_k^*}\Big\Vert \langle 2^j(y-z_1),2^jz_2,2^jz_3\rangle^s\big| \sigma_j^{\vee}(y-z_1,z_2,z_3)\big|\Big\Vert_{L^2(z_2,z_3)} \; dz_1.  
\end{equation}
This proves that
\begin{align}\label{keyinest}
&\int_{(\bbrn)^3}     \big| \sigma_j^{\vee}(y-z_1,z_2,z_3)\big|\big| \La_ja_k(z_1)\big|  \chi_{Q_k^*}(z_1) \; d\zzz \nonumber\\
&\lesssim \ell(Q_k)^{-{n}/{p_1}}\min \big\{1,\big(2^j\ell(Q_k)\big)^M \big\}  \langle 2^j(x-\xx_{Q_k})\rangle^{-s_1} I_{k,j,s}^{\mathrm{in}}(y)
\end{align}
and therefore, we obtain
\begin{align}\label{uu11inest}
\big| \UU_1^{\mathrm{in}}(x,y)\big|&\lesssim  g^2\big(\{b_P^2\}_{P\in\DD} \big)(x) g^{\infty}\big( \{b_R^3\}_{R\in\DD}\big)(x) \ell(Q_k)^{-{n}/{p_1}} | x-\xx_{Q_k}|^{-s_1}\\
&\qq\qq \times \bigg(\sum_{j\in\bbz}\Big(2^{-s_1j} \min \big\{1,\big(2^j\ell(Q_k)\big)^{M} \big\}  I_{k,j,s}^{\mathrm{in}}(y)\Big)^2   \bigg)^{1/2}\nonumber
\end{align}

\hfill 

Similar to \eqref{u1inoutest}, we write
\begin{align*}
\big|\UU_1^{\mathrm{out}}(x,y)\big|&\lesssim  g^2\big(\{b_P^2\}_{P\in\DD} \big)(x) g^{\infty}\big( \{b_R^3\}_{R\in\DD}\big)(x) \\
&\qq \times  \bigg( \sum_{j\in\bbz}\Big(\int_{(\bbrn)^3}     \big| \sigma_j^{\vee}(y-z_1,z_2,z_3)\big|\big| \La_ja_k(z_1)\big| \chi_{(Q_k^*)^c}(z_1)  d\zzz \Big)^2 \bigg)^{{1}/{2}} .
\end{align*}
Instead of \eqref{xcyz}, we make use of the estimate
\begin{equation}\label{alterest}
\langle 2^j(x-\xx_{Q_k})\rangle \lesssim \langle 2^j(y-\xx_{Q_k})\rangle\le \langle 2^j(y-z_1)\rangle \langle 2^j(z_1-\xx_{Q_k})\rangle
\end{equation}
for $|x-y|\le 2^{-l}$ and $x\in (Q_k^{***})^c\cap (B_k^l)^c$.
Then, using the argument that led to \eqref{keyinest}, we have
\begin{align}\label{keyoutest}
&\int_{(\bbrn)^3}     \big| \sigma_j^{\vee}(y-z_1,z_2,z_3)\big|\big| \La_ja_k(z_1)\big| \chi_{(Q_k^*)^c}(z_1) \;  d\zzz \nonumber\\
&\lesssim \ell(Q_k)^{-{n}/{p_1}}\min\big\{ 1,\big( 2^j\ell(Q_k)\big)^M\big\} \langle 2^j(x-\xx_{Q_k})\rangle^{-s_1}I_{k,j,s}^{\mathrm{out}}(y)
\end{align}
where $M$, $L_0$ are sufficiently large numbers and
\begin{align}\label{ikjsout}
I_{k,j,s}^{\mathrm{out}}(y)&:=2^{-jn}\int_{(Q_k^*)^c}   \frac{(2^j\ell(Q_k))^n}{\langle 2^j(z_1-\xx_{Q_k})\rangle^{L_0-s_1}} \\
&\qq\qq\times  \Big\Vert \langle 2^j(y-z_1),2^jz_2,2^jz_3\rangle^s\big| \sigma_j^{\vee}(y-z_1,z_2,z_3)\big|\Big\Vert_{L^2(z_2,z_3)}    \;     dz_1.\nonumber
\end{align}
Now we deduce
\begin{align}\label{uu11outest}
\big| \UU_1^{\mathrm{out}}(x,y)\big|&\lesssim g^2\big(\{b_P^2\}_{P\in\DD} \big)(x) g^{\infty}\big( \{b_R^3\}_{R\in\DD}\big)(x) \ell(Q_k)^{-{n}/{p_1}}| x-\xx_{Q_k}|^{-s_1}\\
&\qq\times \bigg(\sum_{j\in\bbz}\Big(2^{-s_1j} \min \big\{1,\big(2^j\ell(Q_k)\big)^{M} \big\}  I_{k,j,s}^{\mathrm{out}}(y)\Big)^2   \bigg)^{{1}/{2}}.\nonumber
\end{align}

According to \eqref{uu11inest} and \eqref{uu11outest}, the estimate \eqref{keylemma2est} follows from taking
\begin{align*}
u_1^{\mathrm{in}/\mathrm{out}}(x)&:=\sum_{k=0}^{\infty}|\la_k|\ell(Q_k)^{-{n}/{p_1}}\chi_{(Q_k^{***})^c}(x)| x-\xx_{Q_k}|^{-s_1}\\
&\qq\times\mathcal{M}\bigg[\bigg(\sum_{j\in\bbz}\Big(2^{-s_1j}\min{\{1,(2^j\ell(Q_k))^M\}}I_{k,j,s}^{\mathrm{in}/\mathrm{out}}(\cdot) \Big)^2 \bigg)^{{1}/{2}}\bigg](x),\nonumber\\
u_2(x)&:=g^2\big(\{b_P^2\}_{P\in\DD} \big)(x) ,\nonumber\\
u_3(x)&:=g^{\infty}\big( \{b_R^3\}_{R\in\DD}\big)(x).\nonumber
\end{align*}
It is clear that
\begin{align}
\Vert u_2\Vert_{L^{p_2}(\bbrn)}&=\big\Vert \{b_P^2\}_{P\in\DD}\big\Vert_{\dot{f}^{p_2,2}}\sim 1 \label{u2p2est}\\
\Vert u_3\Vert_{L^{p_3}(\bbrn)}&=\big\Vert \{b_R^3\}_{R\in\DD}\big\Vert_{\dot{f}^{p_3,\infty}}\sim 1 \label{u3p3est}
\end{align}
in view of \eqref{bp2} and \eqref{br3}.
To estimate $u_1^{\mathrm{in}}$ and $u_1^{\mathrm{out}}$, we note that
\begin{align}\label{ikjsin}
\big\Vert I_{k,j,s_1}^{\mathrm{in}}\big\Vert_{L^2(\bbrn)}&\le 2^{-jn}\int_{Q_k}\bigg( \int_{\bbrn}\Big\Vert \langle 2^jy,2^jz_2,2^jz_3\rangle^s \big| \sigma_j^{\vee}(y,z_2,z_3)\big|\Big\Vert_{L^2(z_2,z_3)}^2 \; dy\bigg)^{{1}/{2}} \; dz_1\nonumber\\
&=2^{-jn}\ell(Q_k)^n\big\Vert \langle 2^j\ccdot\rangle^s\sigma_j^{\vee}\big\Vert_{L^2((\bbrn)^3)}\le 2^{{jn}/{2}}\ell(Q_k)^n
\end{align}
where we applied Minkowski's inequality and a change of variables,
and similarly,
\begin{align}\label{ikjsoutest}
\big\Vert I_{k,j,s}^{\mathrm{out}}\big\Vert_{L^2(\bbrn)}&\lesssim 2^{-jn}\int_{(Q_k^*)^c}  \frac{(2^j\ell(Q_k))^n}{\langle 2^j(z_1-\xx_{Q_k})\rangle^{L_0-s_1}}  \;  dz_1 \big\Vert \langle 2^j\ccdot\rangle^s\sigma_j^{\vee} \big\Vert_{L^2((\bbrn)^3)}\nonumber\\
&\lesssim 2^{{jn}/{2}}\ell(Q_k)^n\big( 2^j\ell(Q_k)\big)^{-(L_0-s_1-n)} 
\end{align}
for $L_0>s+n$.
Now we have
\begin{align*}
\big\Vert u_1^{\mathrm{in}}\big\Vert_{L^{p_1}(\bbrn)}^{p_1}&\le \sum_{k=0}^{\infty}|\la_k|^{p_1}\ell(Q_k)^{-n}\int_{(Q_k^{***})^c} | x-\xx_{Q_k}|^{-s_1p_1}\\
&\qq\qq\times \bigg( \mathcal{M}\bigg[\bigg(\sum_{j\in\bbz}\Big(2^{-s_1j}\min{\big\{1,\big(2^j\ell(Q_k)\big)^M\big\}}I_{k,j,s_1}^{\mathrm{in}}(\cdot) \Big)^2 \bigg)^{1/2}\bigg](x) \bigg)^{p_1}\, dx
\end{align*}
and the integral is dominated by
\begin{align*}
&\big\Vert  |\cdot-\xx_{Q_k}|^{-{s_1p_1}}\big\Vert_{L^{({2}/{p_1})'}((Q_k^{***})^c)}\\
&\qq\times \bigg\Vert \bigg( \mathcal{M}\bigg[\bigg(\sum_{j\in\bbz}\Big(2^{-s_1j}\min{\big\{1,\big(2^j\ell(Q_k)\big)^M\big\}}I_{k,j,s_1}^{\mathrm{in}}(\cdot) \Big)^2 \bigg)^{1/2}\bigg] \bigg)^{p_1} \bigg\Vert_{L^{{2}/{p_1}}(\bbrn)}.
\end{align*}
The first term is no more than a constant times $\ell(Q_k)^{-p_1(s_1-({n}/{p_1}-{n}/{2}))}$ and the second one is bounded by
\begin{align*}
&\bigg( \sum_{j\in\bbz} \Big(2^{-s_1j}\min \big\{1,\big( 2^j\ell(Q_k)\big)^N \big\}\big\Vert I_{k,j,s_1}^{\mathrm{in}}\big\Vert_{L^2(\bbrn)} \Big)^2\bigg)^{{p_1}/{2}}\\
&\lesssim \ell(Q_k)^{p_1n}\bigg(\sum_{j\in\bbz}\Big( 2^{-s_1j}  \min \big\{1,\big( 2^j\ell(Q_k)\big)^M \big\} 2^{{jn}/{2}}     \Big)^2 \bigg)^{{p_1}/{2}}\lesssim \ell(Q_k)^{s_1p_1+{p_1n}/{2}},
\end{align*}
due to \eqref{ikjsin}.
This proves 
\begin{equation}\label{u1inest}
\big\Vert u_1^{\mathrm{in}}\big\Vert_{L^{p_1}(\bbrn)}\lesssim  \Big( \sum_{k=0}^{\infty}|\la_k|^{p_1}\Big)^{{1}/{p_1}}\lesssim 1.
\end{equation}
In a similar way, together with \eqref{ikjsoutest},  we can also prove 
\begin{equation}\label{u1outest}
\big\Vert u_1^{\mathrm{out}}\big\Vert_{L^{p_1}(\bbrn)}\lesssim \Big( \sum_{k=0}^{\infty}|\la_k|^{p_1}\Big)^{{1}/{p_1}}\lesssim 1,
\end{equation}
 choosing $M>L_0-{3n}/{2}$.

\subsection{Proof of Lemma \ref{keylemma3}}
As in the proof of Lemma \ref{keylemma2}, we pick $s_1,s_2,s_3$ satisfying
$s_1>{n}/{p_1}-{n}/{2}$, $s_2,s_3>{n}/{2}$, and $s=s_1+s_2+s_3>n/p_1+n/2$.

We first consider the case $\eta=1$.
For $x\in P^c\cap (B_P^l)^c$  and      $|x-y|\le 2^{-l}$,
we have
\begin{equation}\label{xcp2j}
 \langle 2^j(x-\xx_{P})\rangle\lesssim \langle 2^j(y-\xx_P)\rangle \le \langle 2^j(y-z_2)\rangle \langle 2^j(z_2-\xx_P)\rangle.
\end{equation}
By using $$\big\Vert \theta^R\big\Vert_{L^{\infty}(\bbrn)}\le |R|^{-{1}/{2}},$$
we have
\begin{align}\label{u21inest}
\big| \UU_2^{1,\mathrm{in}}(x,y)\big| &\lesssim \sum_{j\in\bbz}\Big( \sum_{R\in\DD_j}|b_R^3||R|^{-{1}/{2}}\chi_R(x)\Big)\int_{(\bbrn)^3}  \big|\sigma_j^{\vee}(y-z_1,y-z_2,z_3) \big|\\
&\qq\q\times \big| \La_ja_k(z_1)\big|\chi_{Q_k^*}(z_1)\Big(\sum_{P\in\DD_j} |b_P^2|\chi_{P^c}(x)\chi_{(B_P^l)^c}(x)\big| \psi^P(z_2)\big|     \Big)   \; d\zzz.\nonumber
\end{align}
Using \eqref{xcyz}, \eqref{xcp2j}, and Lemma \ref{technical lemma}, the integral in the preceding expression is bounded by
\begin{align*}
& \ell(Q_k)^{-{n}/{p_1}}\min\big\{1,\big(2^j\ell(Q_k) \big)^M \big\}\langle2^j(x-c_{Q_k})\rangle^{-s_1}\int_{(\bbrn)^3} \langle 2^j(y-z_1)\rangle^{s_1}\langle 2^j(y-z_2)\rangle^{s_2} \\
& \qq\qq \times \big| \sigma_j^{\vee}(y-z_1,y-z_2,z_3)\big|\chi_{Q_k^*}(z_1)\bigg(\sum_{P\in\DD_j}|b_P^2|\frac{\chi_{P^c}(x)}{\langle 2^j(x-\xx_P)\rangle^{s_2}}\big| \wt{\psi^P}(z_2)\big| \bigg)  \;     d\zzz\\
&\lesssim \ell(Q_k)^{-{n}/{p_1}}\min\big\{1,\big(2^j\ell(Q_k) \big)^M \big\} \langle2^j(x-\xx_{Q_k})\rangle^{-s_1}I_{k,j,s}^{\mathrm{in}}(y) \\
&\qq\qq\qq\qq\times 2^{{jn}/{2}}\bigg\Vert \sum_{P\in\DD_j}|b_P^2|\frac{\chi_{P^c}(x)}{\langle 2^j(x-\xx_P)\rangle^{s_2}}\big| \wt{\psi^P}(\cdot)\big|\bigg\Vert_{L^2(\bbrn)}
\end{align*}
for sufficiently large $M>0$,
where $\wt{\psi^P}(z_2):=\langle 2^j(z_2-\xx_P)\rangle^{s_2}\psi^P(z_2)$ for $P\in\DD_j$ and $I_{k,j,s}^{\mathrm{in}}$ is defined as in \eqref{ikjsindef}.
Note that 
\begin{equation}\label{bpf2}
|b_P^2|\lesssim \BB_P^2(f_2):=\Big\langle \big| \wt{\La_j}f_2\big|,\frac{2^{{jn}/{2}}}{\langle 2^j(\cdot-\xx_P)\rangle^L}\Big\rangle \q \text{ for }~L>n,s
\end{equation}
and thus it follows from Lemma \ref{almost orthogonality} that the $L^2$ norm in the last displayed expression is dominated by
$$2^{-{jn}/{2}}\bigg( \sum_{P\in\DD_j} \Big( | \BB_P^2(f_2)| |P|^{-{1}/{2}}   \frac{\chi_{P^c}(x)}{\langle 2^j(x-\xx_P)\rangle^{s_2}}\Big)^2\bigg)^{{1}/{2}}.$$
This yields that
\begin{align}\label{u21inpointest}
\big| \UU_2^{1,\mathrm{in}}(x,y) \big|&\lesssim \ell(Q_k)^{-{n}/{p_1}}|x-c_{Q_k}|^{-s_1}\Big( \sum_{j\in\bbz} \big(2^{-js_1}\min\big\{1,\big(2^j\ell(Q_k) \big)^M \big\}I_{k,j,s}^{\mathrm{in}}(y) \big)^2\Big)^{{1}/{2}}\\
&\qq \times \bigg( \sum_{P\in\DD} \Big( \big| \BB_P^2(f_2)\big| |P|^{-{1}/{2}} \frac{\chi_{P^c}(x)}{\langle 2^j(x-\xx_P)\rangle^{s_2}}\Big)^2   \bigg)^{{1}/{2}}g^{\infty}\big(\{b_R^3\}_{R\in\DD}\big)(x)\nonumber.
\end{align}

Similarly, using \eqref{alterest}, \eqref{xcp2j}, Lemma \ref{technical lemma}, and Lemma \ref{almost orthogonality}, we have
\begin{align}\label{u21outpointest}
\big| \UU_2^{1,\mathrm{out}}(x,y) \big|&\lesssim \sum_{j\in\bbz}\Big(\sum_{R\in\DD_j} |b_R^3| |R|^{-{1}/{2}}\chi_R(x) \Big) \int_{(\bbrn)^3}  \big|\sigma_j^{\vee}(y-z_1,y-z_2,z_3) \big|\nonumber\\
&\qq\qq\times \big| \La_ja_k(z_1)\big|\chi_{(Q_k^*)^c}(z_1)\Big(\sum_{P\in\DD_j} |b_P^2|\chi_{P^c}(x)\chi_{(B_P^l)^c}(x)\big| \psi^P(z_2)\big|     \Big)   \; d\zzz \nonumber\\
&\lesssim \ell(Q_k)^{-{n}/{p_1}}|x-\xx_{Q_k}|^{-s_1}\Big( \sum_{j\in\bbz} \big(2^{-js_1}\min\big\{1,\big(2^j\ell(Q_k) \big)^M \big\}I_{k,j,s}^{\mathrm{out}}(y) \big)^2\Big)^{{1}/{2}}\\
&\qq\qq \times \bigg( \sum_{P\in\DD} \Big( \big| \BB_P^2(f_2)\big| |P|^{-{1}/{2}} \frac{\chi_{P^c}(x)}{\langle 2^j(x-\xx_P)\rangle^{s_2}}\Big)^2   \bigg)^{{1}/{2}}g^{\infty}\big(\{b_R^3\}_{R\in\DD}\big)(x) \nonumber
\end{align}
where $I_{k,j,s}^{\mathrm{out}}$ is defined as in \eqref{ikjsout}.

\hfill

When $\eta=2$, we use the inequality
\begin{equation}\label{xcpxcq}
\langle 2^j(x-\xx_{Q_k})\rangle^{-1}\le \langle 2^j (x-\xx_{P})\rangle^{-1}
\end{equation}  for $x\in (B_k^l)^c \cap B_{P}^l$.
Then, similar to \eqref{u21inest}, we have
\begin{align*}
\big| \UU_2^{2,\mathrm{in}}(x,y) \big|&\lesssim \sum_{j\in\bbz}\Big( \sum_{R\in\DD_j}|b_R^3||R|^{-{1}/{2}}\chi_R(x)\Big)\int_{(\bbrn)^3}  \big|\sigma_j^{\vee}(y-z_1,y-z_2,z_3) \big|\\
&\qq\q\times \big| \La_ja_k(z_1)\big|\chi_{Q_k^*}(z_1)\Big(\sum_{P\in\DD_j} |b_P^2|\chi_{P^c}(x)\chi_{B_P^l}(x)\big| \psi^P(z_2)\big|     \Big)   \; d\zzz\nonumber
\end{align*}
and the integral is dominated by a constant times
\begin{align*}
&\ell(Q_k)^{-{n}/{p_1}}\min \{1,\big( 2^j\ell(Q_k)\big)^M\}\langle 2^j(x-\xx_{Q_k})\rangle^{-(s_1+s_2)}\int_{(\bbrn)^3}\langle 2^j(y-z_1)\rangle^{s_1+s_2}\\
&\qq\qq\qq\times \big| \sigma_j^{\vee}(y-z_1,y-z_2,z_3)\big|\chi_{Q_k^*}(z_1)\Big( \sum_{P\in\DD_j}|b_P^2|\chi_{P^c}(x)\chi_{B_P^l}(x)\big| {\psi^P}(z_2)\big|\Big)\;d\zzz\\
&\le \ell(Q_k)^{-{n}/{p_1}}\min \{1,\big( 2^j\ell(Q_k)\big)^M\}\langle 2^j(x-\xx_{Q_k})\rangle^{-s_1}\int_{(\bbrn)^3}\langle 2^j(y-z_1)\rangle^{s_1+s_2}\\
&\qq\qq\qq\times \big| \sigma_j^{\vee}(y-z_1,y-z_2,z_3)\big|\chi_{Q_k^*}(z_1)\Big( \sum_{P\in\DD_j}|b_P^2|\frac{\chi_{P^c}(x)}{\langle 2^j(x-\xx_P)\rangle^{s_2}}\big| {\psi^P}(z_2)\big|\Big)\;d\zzz\\
&\lesssim \ell(Q_k)^{-{n}/{p_1}}\min\big\{1,\big(2^j\ell(Q_k) \big)^M \big\} \langle2^j(x-\xx_{Q_k})\rangle^{-s_1}I_{k,j,s}^{\mathrm{in}}(y)\\
&\qq\qq\qq\qq\qq\qq\times  \bigg( \sum_{P\in\DD} \Big( \big| \BB_P^2(f_2)\big| |P|^{-{1}/{2}} \frac{\chi_{P^c}(x)}{\langle 2^j(x-\xx_P)\rangle^{s_2}}\Big)^2   \bigg)^{{1}/{2}}
\end{align*}
due to  \eqref{xcyz}, \eqref{xcpxcq}, Lemma \ref{technical lemma}, and Lemma \ref{almost orthogonality}, where $I_{k,j,s}^{\mathrm{in}}$ and $\BB_P^2(f_2)$ are defined as in \eqref{ikjsindef} and \eqref{bpf2}. 
Therefore,
\begin{align}\label{u22inpointest}
\big| \UU_2^{2,\mathrm{in}}(x,y)\big|&\lesssim \ell(Q_k)^{-{n}/{p_1}}|x-\xx_{Q_k}|^{-s_1}\Big( \sum_{j\in\bbz} \big(2^{-js_1}\min\big\{1,\big(2^j\ell(Q_k) \big)^M \big\}I_{k,j,s}^{\mathrm{in}}(y) \big)^2\Big)^{{1}/{2}}\\
&\qq \times \bigg( \sum_{P\in\DD} \Big( \big| \BB_P^2(f_2)\big| |P|^{-{1}/{2}} \frac{\chi_{P^c}(x)}{\langle 2^j(x-\xx_P)\rangle^{s_2}}\Big)^2   \bigg)^{{1}/{2}} g^{\infty}\big(\{b_R^3\}_{R\in\DD}\big)(x).\nonumber
\end{align}

Similarly, we can also prove that
\begin{align}\label{u22outpointest}
\big| \UU_2^{2,\mathrm{out}}(x,y)\big|&\lesssim \ell(Q_k)^{-{n}/{p_1}}|x-\xx_{Q_k}|^{-s_1}\Big( \sum_{j\in\bbz} \big(2^{-js_1}\min\big\{1,\big(2^j\ell(Q_k) \big)^M \big\}I_{k,j,s}^{\mathrm{out}}(y) \big)^2\Big)^{{1}/{2}}\\
&\qq \times \bigg( \sum_{P\in\DD} \Big( \big| \BB_P^2(f_2)\big| |P|^{-{1}/{2}} \frac{\chi_{P^c}(x)}{\langle 2^j(x-\xx_P)\rangle^{s_2}}\Big)^2   \bigg)^{{1}/{2}} g^{\infty}\big(\{b_R^3\}_{R\in\DD}\big)(x).\nonumber
\end{align}
Combining \eqref{u21inpointest}, \eqref{u21outpointest}, \eqref{u22inpointest}, and \eqref{u22outpointest},
the estimate \eqref{keylemma3est} holds with
\begin{align*}
u_1^{\mathrm{in}/\mathrm{out}}(x)&:=\sum_{k=0}^{\infty}|\la_k|\ell(Q_k)^{-{n}/{p_1}}\chi_{(Q_k^{***})^c}(x)| x-\xx_{Q_k}|^{-s_1}\\
&\qq\times\mathcal{M}\bigg[\bigg(\sum_{j\in\bbz}\Big(2^{-s_1j}\min{\{1,(2^j\ell(Q_k))^M\}}I_{k,j,s}^{\mathrm{in}/\mathrm{out}}(\cdot) \Big)^2 \bigg)^{{1}/{2}}\bigg](x),\\
u_2(x)&:= \bigg(\sum_{j\in\bbz} \sum_{P\in\DD_j} \Big( \big| \BB_P^2(f_2)\big| |P|^{-{1}/{2}} \frac{\chi_{P^c}(x)}{\langle 2^j(x-\xx_P)\rangle^{s_2}}\Big)^2   \bigg)^{{1}/{2}},\\
u_3(x)&:=g^{\infty}\big( \{b_R^3\}_{R\in\DD}\big)(x).
\end{align*}
Clearly,  as in \eqref{u1inest}, \eqref{u1outest}, and \eqref{u3p3est},
\begin{align*}
\big\Vert u_1^{\mathrm{in}/\mathrm{out}}\big\Vert_{L^{p_1}(\bbrn)}\lesssim 1, \qq \Vert u_3\Vert_{L^{p_3}(\bbrn)}\lesssim 1
\end{align*}
and Lemma \ref{lacunarylemma} proves that
\begin{equation*}
\Vert u_2\Vert_{L^{p_2}(\bbrn)}\lesssim 1.
\end{equation*}

\subsection{Proof of Lemma \ref{keylemma4}}

The proof is almost same as that of Lemma \ref{keylemma3}.
By letting $M>0$ be sufficiently large and exchanging the role of terms associated with $f_2$ and $f_3$ in the estimate \eqref{u21inest}, 
we may obtain
\begin{align}\label{u31inest}
\big| \UU_3^{1,\mathrm{in}}(x,y)\big|&\lesssim \sum_{j\in\bbz}\Big( \sum_{P\in\DD_j}|b_P^2||P|^{-{1}/{2}}\chi_P(x)\Big)\int_{(\bbrn)^3}  \big|\sigma_j^{\vee}(y-z_1,z_2,y-z_3) \big|\nonumber\\
&\qq\q\times \big| \La_ja_k(z_1)\big|\chi_{Q_k^*}(z_1)\Big(\sum_{R\in\DD_j} |b_R^3|\chi_{R^c}(x)\chi_{(B_R^l)^c}(x)\big| \theta^R(z_3)\big|     \Big)   \; d\zzz\nonumber\\
&\lesssim   \ell(Q_k)^{-{n}/{p_1}}|x-\xx_{Q_k}|^{-s_1}\bigg( \sum_{j\in\bbz}  \Big( 2^{-js_1}\min{\big\{ 1,\big( 2^j\ell(Q_k)\big)^M\big\}}   I_{k,j,s}^{\mathrm{in}}(y)\Big)^2              \bigg)^{{1}/{2}}\\
&\qq\times g^2\big(\{b_P^2 \}_{P\in\DD} \big)(x)\sup_{j\in\bbz}\bigg( \sum_{R\in\DD_j} \Big(\big| \BB_R^3(f_3) \big| |R|^{-{1}/{2}}\frac{1}{\langle 2^j(x-\xx_R)\rangle ^{s_3}}    \Big)^2 \bigg)^{{1}/{2}}\nonumber
\end{align}
where $I_{k,j,s}^{\mathrm{in}}$ is defined as in \eqref{ikjsindef}, $\wt{\theta^R}(z_3):=\langle 2^j(z_3-\xx_R)\rangle^{s_3}\theta^R(z_3)$ for $R\in\DD_j$ and 
\begin{equation}\label{br3f3def}
 \BB_R^3(f_3):=\Big\langle \big| \wt{\Ga_j}f_3\big|,\frac{2^{{jn}/{2}}}{\langle 2^j(\cdot-\xx_R)\rangle^L}\Big\rangle \q \text{ for }~L>s,n.
 \end{equation}

Similarly,
\begin{align}\label{u31outest}
\big| \UU_3^{1,\mathrm{out}}(x,y) \big|&\lesssim   \ell(Q_k)^{-{n}/{p_1}} |x-\xx_{Q_k}|^{-s_1}   \bigg( \sum_{j\in\bbz}  \Big(2^{-js_1} \min{\big\{ 1,\big( 2^j\ell(Q_k)\big)^M\big\}}    I_{k,j,s}^{\mathrm{out}}(y)\Big)^2              \bigg)^{{1}/{2}}\\
&\qq\times g^2\big(\{b_P^2 \}_{P\in\DD} \big)(x)\sup_{j\in\bbz}\bigg( \sum_{R\in\DD_j} \Big(\big| \BB_R^3(f_3) \big| |R|^{-{1}/{2}}\frac{1}{\langle 2^j(x-\xx_R)\rangle ^{s_3}}    \Big)^2 \bigg)^{{1}/{2}}\nonumber
\end{align}
where $I_{k,j,s}^{\mathrm{out}}$ is defined as in \eqref{ikjsout}.

For the case $\eta=2$, we use the fact that for $x\in (B_k^l)^c\cap B_R^l$,
$$\langle 2^j(x-\xx_{Q_k})\rangle^{-1}\le \langle 2^j(x-\xx_R)\rangle^{-1},$$
instead of \eqref{xcpxcq}.
Then we have
\begin{align*}
\big| \UU_3^{2,\mathrm{in}}(x,y)\big| &\lesssim \ell(Q_k)^{-{n}/{p_1}}|x-\xx_{Q_k}|^{-s_1}\Big( \sum_{j\in\bbz} \big(2^{-js_1}\min\big\{1,\big(2^j\ell(Q_k) \big)^M \big\}I_{k,j,s}^{\mathrm{in}}(y) \big)^2\Big)^{{1}/{2}}\\
&\qq \times g^{2}\big(\{b_P^2\}_{P\in\DD}\big)(x)    \sup_{j\in\bbz}\bigg( \sum_{R\in\DD_j} \Big(\big| \BB_R^3(f_3) \big| |R|^{-{1}/{2}}\frac{1}{\langle 2^j(x-\xx_R)\rangle ^{s_3}}    \Big)^2 \bigg)^{{1}/{2}} 
\end{align*}
and
\begin{align*}
\big| \UU_3^{2,\mathrm{out}}(x,y) \big|&\lesssim   \ell(Q_k)^{-{n}/{p_1}}\bigg( \sum_{j\in\bbz}  \Big( \min{\big\{ 1,\big( 2^j\ell(Q_k)\big)^M\big\}}    \langle 2^j(x-\xx_Q)\rangle^{-s_1}I_{k,j,s}^{\mathrm{out}}(y)\Big)^2              \bigg)^{{1}/{2}}\\
&\qq\times g^2\big(\{b_P^2 \}_{P\in\DD} \big)\sup_{j\in\bbz}\bigg( \sum_{R\in\DD_j} \Big(\big| \BB_R^3(f_3) \big| |R|^{-{1}/{2}}\frac{1}{\langle 2^j(x-\xx_R)\rangle ^{s_3}}    \Big)^2 \bigg)^{{1}/{2}}
\end{align*}
 which are analogous to \eqref{u31inest} and \eqref{u31outest}.

Then Lemma \ref{keylemma4} follows from \eqref{u1inest}, \eqref{u1outest}, \eqref{u2p2est}, and Lemma \ref{nonlacunarylemma}, by 
choosing
\begin{align*}
u_1^{\mathrm{in}/\mathrm{out}}(x)&:= \sum_{k=0}^{\infty}|\la_k|\ell(Q_k)^{-{n}/{p_1}}\chi_{(Q_k^{***})^c}(x)| x-\xx_{Q_k}|^{-s_1}\\
&\qq\times\mathcal{M}\bigg[\bigg(\sum_{j\in\bbz}\Big(2^{-s_1j}\min{\{1,(2^j\ell(Q_k))^M\}}I_{k,j,s}^{\mathrm{in}/\mathrm{out}}(\cdot) \Big)^2 \bigg)^{{1}/{2}}\bigg](x)      \\
u_2(x)&:=       g^2\big(\{b_P^2 \}_{P\in\DD} \big)(x)      \\
u_3(x)&:=   \sup_{j\in\bbz}\bigg( \sum_{R\in\DD_j} \Big(\big| \BB_R^3(f_3) \big| |R|^{-{1}/{2}}\frac{1}{\langle 2^j(x-\xx_R)\rangle ^{s_3}}    \Big)^2 \bigg)^{{1}/{2}}        .
\end{align*}

\subsection{Proof of Lemma \ref{keylemma5}}

Let $I_{k,j,s}^{\mathrm{in}}$, $I_{k,j,s}^{\mathrm{out}}$, $\BB_P^2(f_2)$, and $\BB_R^3(f_3)$ be defined as before. Let $M>0$ be a sufficiently large number.
We claim the pointwise estimates that for each $\eta=1,2,3,4$,
\begin{align*}
\big| \UU_4^{\eta,\mathrm{in}/\mathrm{out}}(x,y)\big|&\lesssim \ell(Q_k)^{-{n}/{p_1}}|x-x_{Q_k}|^{-s_1}\bigg( \sum_{j\in\bbz}\Big( 2^{-js_1}\min\big\{1,\big( 2^j\ell(Q_k)\big)^M \big\}I_{k,j,s}^{\mathrm{in}/\mathrm{out}}(y) \Big)^2\bigg)^{{1}/{2}}\\
    &\qq\qq\q \times   \bigg( \sum_{P\in\DD} \Big( \big| \BB_P^2(f_2)\big| |P|^{-{1}/{2}}  \frac{\chi_{P_c}(x)}{\langle 2^j(x-\xx_P)\rangle^{s_2}}    \Big)^2\bigg)^{{1}/{2}} \\
    &\qq\qq\qq\times \sup_{j\in\bbz}\bigg( \sum_{P\in\DD_j}\Big(   \big| \BB_R^3(f_3)\big| |R|^{-{1}/{2}} \frac{1}{\langle 2^j(x-\xx_R)\rangle^{s_3}} \Big)^2\bigg)^{{1}/{2}}.
\end{align*}
The proof of the above claim is a repetition of the arguments used in the proof of  Lemmas \ref{keylemma3} and \ref{keylemma4}, so we omit the details.
 We now take
\begin{align*}
u_1^{\mathrm{in}/\mathrm{out}}(x)&:=   \sum_{k=0}^{\infty}|\la_k|\ell(Q_k)^{-{n}/{p_1}}\chi_{(Q_k^{***})^c}(x)| x-\xx_{Q_k}|^{-s_1}\\
&\qq\times\mathcal{M}\bigg[\bigg(\sum_{j\in\bbz}\Big(2^{-s_1j}\min{\{1,(2^j\ell(Q_k))^M\}}I_{k,j,s}^{\mathrm{in}/\mathrm{out}}(\cdot) \Big)^2 \bigg)^{{1}/{2}}\bigg](x)   \\
u_2(x)&:=        \bigg(\sum_{j\in\bbz} \sum_{P\in\DD_j} \Big( \big| \BB_P^2(f_2)\big| |P|^{-{1}/{2}} \frac{\chi_{P^c}(x)}{\langle 2^j(x-\xx_P)\rangle^{s_2}}\Big)^2   \bigg)^{{1}/{2}}    \\
u_3(x)&:=            \sup_{j\in\bbz}\bigg( \sum_{R\in\DD_j} \Big(\big| \BB_R^3(f_3) \big| |R|^{-{1}/{2}}\frac{1}{\langle 2^j(x-\xx_R)\rangle ^{s_3}}    \Big)^2 \bigg)^{{1}/{2}}     .
\end{align*}
Then it is obvious that  \eqref{keylemma5vconditions} and  \eqref{keylemma5est} hold.

\subsection{Proof of Lemma \ref{keylemma6}}

We choose $0<\epsilon<1$ such that
\begin{equation}\label{np1def}
N_{p_1}:=\big[{n}/{p_1}-n\big]\le {n}/{p_1}-n<\big[ {n}/{p_1}-n\big]+\epsilon<s-{3n}/{2}.
\end{equation}
We note that 
\begin{equation*}
2^l\lesssim |x-\xx_{Q_k}|^{-1} \q\text{ for }~ x\in B_k^l.
\end{equation*}
By using Lemma \ref{epsilon control} with the vanishing moment condition \eqref{vanishingreduction}, 
 we have
\begin{align*}
\big| \phi_l\ast \big( \UU_{1}(x,\cdot)\big)(x)\big|  & \lesssim \sum_{j\in\bbz}\sum_{P\in\DD_j}\sum_{R\in\DD_j}|b_P^2| |b_R^3|\chi_{P\cap R}(x) 2^{l(n+N_{p_1}+\epsilon)}\nonumber\\
&\qq\qq\times \int_{\bbrn}|y-\xx_{Q_k}|^{N_{p_1}+\epsilon}\big|T_{\sigma_j}\big(\La_ja_k,\psi^P,\theta^R\big)(y) \big| \; dy\nonumber\\
&\lesssim \frac{1}{|x-\xx_{Q_k}|^{n+N_{p_1}+\epsilon}}\sum_{j\in\bbz}\sum_{P\in\DD_j}\sum_{R\in\DD_j}|b_P^2| |b_R^3|\chi_{P\cap R}(x)\\
&\qq\qq\times \Big(\mathscr{K}_{N_{p_1}+\epsilon}^{j,\mathrm{in}}(Q_k,P,R)+\mathscr{K}_{N_{p_1}+\epsilon}^{j,\mathrm{out}}(Q_k,P,R) \Big)\nonumber
\end{align*}
where 
$$\KK_{N_{p_1}+\epsilon}^{j,\mathrm{in}}(Q_k,P,R):=\int_{Q_k^{**}} |y-\xx_{Q_k}|^{N_{p_1}+\epsilon}\big| T_{\sigma_j}\big( \La_ja_k,\psi^P,\theta^R\big)(y)\big| \; dy,$$
and
$$\KK_{N_{p_1}+\epsilon}^{j,\mathrm{out}}(Q_k,P,R):=\int_{(Q_k^{**})^c} |y-\xx_{Q_k}|^{N_{p_1}+\epsilon}\big| T_{\sigma_j}\big( \La_ja_k,\psi^P,\theta^R\big)(y)\big| \; dy.$$
Now the left-hand side of \eqref{keylemma6est} is dominated by $\mathscr{J}^{in}(x)+\mathscr{J}^{out}(x)$
\begin{align*}
\mathscr{J}^{\mathrm{in}/\mathrm{out}}(x)&:=\sum_{k=0}^{\infty}|\la_k|\chi_{(Q_{k}^{***})^c}(x)\frac{1}{|x-\xx_{Q_k}|^{n+N_{p_1}+\epsilon}}\\
&\qq\qq\qq\times \sum_{j\in\bbz}\sum_{P\in\DD_j}\sum_{R\in\DD_j}|b_P^2| |b_R^3|  \chi_{P\cap R}(x)\KK_{N_{p_1}+\epsilon}^{j,\mathrm{in}/ \mathrm{out}}(Q_k,P,R).
\end{align*}

To estimate $\mathscr{J}^{\mathrm{in}}$, we first see that
\begin{align*}
&\KK_{N_{p_1}+\epsilon}^{j,\mathrm{in}}(Q_k,P,R)\\
&\lesssim \ell(Q_k)^{N_{p_1}+\epsilon}|P|^{-{1}/{2}}|R|^{-{1}/{2}}\int_{y\in Q_k^{**}}\int_{(\bbrn)^3}\big| \sigma_j^{\vee}(y-z_1,z_2,z_3)\big| \big|\La_ja_k(z_1) \big| \; d\zzz dy\\
&\lesssim \ell(Q_k)^{N_{p_1}+\epsilon}|P|^{-{1}/{2}}|R|^{-{1}/{2}}\Vert \La_ja_k\Vert_{L^{1}(\bbrn)},
 \end{align*} using the Cauchy-Schwarz inequality with $s>{3n}/{2}$, 
 and thus
\begin{align}\label{jinest}
&\sum_{j\in\bbz}\sum_{P\in\DD_j}\sum_{R\in\DD_j} |b_P^2| |b_R^3| \chi_{P\cap R}(x)\KK_{N_{p_1}+\epsilon}^{j,\mathrm{in}}(Q_k,P,R)\nonumber\\
&\lesssim \ell(Q_k)^{N_{p_1}+\epsilon}\big\Vert \big\{ \La_ja_k\big\}_{j\in\bbz}\big\Vert_{L^1(\ell^2)}g^2\big(\big\{ b_P^2\big\}_{P\in\DD} \big)(x)g^{\infty}\big(\big\{ b_R^3\big\}_{R\in\DD} \big)(x)\nonumber\\
&\lesssim  |Q_k|^{-1/p_1} \ell(Q_k)^{n+N_{p_1}+\epsilon}g^2\big(\big\{ b_P^2\big\}_{P\in\DD} \big)(x)g^{\infty}\big(\big\{ b_R^3\big\}_{R\in\DD} \big)(x)
\end{align} by using the fact that
\begin{equation*}
\big\Vert \big\{ \La_ja_k\big\}_{j\in\bbz}\big\Vert_{L^1(\ell^2)}\sim \Vert a_k\Vert_{H^1(\bbrn)}\lesssim \ell(Q_k)^{-{n}/{p_1}+n}.
\end{equation*}

For the other term $\mathscr{J}^{out}$, 
we choose $s_1$ such that
\begin{equation}\label{npcondition1}
N_{p_1}+{n}/{2}+\epsilon<s_1<s-n,
\end{equation}
which is possible due to \eqref{np1def}, and  $s_2,s_3>{n}/{2}$ such that 
\begin{equation}\label{npcondition2}
s_1+n<s_1+s_2+s_3=s.
\end{equation}
We observe that for $y\in (Q_k^{**})^c$,
\begin{align*}
&\langle 2^j(y-\xx_{Q_k})\rangle^{s_1}\big| \La_ja_k(z_1)\big|\\
&\lesssim  \ell(Q_k)^{-{n}/{p_1}}\min{\big\{1,\big( 2^j\ell(Q_k)\big)^M \big\}}\langle 2^j(y-\xx_{Q_k})\rangle^{s_1} \\
&\qq\qq\qq\qq\times \Big(\chi_{Q_k^*}(z_1)+\chi_{(Q_k^*)^c}(z_1)\frac{(2^j\ell(Q_k))^n}{\langle 2^j(z_1-\xx_{Q_k})\rangle^{L_0}} \Big)\\
&\lesssim \ell(Q_k)^{-{n}/{p_1}}\min{\big\{1,\big( 2^j\ell(Q_k)\big)^M \big\}} \langle 2^j(y-z_1)\rangle^{s_1}\\
&\qq\qq\qq\qq\times \Big(\chi_{Q_k^*}(z_1)+\chi_{(Q_k^*)^c}(z_1)\frac{(2^j\ell(Q_k))^n}{\langle 2^j(z_1-\xx_{Q_k})\rangle^{L_0-s_1}} \Big)
\end{align*}
where Lemma \ref{technical lemma} is applied in the first inequality.
Here, $M$ and $L_0$ are sufficiently large numbers such that $L_0-s_1>n$ and $M-L_0+3n/2>0$.
By letting 
\begin{equation}\label{ajqkdef}
A_{j,Q_k}(z_1):=\chi_{Q_k^*}(z_1)+\chi_{(Q_k^*)^c}(z_1)\frac{(2^j\ell(Q_k))^n}{\langle 2^j(z_1-\xx_{Q_k})\rangle^{L_0-s_1}},
\end{equation}
we have
\begin{align}\label{tsigmajlaest}
&\big| T_{\sigma_j}\big( \La_ja_k,\psi^P,\theta^R\big)(y)\big|\chi_{(Q_k^{**})^c}(y)\nonumber\\
&\lesssim  \ell(Q_k)^{-{n}/{p_1}}\min \big\{ 1,\big( 2^j\ell(Q_k)\big)^M\big\}\frac{1}{\langle 2^j(y-\xx_{Q_k})\rangle^{s_1}}|P|^{-{1}/{2}} |R|^{-{1}/{2}}\\
&\qq\qq\qq\times \int_{(\bbrn)^3} \langle 2^j(y-z_1)\rangle^{s_1}\big| \sigma_j^{\vee}(y-z_1,z_2,z_3)\big| A_{j,Q_k}(z_1)\; d\zzz \nonumber
\end{align}
and the integral is, via the Cauchy-Schwarz inequality, less than
$$2^{-jn}\int_{\bbrn} |A_{j,Q_k}(z_1)|   \big\Vert \langle 2^j(y-z_1,z_2,z_3)\rangle^{s}\sigma_j^{\vee}(y-z_1,z_2,z_3)\big\Vert_{L^2(z_2,z_3)}    \;  dz_1.$$
This deduces that 
\begin{align*}
&\KK_{N_{p_1}+\epsilon}^{j,\mathrm{out}}(Q_k,P,R)\\
& \lesssim \ell(Q_k)^{-n/{p_1}}\min\big\{1,\big( 2^j\ell(Q_k)\big)^M \big\} |P|^{-{1}/{2}}|R|^{-{1}/{2}}2^{-js_1}2^{-jn} \int_{\bbrn} \big|A_{j,Q_k}(z_1)\big|\\
&\q\times \Big(\int_{(Q_k^{**})^c} \frac{1}{|y-\xx_{Q_k}|^{s_1-(N_{p_1}+\epsilon)}}      \big\Vert \langle 2^j(y-z_1,z_2,z_3)\rangle^{s}\sigma_j^{\vee}(y-z_1,z_2,z_3)\big\Vert_{L^2(z_2,z_3)}       \,   dy\Big) \; dz_1\\
&\lesssim  \ell(Q_k)^{-n/{p_1}}\min\big\{1,\big( 2^j\ell(Q_k)\big)^M \big\} |P|^{-{1}/{2}}|R|^{-{1}/{2}}2^{-js_1}2^{-jn} \Vert A_{j,Q_k}\Vert_{L^1(\bbrn)}\\
&\qq\qq\qq\qq\qq\times \bigg\Vert \frac{1}{|\cdot-\xx_{Q_k}|^{s_1-(N_{p_1}+\epsilon)}}\bigg\Vert_{L^2((Q_k^{**})^c)}\big\Vert \langle 2^j\ccdot\rangle^s\sigma_j^{\vee}\big\Vert_{L^2((\bbrn)^3)}\\
&\lesssim \ell(Q_k)^{-n/p_1+n+N_{p_1}+\epsilon}   \big( 2^j\ell(Q_k)\big)^{-(s_1-{n}/{2})}    \min\big\{1,\big( 2^j\ell(Q_k)\big)^{M-L_0+s_1+n} \big\}                     |P|^{-{1}/{2}}|R|^{-{1}/{2}}
\end{align*}
since $$\Vert A_{j,Q_k}\Vert_{L^1(\bbrn)}\lesssim \ell(Q_k)^n\Big(1+(2^j\ell(Q_k))^{-(L_0-s_1-n)} \Big) \q \text{ for }~ L_0-s_1>n.$$
Therefore, by using the Cauchy-Schwarz inequality
\begin{align}\label{joutest}
&\sum_{j\in\bbz}\sum_{P\in\DD_j}\sum_{R\in\DD_j} |b_P^2| |b_R^3| \chi_{P\cap R}(x)\KK_{N_{p_1}+\epsilon}^{j,\mathrm{out}}(Q_k,P,R)\nonumber\\
&\lesssim |Q_k|^{-{1}/{p_1}}\ell(Q_k)^{n+N_{p_1}+\epsilon}g^2\big(\big\{ b_P^2\big\}_{P\in\DD} \big)(x)g^{\infty}\big(\big\{ b_R^3\big\}_{R\in\DD} \big)(x)\nonumber\\
&\qq\times \bigg( \sum_{j\in\bbz}\Big(\big( 2^j\ell(Q_k)\big)^{-(s_1-{n}/{2})}    \min\big\{1,\big( 2^j\ell(Q_k)\big)^{M-L_0+s_1+n} \big\}  \Big)^2\bigg)^{{1}/{2}} \nonumber\\
&\lesssim |Q_k|^{-{1}/{p_1}}\ell(Q_k)^{n+N_{p_1}+\epsilon}g^2\big(\big\{ b_P^2\big\}_{P\in\DD} \big)(x)g^{\infty}\big(\big\{ b_R^3\big\}_{R\in\DD} \big)(x)
\end{align}  where the last inequality holds due to $s_1>{n}/{2}$ and $M-L_0+{3n}/{2}>0$.

In conclusion, the estimate \eqref{keylemma6est} can be derived from \eqref{jinest} and \eqref{joutest}, using the choices of
\begin{align*}
u_1(x)&:=     \sum_{k=0}^{\infty}|\la_k| |Q_k|^{-{1}/{p_1}} \chi_{(Q^{**})^c} \frac{\ell(Q_k)^{n+N_{p_1}+\epsilon}}{|x-\xx_{Q_k}|^{n+N_{p_1}+\epsilon}} ,  \\
u_2(x)&:=     g^2\big(\big\{ b_P^2\big\}_{P\in\DD} \big)(x),  \\
u_3(x)&:=      g^{\infty}\big(\big\{ b_R^3\big\}_{R\in\DD} \big)(x).
\end{align*}
It is obvious from \eqref{bp2} and \eqref{br3} that $\Vert u_2\Vert_{L^{p_2}(\bbrn)}, \Vert u_3\Vert_{L^{p_3}(\bbrn)}\lesssim 1$.
Furthermore,
\begin{align}\label{u1p1new}
\Vert u_1\Vert_{L^{p_1}(\bbrn)}&\lesssim \bigg(\sum_{k=0}^{\infty}|\la_k|^{p_1}|Q_k|^{-1}\int_{(Q_k^{***})^c}\frac{\ell(Q_k)^{(n+N_{p_1}+\epsilon)p_1}}{|x-\xx_{Q_k}|^{(n+N_{p_1}+\epsilon)p_1}} dx \bigg)^{{1}/{p_1}}\nonumber\\
&\lesssim \Big( \sum_{k=0}^{\infty}|\la_k|^{p_1}\Big)^{{1}/{p_1}}\lesssim 1.
\end{align}
This completes the proof.

\subsection{Proof of Lemma \ref{keylemma7}}
Choose $s_1$, $s_2$, and $s_3$ such that $s_1>{n}/{p_1}-{n}/{2}$, $s_2>{n}/{2}$, $s_3>{n}/{2}$, and $s=s_1+s_2+s_3$.
For $x\in B_k^l\cap (B_P^l)^c$ and $|x-y|\le 2^{-l}$, we have
\begin{equation}\label{xxxqk}
|x-\xx_{Q_k}|\le |x-\xx_{P}|\lesssim |y-\xx_P|.
\end{equation}
This implies
\begin{align*}
&\langle 2^j(x-\xx_{Q_k})\rangle^{s_1}\langle 2^j(x-\xx_P)\rangle^{s_2}\big| \phi_l\ast T_{\sigma_j}\big(\La_ja_k,\psi^P,\theta^R \big)(x)\big|\\
&\lesssim |R|^{-{1}/{2}}  2^{ln}\int_{|x-y|\le 2^{-l}}  \int_{(\bbrn)^3}\langle 2^j(y-\xx_P)\rangle^{s_1+s_2}\\
&\qq\qq\qq\qq\qq\times\big| \sigma_j^{\vee}(y-z_1,y-z_2,z_3)\big| \big| \La_ja_k(z_1)\big| \big| {\psi^P}(z_2)\big|\;d\zzz\; dy\\
&\le |R|^{-{1}/{2}}  2^{ln}\int_{|x-y|\le 2^{-l}}\int_{(\bbrn)^3} \langle 2^j(y-z_2)\rangle^{s_1+s_2} \\
&\qq\qq\qq\qq\qq\times \big| \sigma_j^{\vee}(y-z_1,y-z_2,z_3)\big| \big| \La_ja_k(z_1)\big| \big|\wt{\psi^P}(z_2) \big|\;d\zzz\; dy
\end{align*}
where
$$\wt{\psi^P}(z_2) :=\langle 2^j(z_2-\xx_P)\rangle^{s_1+s_2}\psi^P(z_2).$$
 By using the Cauchy-Schwarz inequality and Lemma \ref{almost orthogonality}, we obtain
\begin{align}\label{philu21}
&\big| \phi_l\ast \big( \UU_2^1(x,\cdot)\big)(x)\big| \nonumber\\
&\lesssim \sum_{j\in\bbz}\sum_{P\in\DD_j}\sum_{R\in\DD_j}|b_P^2| |b_R^3|\chi_{P^c}(x)\chi_R(x)\frac{1}{\langle 2^j(x-\xx_{Q_k})\rangle^{s_1}} \frac{1}{\langle 2^j(x-\xx_P)\rangle^{s_2}} |R|^{-{1}/{2}}\nonumber\\
&\qq \times 2^{ln}\int_{|x-y|\le 2^{-l}} \int_{(\bbrn)^3}\langle 2^j(y-z_2)\rangle^{s_1+s_2}\big|  \sigma_j^{\vee}(y-z_1,y-z_2,z_3) \big| \big| \La_ja_k(z_1)\big|      \;    d\zzz dy\nonumber\\
&\lesssim g^{\infty}\big( \big\{ b_R^3\big\}_{R\in\DD}\big)(x) \sum_{j\in\bbz}\frac{1}{\langle 2^j(x-\xx_{Q_k})\rangle^{s_1}} 2^{ln}\int_{|x-y|\le 2^{-l}} \int_{(\bbrn)^3} \langle 2^j(y-z_2)\rangle^{s_1+s_2}  \nonumber\\
&\qq  \times \big| \sigma_j^{\vee}(y-z_1,y-z_2,z_3)\big| \big| \La_ja_k(z_1)\big|\Big(\sum_{P\in\DD_j}|b_P^2|\frac{\chi_{(P)^c}(x)}{\langle 2^j(x-\xx_P)\rangle^{s_2}} \big|\wt{\psi^P}(z_3) \big| \Big)      \; d\zzz dy\nonumber\\
&\lesssim g^{\infty}\big( \big\{ b_R^3\big\}_{R\in\DD}\big)(x) \frac{1}{|x-\xx_{Q_k}|^{s_1}}\Big( \sum_{j\in\bbz} \big(2^{-js_1}\mathcal{M}J_{k,j,s}^1(x) \big)^2\Big)^{{1}/{2}}\\
&\qq\qq\qq\times\Big( \sum_{j\in\bbz}\sum_{P\in\DD_j} \Big( \big|\BB_P^2(f_2)\big||P|^{-{1}/{2}}\frac{\chi_{P^c}(x)}{\langle 2^j(x-\xx_P)\rangle^{s_2}}\Big)^2     \Big)^{{1}/{2}} \nonumber
\end{align}
where $\BB_P^2(f_2)$ is defined as in \eqref{bpf2} for some $L>n,s_2$, and 
\begin{equation}\label{jkjs1def}
J_{k,j,s}^1(y):= 2^{-jn}\int_{\bbrn} \big| \La_ja_k(z_1)\big|\Big\Vert \langle 2^j(y-z_1,z_2,z_3)\rangle^s\sigma_j^{\vee}(y-z_1,z_2,z_3) \Big\Vert_{L^2(z_2,z_3)}       dz_1.    
\end{equation}
Now we choose
\begin{align*}
u_1(x)&:=\sum_{k=0}^{\infty}|\la_k|\chi_{(Q_k^{***})^c}(x)\frac{1}{|x-\xx_{Q_k}|^{s_1}}\Big( \sum_{j\in\bbz}\big(2^{-js_1}\mathcal{M}J_{k,j,s}^1(x) \big)^2\Big)^{{1}/{2}},\\
u_2(x)&:=\bigg( \sum_{j\in\bbz}\sum_{P\in\DD_j} \Big( |\BB_P^2(f_2)||P|^{-{1}/{2}}\frac{\chi_{P^c}(x)}{\langle 2^j(x-\xx_P)\rangle^{s_2}}\Big)^2     \bigg)^{{1}/{2}},\\
u_3(x)&:=g^{\infty}\big( \big\{ b_R^3\big\}_{R\in\DD}\big)(x). 
\end{align*}
Clearly, \eqref{keylemma7est} holds and $\Vert u_2\Vert_{L^{p_2}(\bbrn)}, \Vert u_3\Vert_{L^{p_3}(\bbrn)} \lesssim 1$ due to Lemma \ref{lacunarylemma} and \eqref{br3}.
In addition,
\begin{equation*}
\Vert u_1\Vert_{L^{p_1}(\bbrn)}^{p_1}\lesssim  \sum_{k=0}^{\infty} |\la_k|^{p_1}\int_{(Q_k^{***})^c} |x-\xx_{Q_k}|^{-s_1p_1}\Big(\sum_{j\in\bbz}\big(2^{-s_1j}\mathcal{M}J_{k,j,s}^1(x) \big)^2 \Big)^{{p_1}/{2}}     dx 
\end{equation*}
and the integral is controlled by
\begin{align*}
&\big\Vert |\cdot-\xx_{Q_k}|^{-s_1p_1}\big\Vert_{L^{({2}/{p_1})'}((Q_k^{***})^c)}\bigg\Vert \Big(     \sum_{j\in\bbz}\big(2^{-s_1j}\mathcal{M}J_{k,j,s}^1(x) \big)^2       \Big)^{{p_1}/{2}}\bigg\Vert_{L^{{2}/{p_1}}(\bbrn)}\\
&\lesssim \ell(Q_k)^{-p_1(s_1-({n}/{p_1}-{n}/{2}))} \Big( \sum_{j\in\bbz}2^{-2js_1}\big\Vert J_{k,j,s}^1 \big\Vert_{L^2(\bbrn)}^2\Big)^{{p_1}/{2}}
\end{align*}
by using H\"older's inequality and the $L^2$ boundedness of $\mathcal{M}$.
It follows from Minkowski's inequality and Lemma \ref{technical lemma} that
\begin{align*}
\big\Vert J_{k,j,s}^{1}\big\Vert_{L^2(\bbrn)}&\lesssim 2^{-jn}\big\Vert \La_ja_k\big\Vert_{L^1(\bbrn)}\Big\Vert \langle 2^j\ccdot\rangle^s\big| \sigma_j^{\vee}\big|\Big\Vert_{L^2((\bbrn)^3)}\\
&\lesssim \ell(Q_k)^{-{n}/{p_1}+n}2^{{jn}/{2}}\min\big\{1,\big(2^j\ell(Q_k) \big)^M \big\}
\end{align*}
and this finally yields that
\begin{equation}\label{v1p1est}
\Vert u_1\Vert_{L^{p_1}(\bbrn)}\lesssim \Big( \sum_{k=0}^{\infty}|\la_k|^{p_1}\Big)^{{1}/{p_1}}\lesssim 1.
\end{equation}

\subsection{Proof of Lemma \ref{keylemma8}}
For $x\in B_k^l\cap B_P^l$,
\begin{equation}\label{2lboundaa}
2^l\lesssim |x-\xx_{Q_k}|^{-1}, |x-\xx_P|^{-1}.
\end{equation}
Since $s>{n}/{p_1}+{n}/{2}$, there exist $0<\epsilon_0,\epsilon_1<1$ such that 
$${n}/{p_1}+{n}/{p_2}<\big[ {n}/{p_1}+{n}/{p_2}\big]+\epsilon_0 \q \text{ and }\q  \big[ {n}/{p_1}+{n}/{p_2}\big]+\epsilon_0+\epsilon_1 <s-\big({n}/{2} -{n}/{p_2}\big).$$
Choose $t_1$ and $t_2$ satisfying $t_1>{n}/{p_1}$, $t_2>{n}/{p_2}$, and $t_1+t_2=\big[ {n}/{p_1}+{n}/{p_2}\big]+\epsilon_0$ and
let $N_0:=\big[ {n}/{p_1}+{n}/{p_2}\big]-n $.
Then Lemma \ref{epsilon control}, together with the vanishing moment condition \eqref{vanishingreduction}, and the estimate \eqref{2lboundaa} yield that
\begin{align*}
&\big| \phi_l\ast T_{\sigma_j}\big( \La_ja_k,\psi^P,\theta^R\big)(x)\big|\\
&\lesssim 2^{l(N_0+n+\epsilon_0)}\int_{\bbrn}|y-\xx_P |^{N_0+\epsilon_0}  \big| T_{\sigma_j} \big( \La_ja_k,\psi^P,\theta^R\big)(y)\big|  \;     dy\\
&\lesssim |R|^{-{1}/{2}}\frac{1}{|x-\xx_{Q_k}|^{t_1}}2^{-j(t_1-n)}\int_{\bbrn} \int_{(\bbrn)^3}\langle 2^j(y-z_2)\rangle^{N_0+\epsilon_0}\\
&\qq\qq\qq\times \big| \sigma_j^{\vee}(y-z_1,y-z_2,z_3)\big| \big| \La_ja_k(z_1)\big| \frac{| \wt{\psi^P}(z_2)|}{\langle 2^j(x-\xx_P)\rangle^{t_2}}     \;    d\zzz dy
  \end{align*}
  where
  $$\wt{\psi^P}(z_2):= \langle z_2-\xx_P\rangle^{N_0+\epsilon_0}\psi^P(z_2).$$
This deduces
\begin{align}\label{philu22est}
&\big| \phi_l\ast \big( \UU_2^2(x,\cdot)\big)(x)\big| \nonumber\\
&\lesssim g^{\infty}\big(\big\{ b_R^3\big\}_{R\in\DD} \big)(x)\frac{1}{|x-\xx_{Q_k}|^{t_1}}\sum_{j\in\bbz}2^{-j(t_1-n)}\int_{\bbrn}  \int_{(\bbrn)^3}   \langle 2^j(y-z_2)\rangle^{N_0+\epsilon_0} \\
&\qq\times \big| \sigma_j(y-z_1,y-z_2,z_3)\big| \big| \La_ja_k(z_1)\big| \Big( \sum_{P\in\DD_j} |b_P^2|\frac{\chi_{P^c}(x)}{\langle 2^j(x-\xx_P)\rangle^{t_2}}\big| \wt{\psi^P}(z_2)\big|    \Big)      \; d\zzz     dy.\nonumber
\end{align}
Using H\"older's inequality with $\frac{1}{2}+\frac{1}{(1/p_2'-1/2)^{-1}}+\frac{1}{p_2}=1$ and Lemma \ref{almost orthogonality}, we see that
\begin{align*}
&\int_{\bbrn} \langle 2^j(y-z_2)\rangle^{N_0+\epsilon_0} \big| \sigma_j(y-z_1,y-z_2,z_3)\big|    \Big( \sum_{P\in\DD_j} |b_P^2|\frac{\chi_{P^c}(x)}{\langle 2^j(x-\xx_P)\rangle^{t_2}}\big| \wt{\psi^P}(z_2)\big|    \Big) \;    dz_2\\
&\le \big\Vert \langle 2^jz_2\rangle^{s-n-\epsilon_1}\sigma_j^{\vee}(y-z_1,z_2,z_3)\big\Vert_{L^2(z_2)}\big\Vert \langle 2^j\cdot \rangle^{-(s-t_1-t_2-\epsilon_1)}\big\Vert_{L^{({1}/{p_2'}-{1}/{2})^{-1}}(\bbrn)}\\
&\qq\qq\times \Big\Vert   \sum_{P\in\DD_j} |b_P^2|\frac{\chi_{P^c}(x)}{\langle 2^j(x-\xx_P)\rangle^{t_2}}\big| \wt{\psi^P}(z_2)\big|     \Big\Vert_{L^{p_2}(\bbrn)}\\
&\lesssim 2^{-\frac{jn}{2}}\big\Vert \langle 2^jz_2\rangle^{s-n-\epsilon_1}\sigma_j^{\vee}(y-z_1,z_2,z_3)\big\Vert_{L^2(z_2)}\\
&\qq\qq\times \bigg( \sum_{P\in\DD_j}\Big( |\BB_P^2(f_2)||P|^{-{1}/{2}}\frac{\chi_{P^c}(x)}{\langle 2^j(x-\xx_P)\rangle^{t_2}}\Big)^{p_2}  \bigg)^{{1}/{p_2}}
\end{align*}
because $s-n-\epsilon_1=s-t_1-t_2-\epsilon_1+N_0+\epsilon_0$, $s-t_1-t_2-\epsilon_1>n(1/p_2'-1/2)$.
This shows that the integral in the right-hand side of \eqref{philu22est} is dominated by a constant times
\begin{align*}
&\Vert \La_ja_k\Vert_{L^1(\bbrn)} \bigg( \sum_{P\in\DD_j}\Big( |\BB_P^2(f_2)||P|^{-{1}/{2}}\frac{\chi_{P^c}(x)}{\langle 2^j(x-\xx_P)\rangle^{t_2}}\Big)^{p_2}  \bigg)^{{1}/{p_2}} \\
&\qq\qq\qq\qq\times 2^{-{jn}/{2}}\int_{(\bbrn)^2}  \big\Vert \langle 2^jz_2\rangle^{s-n-\epsilon_1}\sigma_j^{\vee}(y,z_2,z_3)\big\Vert_{L^2(z_2)} \;dy\; dz_3\\
&\lesssim \ell(Q_k)^{-{n}/{p_1}+n}\min\big\{1,\big( 2^j\ell(Q_k)\big)^M \big\} \bigg( \sum_{P\in\DD_j}\Big( |\BB_P^2(f_2)||P|^{-{1}/{2}}\frac{\chi_{P^c}(x)}{\langle 2^j(x-\xx_P)\rangle^{t_2}}\Big)^{p_2}  \bigg)^{{1}/{p_2}}
\end{align*}
where $\BB_P^2(f_2)$ is defined as in \eqref{bpf2} and $M$ is sufficiently large.
Consequently,
\begin{align}\label{philu22est2}
&\big| \phi_l\ast \big( \UU_2^2(x,\cdot)\big)(x)\big| \nonumber\\
& \lesssim g^{\infty}\big(\big\{ b_R^3\big\}_{R\in\DD} \big)(x) \;\sup_{j\in\bbz} \bigg( \sum_{P\in\DD_j}\Big( |\BB_P^2(f_2)||P|^{-{1}/{2}}\frac{\chi_{P^c}(x)}{\langle 2^j(x-\xx_P)\rangle^{t_2}}\Big)^{p_2}  \bigg)^{{1}/{p_2}}\nonumber\\
&\qq\qq\qq\times \frac{1}{|x-\xx_{Q_k}|^{t_1}}\ell(Q_k)^{-{n}/{p_1}+n}\sum_{j\in\bbz}  2^{-j(t_1-n)} \min\big\{1,\big( 2^j\ell(Q_k)\big)^M \big\}\nonumber\\
&\lesssim  |Q_k|^{-1/p_1} \frac{\ell(Q_k)^{t_1}}{|x-\xx_{Q_k}|^{t_1}} \bigg(  \sum_{j\in\bbz} \sum_{P\in\DD_j}\Big( |\BB_P^2(f_2)||P|^{-{1}/{2}}\frac{\chi_{P^c}(x)}{\langle 2^j(x-\xx_P)\rangle^{t_2}}\Big)^{p_2}  \bigg)^{{1}/{p_2}}\\
&\qq\qq\qq\times g^{\infty}\big(\big\{ b_R^3\big\}_{R\in\DD} \big)(x).\nonumber
\end{align}
Now we are done with
\begin{align*}
u_1(x)&:= \sum_{k=0}^{\infty}|\la_k| |Q_k|^{-{1}/{p_1}}\frac{\ell(Q_k)^{t_1}}{|x-\xx_{Q_k}|^{t_1}}\chi_{(Q_k^{***})^c}(x)\\
u_2(x)&:=   \bigg(\sum_{j\in\bbz} \sum_{P\in\DD_j}\Big( |\BB_P^2(f_2)||P|^{-{1}/{2}}\frac{\chi_{P^c}(x)}{\langle 2^j(x-\xx_P)\rangle^{t_2}}\Big)^{p_2}  \bigg)^{{1}/{p_2}}\\
u_3(x)&:= g^{\infty}\big(\big\{ b_R^3\big\}_{R\in\DD} \big)(x)
\end{align*}
as $\Vert u_{\ii}\Vert_{L^{p_{\ii}}(\bbrn)}\lesssim 1$, $\ii=1,2,3$, follow from
Lemma \ref{lacunarylemma}, \eqref{br3}, and the argument that led to \eqref{u1p1new} with $t_1>n/p_1$.

\subsection{Proof of Lemma \ref{keylemma9}}
Let $s_1$, $s_2$, and $s_3$ satisfy $s_1>{n}/{p_1}-{n}/{2}$, $s_2>{n}/{2}$, $s_3>{n}/{2}$, and $s=s_1+s_2+s_3$.
By mimicking the argument that led to \eqref{philu21} with
$$|x-\xx_{Q_k}|\le |x-\xx_R|\lesssim |y-\xx_R|$$
for $x\in B_k^l\cap (B_R^l)^c$ and $|x-y|\le 2^{-l}$, instead of \eqref{xxxqk}, we can prove
\begin{align*}
\big| \phi_l\ast \big(\UU_{3}^{1}(x,\cdot)\big)(x) \big|&\lesssim \frac{1}{|x-\xx_{Q_k}|^{s_1}}\Big(  \sum_{j\in\bbz} \big(2^{-js_1}\mathcal{M}J_{k,j,s}^1(x)   \big)^2\Big)^{{1}/{2}} \; g^2\big(\big\{ b_P^2\big\}_{P\in\DD} \big)(x)\\
&\qq\qq\times \sup_{j\in\bbz}\bigg( \sum_{R\in\DD_j} \Big( \big|\BB_R^3(f_3)\big||R|^{-{1}/{2}}\frac{\chi_{R^c}(x)}{\langle 2^j(x-\xx_R)\rangle^{s_3}}\Big)^2\bigg)^{{1}/{2}}
\end{align*}
where $J_{k,j,s}^1$ and $\BB_R^3(f_3)$ are defined as in \eqref{jkjs1def} and \eqref{br3f3def} for some $L>n,s_3$.

Now let
\begin{align*}
u_1(x)&:=    \sum_{k=0}^{\infty}|\la_k|\chi_{(Q_k^{***})^c}(x)\frac{1}{|x-\xx_{Q_k}|^{s_1}}\Big( \sum_{j\in\bbz}\big(2^{-js_1}\mathcal{M}J_{k,j,s}^1(x) \big)^2\Big)^{{1}/{2}}      \\
u_2(x)&:=g^2\big( \big\{ b_P^2\big\}_{P\in\DD}\big)(x)\\
u_3(x)&:= \sup_{j\in\bbz}\bigg( \sum_{P\in\DD_j}\Big( |\BB_R^3(f_3)| |R|^{-{1}/{2}}\frac{\chi_{R^c}(x)}{\langle 2^j(x-\xx_R)\rangle^{s_3}}\Big)^2\bigg)^{{1}/{2}}.
\end{align*}
Then the estimate \eqref{keylemma9est} is clear and 
it follows from \eqref{v1p1est}, \eqref{bp2}, and Lemma \ref{nonlacunarylemma} that \eqref{keylemma9conditions} holds.

\subsection{Proof of Lemma \ref{keylemma10}}

Let $0<\epsilon_0,\epsilon_1<1$ satisfy
$${n}/{p_1}+{n}/{p_3}<\big[ {n}/{p_1}+{n}/{p_3}\big]+\epsilon_0 \q\text{ and }\q \big[ {n}/{p_1}+{n}/{p_3}\big]+\epsilon_0+\epsilon_1<s-\big( {n}/{2}-{n}/{p_3}\big)$$
 and select $t_1,t_3$ so that $t_1>{n}/{p_1}$, $t_3>{n}/{p_3}$, and $t_1+t_3=\big[{n}/{p_1}+{n}/{p_3} \big]+\epsilon_0$.
Let $N_0:= \big[ {n}/{p_1}+{n}/{p_3}\big]-n$ and $B_R^3(f_3)$ be defined as in \eqref{br3f3def}.
Then, as the counterpart of \eqref{philu22est2}, we can get
\begin{align*}
\big| \phi_l\ast \big(\UU_{3}^{2}(x,\cdot)\big)(x)\big|&\lesssim |Q_k|^{-1/p_1} \frac{\ell(Q_k)^{t_1}}{|x-\xx_{Q_k}|^{t_1}}g^{2}\big( \big\{b_P^2 \big\}_{P\in\DD}\big)(x)\\
&\qq\qq\times  \sup_{j\in\bbz}\bigg( \sum_{R\in\DD_j}\Big( |\BB_R^3(f_3)||R|^{-{1}/{2}}\frac{\chi_{R^c}(x)}{\langle 2^j(x-\xx_R)\rangle^{t_3}}\Big)^{p_3}\bigg)^{{1}/{p_3}}
\end{align*}
where the embedding $\ell^2\hookrightarrow \ell^{\infty}$ is applied.
By taking
\begin{align*}
u_1(x)&:=     \sum_{k=0}^{\infty}|\la_k| |Q_k|^{-{1}/{p_1}}\frac{\ell(Q_k)^{t_1}}{|x-\xx_{Q_k}|^{t_1}}\chi_{(Q_k^{***})^c}(x)       \\
u_2(x)&:= g^2\big( \big\{ b_P^2\big\}_{P\in\DD}\big)(x),     \\
u_3(x)&:= \sup_{j\in\bbz}\bigg( \sum_{R\in\DD_j}\Big( |\BB_R^3(f_3)||R|^{-{1}/{2}}\frac{\chi_{R^c}(x)}{\langle 2^j(x-\xx_R)\rangle^{t_3}}\Big)^{p_3}\bigg)^{{1}/{p_3}},
\end{align*}
we obtain the inequality \eqref{keylemma10conditions} and  \eqref{keylemma10est}.

\subsection{Proof of Lemma \ref{keylemma11}}
The proof is almost same as that of Lemmas \ref{keylemma7} and \ref{keylemma9}.
Let $s_1$, $s_2$, and $s_3$ be numbers such that $s_1>{n}/{p_1}-{n}/{2}$, $s_2>{n}/{2}$, $s_3>{n}/{2}$, and $s=s_1+s_2+s_3$.
We claim that for $\eta=1,2,3$,
\begin{align}\label{eta123claim}
\big| \phi_l\ast \big( \UU_4^{\eta}(x,\cdot)\big)(x)\big| &\lesssim \frac{1}{|x-\xx_{Q_k}|^{s_1}} \Big( \sum_{j\in\bbz}\big( 2^{-js_1}\mathcal{M}J_{k,j,s}^1(x)\big)^2\Big)^{{1}/{2}}\\
&\q\times\bigg( \sum_{j\in\bbz}\sum_{P\in\DD_j} \Big(|\BB_P^2(f_2)| |P|^{-{1}/{2}} \frac{\chi_{P^c}(x)}{\langle 2^j(x-\xx_P)\rangle^{s_2}}\Big)^2   \bigg)^{{1}/{2}}\nonumber\\
&\qq\times \sup_{j\in\bbz}\bigg(\sum_{R\in\DD_j} \Big( |\BB_R^3(f_3)| |R|^{-{1}/{2}}\frac{\chi_{R^c}(x)}{\langle 2^j(x-\xx_R)\rangle^{s_3}}\Big)^2\bigg)^{{1}/{2}}\nonumber
\end{align}
where $J_{k,j,s}^1$, $\BB_P^2(f_2)$, and $\BB_R^3(f_3)$ are defined as in \eqref{jkjs1def}, \eqref{bpf2}, and \eqref{br3f3def}, respectively.
Then we have \eqref{keylemma11est} with the choice
\begin{align*}
u_1(x)&:=    \sum_{k=0}^{\infty}|\la_k|\chi_{(Q_k^{***})^c}(x)\frac{1}{|x-\xx_{Q_k}|^{s_1}}\Big( \sum_{j\in\bbz}\big(2^{-js_1}\mathcal{M}J_{k,j,s}^1(x) \big)^2\Big)^{{1}/{2}},        \\
u_2(x)&:=     \bigg( \sum_{j\in\bbz}\sum_{P\in\DD_j} \Big(|\BB_P^2(f_2)| |P|^{-{1}/{2}} \frac{\chi_{P^c}(x)}{\langle 2^j(x-\xx_P)\rangle^{s_2}}\Big)^2   \bigg)^{{1}/{2}},    \\
u_3(x)&:=     \sup_{j\in\bbz}\bigg( \sum_{P\in\DD_j}\Big( |\BB_R^3(f_3)| |R|^{-{1}/{2}}\frac{\chi_{R^c}(x)}{\langle 2^j(x-\xx_R)\rangle^{s_3}}\Big)^2\bigg)^{{1}/{2}}   .
\end{align*}
The estimates for $u_1,u_2,u_3$ follow from \eqref{v1p1est}, Lemma \ref{lacunarylemma}, and Lemma \ref{nonlacunarylemma}.

Now we return to the proof of \eqref{eta123claim}.
For $x\in B_k^l\cap (B_P^l)^c\cap (B_R^l)^c$ and $|x-y|\le 2^{-l}$, we have
\begin{equation}\label{xxqkxxp}
|x-\xx_{Q_k}|\le |x-\xx_P|\lesssim |y-\xx_P| \q \text{and} \q |x-\xx_R|\lesssim |y-\xx_R|.
\end{equation}
Then we have
\begin{align*}
&\langle 2^j(x-\xx_{Q_k})\rangle^{s_1}\langle 2^j(x-\xx_P)\rangle^{s_2}\langle 2^j(x-\xx_R)\rangle^{s_3}\big| \phi_l\ast T_{\sigma_j}\big( \La_ja_k,\psi^P,\theta^R\big)(x)\big|\\
&\lesssim 2^{ln}\int_{|x-y|\le 2^{-l}} \int_{(\bbrn)^3}\langle 2^j(y-z_2)\rangle^{s_1+s_2}\langle 2^j(y-z_3)\rangle^{s_3}\big|\sigma_j^{\vee}(y-z_1,y-z_2,y-z_3) \big| \\
&\qq\qq\times |\La_ja_k(z_1)|\big|\wt{\psi^P}(z_2)\big| \big|\wt{\theta^R}(z_3)\big| \; d\zzz      dy
\end{align*}
where 
\begin{align*}
\wt{\psi^P}(z_2)&:=\langle 2^j(z_2-\xx_P)\rangle^{s_1+s_2}\psi^P(z_2),\\
\wt{\theta^R}(z_3)&:=\langle 2^j(z_3-\xx_R)\rangle^{s_3}\theta^R(z_3).
\end{align*}
Now, using the method similar to that used in the proof of \eqref{philu21}, we obtain \eqref{eta123claim} for $\eta=1$.

For the case $\eta=2$, we use the fact, instead of \eqref{xxqkxxp}, that for $x\in B_k^l\cap (B_P^l)^c\cap B_R^l$ and $|x-y|\le 2^{-l}$,
$$|x-\xx_{Q_k}|, |x-\xx_R|\le |x-\xx_P|\lesssim |y-\xx_P|.$$
This shows that
\begin{align*}
&\langle 2^j(x-\xx_{Q_k})\rangle^{s_1}\langle 2^j(x-\xx_P)\rangle^{s_2}\langle 2^j(x-\xx_R)\rangle^{s_3}\big| \phi_l\ast T_{\sigma_j}\big( \La_ja_k,\psi^P,\theta^R\big)(x)\big|\\
&\lesssim 2^{ln}\int_{|x-y|\le 2^{-l}} \int_{(\bbrn)^3}\langle 2^j(y-z_2)\rangle^{s}\big|\sigma_j^{\vee}(y-z_1,y-z_2,y-z_3) \big| \\
&\qq\qq\qq\times  |\La_ja_k(z_1)|\big|{\wt{\psi^P}}(z_2)\big| \big|{\theta^R}(z_3)\big| \; d\zzz      dy
\end{align*}
where $$\wt{\psi^P}(z_2):=\langle 2^j(z_2-\xx_P)\rangle^{s}\psi^P(z_2),$$
and then \eqref{eta123claim} for $\eta=2$ follows. 

Similarly, we can prove that for $x\in B_k^l\cap B_P^l\cap (B_R^l)^c$ and $|x-y|\le 2^{-l}$,
\begin{align*}
&\langle 2^j(x-\xx_{Q_k})\rangle^{s_1}\langle 2^j(x-\xx_P)\rangle^{s_2}\langle 2^j(x-\xx_R)\rangle^{s_3}\big| \phi_l\ast T_{\sigma_j}\big( \La_ja_k,\psi^P,\theta^R\big)(x)\big|\\
&\lesssim 2^{ln}\int_{|x-y|\le 2^{-l}} \int_{(\bbrn)^3}\langle 2^j(y-z_3)\rangle^{s}\big|\sigma_j^{\vee}(y-z_1,y-z_2,y-z_3) \big| \\
&\qq\qq\qq\times  |\La_ja_k(z_1)|\big|{\psi^P}(z_2)\big| \big|\wt{\theta^R}(z_3)\big| \; d\zzz      dy
\end{align*}
where $$\wt{\theta^R}(z_3):=\langle 2^j(z_3-\xx_R)\rangle^{s}\theta^R(z_3).$$
This proves \eqref{eta123claim} for $\eta=3$.

\subsection{Proof of Lemma \ref{keylemma12}}

We first note that
\begin{equation}\label{2lbound}
2^l\lesssim |x-\xx_{Q_k}|^{-1}, |x-\xx_P|^{-1}, |x-\xx_R|^{-1}
\end{equation} for $x\in B_k^l\cap B_P^l\cap B_R^l$.
Since ${n}/{p}<s-\big({n}/{2}-{n}/{p_2}-{n}/{p_3} \big)$, there exist $0<\epsilon_0,\epsilon_1<1$ such that 
$${n}/{p}<\big[ {n}/{p}\big]+\epsilon_0 \q   \text{ and }\q       \big[ {n}/{p}\big]+\epsilon_0+\epsilon_1<s-\big({n}/{2} -{n}/{p_2}-{n}/{p_3}\big).$$
Choose $t_1$, $t_2$, and $t_3$ satisfying $t_1>{n}/{p_1}$, $t_2>{n}/{p_2}$, $t_3>{n}/{p_3}$, and $t_1+t_2+t_3=\big[ {n}/{p}\big]+\epsilon_0$ and let $N_0:=\big[ {n}/{p}\big]-n $.
Then it follows from Lemma \ref{epsilon control} and the estimate \eqref{2lbound} that
\begin{align*}
&\big| \phi_l\ast T_{\sigma_j}\big( \La_ja_k,\psi^P,\theta^R\big)(x)\big|\\
&\lesssim 2^{l(N_0+n+\epsilon_0)}\int_{\bbrn}|y-\xx_P |^{N_0+\epsilon_0}  \big| T_{\sigma_j} \big( \La_ja_k,\psi^P,\theta^R\big)(y)\big|   \;    dy\\
&\lesssim \frac{1}{|x-\xx_{Q_k}|^{t_1}}2^{-j(t_1-n)}\int_{\bbrn} \int_{(\bbrn)^3}\langle 2^j(y-z_2)\rangle^{N_0+\epsilon_0}\big| \sigma_j^{\vee}(y-z_1,y-z_2,y-z_3)\big| \\
&\qq\qq\qq\qq\qq\qq\times \big| \La_ja_k(z_1)\big| \frac{| \wt{\psi^P}(z_2)|}{\langle 2^j(x-\xx_P)\rangle^{t_2}}\frac{|\theta^R(z_3)|}{\langle 2^j(x-\xx_R)\rangle^{t_3}}       \;  d\zzz dy
  \end{align*}
  where
  $$\wt{\psi^P}(z_2):= \langle z_2-\xx_P\rangle^{N_0+\epsilon_0}\psi^P(z_2).$$
This deduces that
\begin{align}\label{philu44est}
&\big| \phi_l\ast \big( \UU_4^4(x,\cdot)\big)(x)\big| \nonumber\\
&\lesssim \frac{1}{|x-\xx_{Q_k}|^{t_1}}\sum_{j\in\bbz}2^{-j(t_1-n)}\int_{\bbrn}  \int_{(\bbrn)^3}   \langle 2^j(y-z_2)\rangle^{N_0+\epsilon_0} \big| \sigma_j(y-z_1,y-z_2,z_3)\big|  \big| \La_ja_k(z_1)\big|\\
&\qq\times  \bigg( \sum_{P\in\DD_j} |b_P^2|\frac{\chi_{P^c}(x)}{\langle 2^j(x-\xx_P)\rangle^{t_2}}\big| \wt{\psi^P}(z_2)\big|    \bigg) 
\bigg( \sum_{R\in\DD_j} |b_R^3|\frac{\chi_{R^c}(x)}{\langle 2^j(x-\xx_R)\rangle^{t_3}}\big| {\theta^R}(z_3)\big|    \bigg)      \;   d\zzz     dy. \nonumber
\end{align}
Since $s-\big[{n}/{p} \big]+{n}/{2}-\epsilon_0-\epsilon_1>\big({n}/{2}-{n}/{p_2} \big)+\big({n}/{2}-{n}/{p_3}\big)$, there exist $\mu_2$ and $\mu_3$ such that $\mu_2>{n}/{2}-{n}/{p_2}$, $\mu_3>{n}/{2}-{n}/{p_3}$, and $\mu_1+\mu_2=s-\big[ {n}/{p}\big]+{n}/{2}-\epsilon_0-\epsilon_1$.
Using H\"older's inequality with
 $$\frac{1}{2}+\frac{1}{(1/p_2'-1/2)^{-1}}+\frac{1}{p_2}=\frac{1}{2}+\frac{1}{(1/p_3'-1/2)^{-1}}+\frac{1}{p_3}=1,$$ we have
\begin{align*}
&\int_{(\bbrn)^2} \langle 2^j(y-z_2)\rangle^{N_0+\epsilon_0} \big| \sigma_j(y-z_1,y-z_2,y-z_3)\big| \\
&\qq\times   \bigg( \sum_{P\in\DD_j} |b_P^2|\frac{\chi_{P^c}(x)}{\langle 2^j(x-\xx_P)\rangle^{t_2}}\big| \wt{\psi^P}(z_2)\big|    \bigg) \bigg( \sum_{R\in\DD_j} |b_R^3|\frac{\chi_{R^c}(x)}{\langle 2^j(x-\xx_R)\rangle^{t_3}}\big| {\theta^R}(z_3)\big|    \bigg)    \;   dz_2 dz_3\\
&\le \big\Vert \langle 2^jz_2\rangle^{N_0+\epsilon_0+\mu_2}\langle 2^jz_3\rangle^{\mu_3}\sigma_j^{\vee}(y-z_1,z_2,z_3)\big\Vert_{L^2(z_2,z_3)}\big\Vert \langle 2^j\cdot \rangle^{-\mu_2}\big\Vert_{L^{({1}/{p_2'}-{1}/{2})^{-1}}(\bbrn)}\\
&\qq\times \big\Vert \langle 2^j\cdot \rangle^{-\mu_3}\big\Vert_{L^{({1}/{p_3'}-{1}/{2})^{-1}}(\bbrn)}\bigg\Vert   \sum_{P\in\DD_j} |b_P^2|\frac{\chi_{P^c}(x)}{\langle 2^j(x-\xx_P)\rangle^{t_2}}\big| \wt{\psi^P}(z_2)\big|     \bigg\Vert_{L^{p_2}(z_2)}\\
&\qq\qq\times \bigg\Vert   \sum_{R\in\DD_j} |b_R^3|\frac{\chi_{R^c}(x)}{\langle 2^j(x-\xx_R)\rangle^{t_3}}\big| {\theta^R}(z_3)\big|     \bigg\Vert_{L^{p_2}(z_3)}
\end{align*}
and then Lemma \ref{almost orthogonality} yields that the preceding expression is less than a constant times
\begin{align*}
& 2^{-jn}\big\Vert \langle 2^j(z_2,z_3)\rangle^{N_0+\epsilon_0+\mu_1+\mu_2}\sigma_j^{\vee}(y-z_1,z_2,z_3)\big\Vert_{L^2(z_2,z_3)}\\
&\qq\times \bigg( \sum_{P\in\DD_j}\Big( |\BB_P^2(f_2)||P|^{-{1}/{2}}\frac{\chi_{P^c}(x)}{\langle 2^j(x-\xx_P)\rangle^{t_2}}\Big)^{p_2}  \bigg)^{{1}/{p_2}}\\
&\qq\qq\times \bigg( \sum_{R\in\DD_j}\Big( |\BB_R^3(f_3)||R|^{-{1}/{2}}\frac{\chi_{R^c}(x)}{\langle 2^j(x-\xx_R)\rangle^{t_3}}\Big)^{p_3}  \bigg)^{{1}/{p_3}}
\end{align*}
because  $\mu_2>n(1/p_2'-1/2)$ and $\mu_3>n(1/p_3'-1/2)$, where $\BB_P^2(f_2)$ and $\BB_R(f_3)$ are defined as in \eqref{bpf2} and \eqref{br3f3def}.

Now the integral in the right-hand side of \eqref{philu44est} is dominated by a constant times
\begin{align*}
&2^{-{jn}}\Vert \La_ja_k\Vert_{L^1(\bbrn)} \int_{\bbrn }  \big\Vert \langle 2^j(z_2,z_3)\rangle^{s-{n}/{2}-\epsilon_1}\sigma_j^{\vee}(y,z_2,z_3)\big\Vert_{L^2(z_2,z_3)} dy \\
&\qq\times  \bigg( \sum_{P\in\DD_j}\Big( |\BB_P^2(f_2)||P|^{-{1}/{2}}\frac{\chi_{P^c}(x)}{\langle 2^j(x-\xx_P)\rangle^{t_2}}\Big)^{p_2}  \bigg)^{{1}/{p_2}} \\
&\qq\qq\times\bigg( \sum_{R\in\DD_j}\Big( |\BB_R^3(f_3)||R|^{-{1}/{2}}\frac{\chi_{R^c}(x)}{\langle 2^j(x-\xx_R)\rangle^{t_3}}\Big)^{p_3}  \bigg)^{{1}/{p_3}}   \\
\end{align*}
and this is no more than
\begin{align*}
& \ell(Q_k)^{-n/p_1+n}\min\big\{1,\big( 2^j\ell(Q_k)\big)^M \big\} \\
&\qq\times \bigg( \sum_{P\in\DD_j}\Big( |\BB_P^2(f_2)||P|^{-{1}/{2}}\frac{\chi_{P^c}(x)}{\langle 2^j(x-\xx_P)\rangle^{t_2}}\Big)^{p_2}  \bigg)^{{1}/{p_2}}\\
&\qq\qq\times \bigg( \sum_{R\in\DD_j}\Big( |\BB_R^3(f_3)||R|^{-{1}/{2}}\frac{\chi_{R^c}(x)}{\langle 2^j(x-\xx_R)\rangle^{t_3}}\Big)^{p_3}  \bigg)^{{1}/{p_3}} 
\end{align*}
where $N_0+\epsilon_0+\mu_2+\mu_3=s-\frac{n}{2}-\epsilon_1$.
Hence, it follows that
\begin{align*}
&\big| \phi_l\ast \big( \UU_4^4(x,\cdot)\big)(x)\big|\\
& \lesssim  \frac{1}{|x-\xx_{Q_k}|^{t_1}}\ell(Q_k)^{-{n}/{p_1}+n}\sum_{j\in\bbz}  2^{-j(t_1-n)} \min\big\{1,\big( 2^j\ell(Q_k)\big)^M \big\}\\
&\qq\times \sup_{j\in\bbz} \bigg( \sum_{P\in\DD_j}\Big( |\BB_P^2(f_2)||P|^{-{1}/{2}}\frac{\chi_{P^c}(x)}{\langle 2^j(x-\xx_P)\rangle^{t_2}}\Big)^{p_2}  \bigg)^{{1}/{p_2}} \\
&\qq\qq\times \sup_{j\in\bbz} \bigg( \sum_{R\in\DD_j}\Big( |\BB_R^3(f_3)||R|^{-{1}/{2}}\frac{\chi_{R^c}(x)}{\langle 2^j(x-\xx_R)\rangle^{t_3}}\Big)^{p_3}  \bigg)^{{1}/{p_3}}\\
&\lesssim  |Q_k|^{-1/p_1} \frac{\ell(Q_k)^{t_1}}{|x-\xx_{Q_k}|^{t_1}} \bigg(  \sum_{j\in\bbz} \sum_{P\in\DD_j}\Big( |\BB_P^2(f_2)||P|^{-{1}/{2}}\frac{\chi_{P^c}(x)}{\langle 2^j(x-\xx_P)\rangle^{t_2}}\Big)^{p_2}  \bigg)^{{1}/{p_2}}\\
&\qq\qq\qq\qq\qq\times \sup_{j\in\bbz} \bigg( \sum_{R\in\DD_j}\Big( |\BB_R^3(f_3)||R|^{-{1}/{2}}\frac{\chi_{R^c}(x)}{\langle 2^j(x-\xx_R)\rangle^{t_3}}\Big)^{p_3}  \bigg)^{{1}/{p_3}}.
\end{align*}
Now let
\begin{align*}
u_1(x)&:=  \sum_{k=0}^{\infty}|\la_k| |Q_k|^{-{1}/{p_1}}\frac{\ell(Q_k)^{t_1}}{|x-\xx_{Q_k}|^{t_1}}\chi_{(Q_k^{***})^c}(x),       \\
u_2(x)&:=     \bigg(\sum_{j\in\bbz} \sum_{P\in\DD_j}\Big( |\BB_P^2(f_2)||P|^{-{1}/{2}}\frac{\chi_{P^c}(x)}{\langle 2^j(x-\xx_P)\rangle^{t_2}}\Big)^{p_2}  \bigg)^{{1}/{p_2}},     \\
u_3(x)&:=     \sup_{j\in\bbz} \bigg( \sum_{R\in\DD_j}\Big( |\BB_R^3(f_3)||R|^{-{1}/{2}}\frac{\chi_{R^c}(x)}{\langle 2^j(x-\xx_R)\rangle^{t_3}}\Big)^{p_3}  \bigg)^{{1}/{p_3}}    .
\end{align*}
Then it is easy to prove  \eqref{keylemma12conditions} and \eqref{keylemma12est}.

\subsection{Proof of Lemma \ref{keylemma13}}
Using the fact that $\sum_{j\in\bbz}\wh{\Psi}(2^{-j}\xxxi)=1$ for $\xxxi\not=\000$,
we can write
\begin{equation}\label{tsigmadecomp}
T_{\sigma}(a_k,f_2,f_3)=\sum_{j\in\bbz}T_{\wt{\sigma_j}}(a_k,f_2,f_3)
\end{equation}
where $\wt{\sigma_j}(\xxxi):=\sigma(\xxxi)\wh{\Psi}(2^{-j}\xxxi)$ so that 
$$\sup_{k\in\bbz}\big\Vert \wt{\sigma_k}(2^k\ccdot)\big\Vert_{L^2_s((\bbrn)^3)}=\LL_s^2[\sigma]=1.$$
Moreover, due to the support of $\wt{\sigma_j}$,
\begin{equation}\label{furtherdecomp}
T_{\wt{\sigma_j}}(a_k,f_2,f_3)=T_{\wt{\sigma_j}}(\Ga_{j+1}a_k,f_2,f_3).
\end{equation}
Now the left-hand side of \eqref{keylemma13est} is less than
$$\sup_{l\in\bbz} \Big|\sum_{k=1}^{\infty}\la_k\chi_{(Q_k^{***})^c}(x)\chi_{(B_k^l)^c}(x)\phi_l\ast \Big(\sum_{j\in\bbz}T_{\wt{\sigma_j}}\big( \Ga_{j+1}a_k,f_2,f_3\big)(x) \Big)(x) \Big|.$$

Let $s_1,s_2,s_3$ be numbers such that $s_1>n/p-n/2$, $s_2,s_3>n/2$, and $s=s_1+s_2+s_3$.
For $x\in (Q^{***})^c\cap(B_k^l)^c$ and $|x-y| \le 2^{-l}$
\begin{equation*}
|x - \xx_{Q_k}| \lesssim | y - \xx_{Q_k}|.
\end{equation*}
In the same argument as in the proof of \eqref{keyinest} and \eqref{keyoutest}, with \eqref{estlaak} replaced by \eqref{estgaak}, we can get
\begin{align*}
&\langle 2^j(x-\xx_{Q_k})\rangle^{s_1}\big|  T_{\wt{\sigma_j}}\big(\Ga_{j+1}a_k,f_2,f_3\big)(y) \big|\\
& \lesssim\Vert f_2\Vert_{L^{\infty}(\bbrn)}\Vert f_3\Vert_{L^{\infty}(\bbrn)}  \langle2^j(x-\xx_{Q_k})\rangle^{s_1} \int_{(\bbrn)^3}\big|\wt{\sigma_j}^{\vee}(y-z_1,z_2,z_3)\big| |\Ga_{j+1}a_k(z_1)| \; d\zzz \\
&\lesssim \ell(Q_k)^{-n/p}\min{\big\{1,\big(2^j\ell(Q_k) \big)^M \big\}}I_{k,j,s}(y)
\end{align*}
where $I_{k,j,s}^{in}$ and $I_{k,j,s}^{out}$ are defined as in \eqref{ikjsindef} and \eqref{ikjsout}, respectively, and
$$I_{k,j,s}(y):=I_{k,j,s}^{in}(y)+I_{k,j,s}^{out}(y).$$ 
This yields that
\begin{align*}
&\Big| \phi_l\ast \Big(\sum_{j\in\bbz}T_{\wt{\sigma_j}}\big(\Ga_{j+1}a_k,f_2,f_3\big) \Big)(x)\Big|\\
&\lesssim \ell(Q_k)^{-n/p}|x-\xx_{Q_k}|^{-s_1}\mathcal{M}\Big( \sum_{j\in\bbz}2^{-s_1 j}\min{\big\{1,\big( 2^j\ell(Q_k)\big)^M\big\}}I_{k,j,s}(\cdot) \Big)(x)
\end{align*}
and thus \eqref{keylemma13est} follows from choosing $u_2(x)=u_3(x):=1$ and
\begin{align*}
u_1(x)&:=\sum_{k=1}^{\infty}|\la_k|\chi_{(Q_k^{***})^c}(x)\ell(Q_k)^{-n/p}|x-\xx_{Q_k}|^{-s_1}\\
&\qq\qq\times \mathcal{M}\Big( \sum_{j\in\bbz}2^{-s_1 j}\min{\big\{1,\big( 2^j\ell(Q_k)\big)^M\big\}}I_{k,j,s}(\cdot) \Big)(x).
\end{align*}
Now it is straightforward that $\Vert u_1\Vert_{L^p(\bbrn)}$ is less than
\begin{equation*}
\bigg(\sum_{k=1}^{\infty}|\la_k|^p\ell(Q_k)^{-n}\bigg\Vert  |\cdot-\xx_{Q_k}|^{-s_1}  \mathcal{M}\Big( \sum_{j\in\bbz}2^{-s_1 j}\min{\big\{1,\big( 2^j\ell(Q_k)\big)^M\big\}}I_{k,j,s}(\cdot) \Big)     \bigg\Vert_{L^p((Q_k^{***})^c)}^p \bigg)^{1/p}
\end{equation*}
and the $L^p$-norm in the preceding expression is less than
\begin{align*}
&\big\Vert  |\cdot-\xx_{Q_k}|^{-{s_1}}\big\Vert_{L^{p({2}/{p})'}((Q_k^{***})^c)}  \bigg\Vert  \mathcal{M}\Big(\sum_{j\in\bbz}2^{-s_1j}\min{\big\{1,\big(2^j\ell(Q_k)\big)^M\big\}}I_{k,j,s_1}(\cdot)  \Big)  \bigg\Vert_{L^{{2}}(\bbrn)}\\
&\lesssim \ell(Q_k)^{-(s_1-(n/p-n/2))}\sum_{j\in\bbz}2^{-s_1 j}\min\big\{1,\big(2^j\ell(Q_k) \big)^M\big\}\Vert I_{k,j,s}\Vert_{L^2(\bbrn)}\\
&\lesssim  \ell(Q_k)^{-(s_1-(n/p-n/2))}\ell(Q_k)^n\sum_{j\in\bbz}2^{-(s_1-n/2) j} \min\big\{1,\big(2^j\ell(Q_k) \big)^M\big\}\lesssim \ell(Q_k)^{n/p}
\end{align*}
where \eqref{ikjsin} and \eqref{ikjsoutest} are applied in the penultimate inequality for sufficiently large $M$.
This concludes that 
$$\Vert u_1\Vert_{L^p(\bbrn)}\lesssim \Big( \sum_{k=1}^{\infty}|\la_k|^p\Big)^{1/p}\lesssim 1.$$

\subsection{Proof of Lemma \ref{keylemma14}}

Select $0<\epsilon<1$ such that
$$N_p:=[n/p-n]\le n/p-n<[n/p-n]+\epsilon<s-3n/2.$$
Then Lemma \ref{epsilon control} yields that
\begin{align*}
\big| \phi_l \ast T_{\sigma}\big(a_k, f_2, f_3\big)(x) \big| &\lesssim 2^{l(N_p+n+\epsilon)} \int_{\bbrn} |y - \xx_{Q_k}|^{N_p  + \epsilon} \big|T_{\sigma}(a_k, f_2, f_3)(y) \big| \;dy\\
&\lesssim \frac{1}{ |x - \xx_{Q_k}|^{N_p+n+\epsilon} } \int_{\mathbb{R}^n}| y - \xx_{Q_k} |^{N_p+\e} \big|T_{\sigma}(a_k, f_2, f_3)(y) \big| \; dy\\
&\le  \frac{1}{ |x - \xx_{Q_k}|^{N_p+n+\epsilon} } \big( \mathcal{K}_{N_p+\epsilon}^{in}(a_k,f_2,f_3)+ \mathcal{K}_{N_p+\epsilon}^{in}(a_k,f_2,f_3)\big)
\end{align*}
where we applied $2^l\lesssim |x-\xx_{Q_k}|$ for $x\in B_k^l$ in the penultimate inequality and
$$ \mathcal{K}_{N_p+\epsilon}^{\mathrm{in}}(a_k,f_2,f_3):=\int_{Q_k^{**}}| y - \xx_{Q_k} |^{N_p+\e} \big|T_{\sigma}(a_k, f_2, f_3)(y) \big| \; dy,$$
 $$\mathcal{K}_{N_p+\epsilon}^{\mathrm{out}}(a_k,f_2,f_3):=\int_{(Q_k^{**})^c}| y - \xx_{Q_k} |^{N_p+\e} \big|T_{\sigma}(a_k, f_2, f_3)(y) \big| \; dy.$$
Now we claim that
\begin{equation}\label{keylemma14claim}
\mathcal{K}_{N_p+\epsilon}^{\mathrm{in}/\mathrm{out}}(a_k,f_2,f_3)\lesssim \ell(Q_k)^{N_p-n/p+n+\epsilon}.
\end{equation}
Once \eqref{keylemma14claim} holds, 
we obtain
$$\big| \phi_l \ast T_{\sigma}\big(a_k, f_2, f_3\big)(x) \big| \lesssim |Q_k|^{-1/p}\frac{\ell(Q_k)^{N_p+n+\epsilon}}{|x-\xx_{Q_k}|^{N_p+n+\epsilon}},$$
which implies \eqref{keylemma14est} with $u_2(x)=u_3(x):=1$ and
$$u_1(x):=\sum_{k=1}^{\infty}|\la_k||Q_k|^{-1/p}\frac{\ell(Q_k)^{N_p+n+\epsilon}}{|x-\xx_{Q_k}|^{N_p+n+\epsilon}}\chi_{(Q_k^{***})^c}(x).$$
Moreover,
\begin{align*}
\Vert u_1\Vert_{L^p(\bbrn)}&\le \Big( \sum_{k=1}^{\infty}|\la_k|^p|Q_k|^{-1}\ell(Q_k)^{p(N_p+n+\epsilon)}\big\Vert |\cdot-\xx_{Q_k}|^{-(N_p+n+\epsilon)}\big\Vert_{L^p((Q_k^{***})^c)}^p  \Big)^{1/p}\\
&\lesssim \Big(\sum_{k=1}^{\infty}|\la_k|^p \Big)^{1/p}\lesssim 1
\end{align*}
because $N_p+n+\epsilon>n/p$.

Therefore, it remains to show \eqref{keylemma14claim}.
Indeed, it follows from Theorem \ref{thmd} that
\begin{align*}
\mathcal{K}_{N_p+\epsilon}^{\mathrm{in}}(a_k,f_2,f_3)&\lesssim \ell(Q_k)^{N_p+\epsilon}\big\Vert T_{\sigma}(a_k,f_2,f_3)\big\Vert_{L^1(\bbrn)}\\
&\lesssim \ell(Q_k)^{N_p+\epsilon}\Vert a_k\Vert_{L^1(\bbrn)}\lesssim \ell(Q_k)^{N_p-n/p+n+\epsilon}.
\end{align*}

For the other term, we use both \eqref{tsigmadecomp} and \eqref{furtherdecomp} to write
$$\mathcal{K}_{N_p+\epsilon}^{\mathrm{out}}(a_k,f_2,f_3)\lesssim  \sum_{j\in\bbz}2^{-j(N_p+\epsilon)}\int_{(Q_k^{**})^c}{ \langle 2^j(y-\xx_{Q_k})\rangle^{N_p+\epsilon}\big| T_{\wt{\sigma_j}}\big(\Ga_{j+1}a_k,f_2,f_3\big)(y)\big|} \; dy.$$
Let $s_1,s_2,s_3$ be numbers satisfying
$$N_p+n/2+\epsilon<s_1<s-n,\q s_2,s_3>n/2,\q s_1+n<s_1+s_2+s_3=s,$$
similar to \eqref{npcondition1} and \eqref{npcondition2}.
Then, using the argument in \eqref{tsigmajlaest}, we have
\begin{align*}
&\big| T_{\wt{\sigma_j}}\big(\Ga_{j+1}a_k,f_2,f_3 \big)(y)\big|\chi_{(Q_k^{**})^c}(y)\lesssim \ell(Q_k)^{-n/p}\min\big\{1,\big( 2^j\ell(Q_k)\big)^M \big\}\frac{1}{\langle 2^j(y-\xx_{Q_k})\rangle^{s_1}}\\
&\qq\qq\qq\times  2^{-jn}\int_{\bbrn} |A_{j,Q_k}(z_1)|\big\Vert \langle 2^j(y-z_1,z_2,z_3)\rangle^{s}\wt{\sigma_j}^{\vee}(y-z_1,z_2,z_3)\big\Vert_{L^2(z_2,z_3)} dz_1
\end{align*}
where $A_{j,Q_k}$ is defined as in \eqref{ajqkdef}.
This finally yields that
\begin{align*}
&\mathcal{K}_{N_p+\epsilon}^{\mathrm{out}}(a_k,f_2,f_3)\lesssim \ell(Q_k)^{-n/p}\sum_{j\in\bbz}2^{-j(N_p+n+\epsilon)}\min\big\{1,\big( 2^j\ell(Q_k)\big)^M\big\}
\int_{\bbrn}|A_{j,Q_k}(z_1)| \\
&\qq\times \Big(\int_{(Q_k^{**})^c} \frac{1}{\langle 2^j(y-\xx_{Q_k})\rangle^{s_1-(N_p+\epsilon)}}     \big\Vert \langle 2^j(y-z_1,z_2,z_3)\rangle^{s}\wt{\sigma_j}^{\vee}(y-z_1,z_2,z_3)\big\Vert_{L^2(z_2,z_3)}     dy \Big) \; dz_1\\
&\lesssim \ell(Q_k)^{-n/p}\sum_{j\in\bbz}2^{-j(s_1+n)}\min\big\{1,\big( 2^j\ell(Q_k)\big)^M\big\}\Vert A_{j,Q_k}\Vert_{L^1(\bbrn)}\\
&\qq\qq\qq\qq\times \big\Vert \langle 2^j\ccdot \rangle^s\wt{\sigma_j}^{\vee}\big\Vert_{L^2((\bbrn)^3)} \big\Vert |\cdot-\xx_{Q_k}|^{-(s_1-(N_p+\epsilon))}\big\Vert_{L^2(\bbrn)}\\
&\lesssim \ell(Q_k)^{-n/p+n}\ell(Q_k)^{-(s_1-(N_p+n/2+\epsilon))}\sum_{j\in\bbz}2^{-j(s_1-n/2)}\min \big\{1,\big( 2^j\ell(Q_k)\big)^{M-(L_0-s_1-n)} \big\}\\
&\lesssim \ell(Q_k)^{N_p-n/p+n+\epsilon}
\end{align*}
for $M$ and $L_0$ satisfying $M>L_0-s_1-n$, which completes the proof of \eqref{keylemma14claim}.

\appendix

\section{Bilinear Fourier multipliers $(m=2)$}

We remark that Theorem \ref{main} still holds in the bilinear setting where all the arguments above work as well.
\begin{theorem}\label{bilinearthm}
Let $0<p_1,p_2\le \infty$ and $0<p\le 1$ with $1/p=1/p_1+1/p_2$.
Suppose that
$$s>n\q \text{ and }\q \frac{1}{p}-\frac{1}{2}<\frac{s}{n}+\sum_{j\in J}\Big(\frac{1}{p_j}-\frac{1}{2} \Big)$$
where $J$ is an arbitrary subset of $\{1,2\}$. Let $\sigma$ be a function on $(\bbrn)^2$ satisfying 
$$\sup_k\big\Vert \sigma(2^k\ccdot)\wh{\Psi^{(2)}}\big\Vert_{L^2_s((\bbrn)^2)}<\infty$$
and the bilinear analogue of the vanishing moment condition \eqref{vanishingmoment}. Then the bilinear Fourier multiplier $T_{\sigma}$, associated with $\sigma$, satisfies
$$\big\Vert T_{\sigma}(f_1,f_2)\big\Vert_{H^p(\bbrn)}\lesssim \sup_k\big\Vert \sigma(2^k\ccdot)\wh{\Psi^{(2)}}\big\Vert_{L^2_s((\bbrn)^2)}\Vert f_1\Vert_{H^{p_1}(\bbrn)}\Vert f_2\Vert_{H^{p_2}(\bbrn)}$$
 for $f_1,f_2\in\mathscr{S}_0(\bbrn)$.
\end{theorem}
The proof is similar, but much simpler than that of Theorem \ref{main}. 
Moreover, unlike Theorem \ref{main}, Theorem \ref{bilinearthm} covers the results for $p_{j}=\infty$, $j=1,2$, which follow immediately from the bilinear analogue of Proposition \ref{mainproposition2}.

\section{General $m$-linear Fourier multipliers for $m\ge 4$}

The structure of the proof of Theorem \ref{main} is actually very similar to those of Theorems \ref{thmc} and \ref{thmd}, in which $T_{\sigma}(f_1,\dots,f_m)$ is written as a finite sum of $T^{\kappa}(f_1,\dots,f_m)$ for some variant operators $T^{\kappa}$, and then 
\begin{equation}\label{lpkeyest}
\big|T^{\kappa}\big(f_1,\dots,f_m\big)(x)\big|\lesssim \sup_{k\in\bbz}\big\Vert \sigma(2^k\ccdot)\wh{\Psi^{(m)}}\big\Vert_{L^2_s((\bbrn)^m)}u_1(x)\cdots u_m(x)\end{equation}
where $\Vert u_{j}\Vert_{L^{p_{j}}(\bbrn)}\lesssim \Vert f_{j}\Vert_{L^{p_{j}}(\bbrn)}$ for $1\le j\le m$.
Compared to the $H^{p_1}\times\cdots\times H^{p_m}\to L^p$ estimates in Theorems \ref{thmc} and \ref{thmd}, 
one of the obstacles to be overcome for the boundedness into Hardy space $H^p$ is to replace the left-hand side of \eqref{lpkeyest} by
$$\sup_{l\in\bbz}\big| \phi_l\ast T^{\kappa}\big(f_1,\dots,f_m\big)(x)\big|,$$
and we have successfully accomplished this for $m=3$ as mentioned in \eqref{suplinstance}.
One of the methods we have adopted is 
$$\chi_{Q_k^{***}}(x)2^{ln}\int_{|x-y|\le 2^{-l}} F_1(y)F_2(y)F_3(y)\; dy\lesssim \chi_{Q_k^{***}}(x)\mathcal{M}_qF_1(x)\mathcal{M}_{\wt{r}}F_2(x)\mathcal{M}_{\wt{r}}F_3(x)$$
where $2<\wt{r}<p_2,p_3$ and $1/q+2/\wt{r}=1$. Then we have
$$\Vert \mathcal{M}_{\wt{r}}F_2\Vert_{L^{p_2}(\bbrn)}\lesssim \Vert F_j\Vert_{L^{p_j}(\bbrn)},\qq j=2,3$$
by the $L^{p_j}$ boundedness of $\mathcal{M}_{\wt{r}}$ with $\wt{r}<p_j$.
Such an argument is contained in the proof of Lemma \ref{keylemma1}.
However, if we consider $m$-linear operators for $m\ge 4$, then the above argument does not work for $p_2,\dots,p_m>2$.
For example, it is easy to see that $1/q + 3/\wt{r}$ exceeds $1$ if $\wt{r}>2$ is sufficiently close to $2$.
That is, we are not able to obtain $m$-linear estimates for $0<p_1\le 1$ and $2< p_2, \cdots, p_m<\infty$, $m\ge 4$. This is critical because our approach in this paper highly relies on interpolation between the estimates in the regions $\RR_1, \RR_2, \RR_3$, which are trilinear versions of $\{(1/p_1, \cdots, 1/p_m) : 0<p_1\le 1, \, 2< p_2, \cdots, p_m<\infty\}$.

\end{document}